\documentclass[a4paper, 10pt, notitlepage]{article}

\usepackage{amsthm, amsmath, amssymb, latexsym}
\usepackage{mathrsfs}

\theoremstyle{plain}
\newtheorem{theorem}{Theorem}
\newtheorem{lemma}{Lemma}
\newtheorem{corollary}{Corollary}
\newtheorem{proposition}{Proposition}

\theoremstyle{definition}
\newtheorem{definition}{Definition}
\newtheorem{example}{Example}

\theoremstyle{remark}
\newtheorem{remark}{Remark}

\DeclareMathOperator{\cl}{cl}
\DeclareMathOperator{\cone}{cone}
\DeclareMathOperator{\sign}{sign}
\DeclareMathOperator*{\argmin}{arg\,min}
\DeclareMathOperator{\linhull}{span}
\DeclareMathOperator{\interior}{int}
\DeclareMathOperator{\dist}{dist}
\DeclareMathOperator{\dom}{dom}
\DeclareMathOperator{\diag}{diag}
\DeclareMathOperator{\lineal}{lin}
\DeclareMathOperator{\rank}{rank}
\DeclareMathOperator{\trace}{Tr}

\author{M.V. Dolgopolik}
\title{Augmented Lagrangian Functions for Cone Constrained Optimization: the Existence of Global Saddle Points and Exact
Penalty Property}

\begin{document}

\maketitle

\begin{abstract}
In this article we present a general theory of augmented Lagrangian functions for cone constrained optimization
problems that allows one to study almost all known augmented Lagrangians for these problems within a unified framework.
We develop a new general method for proving the existence of global saddle points of augmented Lagrangian functions,
called the localization principle. The localization principle unifies, generalizes and sharpens most of the known
results on the existence of global saddle points, and, in essence, reduces the problem of the existence of global saddle
points to a local analysis of optimality conditions. With the use of the localization principle we obtain first
necessary and sufficient conditions for the existence of a global saddle point of an augmented Lagrangian for cone
constrained minimax problems via both second and first order optimality conditions. In the second part of the paper, we
present a general approach to the construction of globally exact augmented Lagrangian functions. The general approach
developed in this paper allowed us not only to sharpen most of the existing results on globally exact augmented
Lagrangians, but also to construct first globally exact augmented Lagrangian functions for equality constrained
optimization problems, for nonlinear second order cone programs and for nonlinear semidefinite programs. These globally
exact augmented Lagrangians can be utilized in order to design new superlinearly (or even quadratically) convergent
optimization methods for cone constrained optimization problems.
\end{abstract}

\section{Introduction}

The main goal of this article is to present a general theory of augmented Lagrangian functions for cone constrained
optimization problems that provides a theoretical foundation for the development of augmented Lagrangian methods for
these problems. In recent years, several attempts were made to develop a general theory of Lagrangian functions. 
A general duality theory for nonlinear Lagrangian functions for mathematical programming problems was developed in
\cite{RubinovYang,WangYangYang2007}. In \cite{Giannessi,LiFengZhang2013,ZhuLi2014,ZhuLi2014_2}, the image space analysis
was applied to the study of duality theory for augmented Lagrangian functions, while in \cite{BurachikIusemMelo,
WangYangYang,WangLiuQu} some general classes of augmented Lagrangian functions constructed from the Rockafellar-Wets
augmented Lagrangian (\cite{RockafellarWets}, Section~11.K) were studied. A unified theory of augmented Lagrangian
methods for mathematical programming problems was presented in \cite{WangLi2009}. However, there is no satisfactory and
general enough theory of the existence of global saddle points that can be applied to various augmented Lagrangians for
various cone constrained optimization problems. Furthermore, there are no results on \textit{exact} augmented Lagrangian
functions for cone constrained optimization problems. Our aim is to fill these gaps and develop a general theory
containing simple and easily verifiable necessary and sufficient conditions for the existence of global saddle points of
augmented Lagrangian functions, and for the global exactness of these function.

In this paper, instead of utilizing a modification of the Rockafellar-Wets augmented Lagrangian, we consider a more
straightforward approach to the theory of augmented Lagrangian functions (that is very similar to the one used in the
image space analysis), in which an augmented Lagrangian is defined simply as the sum of the objective function and a
convolution function depending on constraints, multipliers and penalty parameter, and satisfying some general
assumptions (axioms). The main advantage of this approach is the fact that it does not rely on the particular structure
of the augmented Lagrangian, and allows one to include almost all particular cases into the general theory.

The axiomatic approach that we use was inspired by \cite{LiuYang}. It provides one with a simple and unified framework
for the study of various augmented Lagrangian functions, such as the Hestenes-Powell-Rockafellar augmented Lagrangian
\cite{BirginMartinez,Rockafellar1974,Rockafellar1993}, the cubic augmented Lagrangian \cite{Kiwiel},
Mangasarian's augmented Lagrangian \cite{Mangasarian,WuLuo2012b}, the exponential penalty function
\cite{Bertsekas,TsengBertsekas,SunLi,LiuYang,WangLi2009}, the Log-Sigmoid Lagrangian \cite{Polyak2001,Polyak2002}, the
penalized exponential-type augmented Lagrangians \cite{Bertsekas,SunLi,LiuYang,WangLi2009}, the modified barrier
functions \cite{Polyak,SunLi,LiuYang,WangLi2009}, the p-th power augmented Lagrangian \cite{Li1995,Li1997,Xu,LiSun2001a,
LiSun2001b,WuLuo2012a,LiuYang}, He-Wu-Meng's augmented Lagrangian \cite{HeWuMeng}, extensions of 
the Hestenes-Powell-Rockafellar augmented Lagrangian to the case of nonlinear second order cone programs
\cite{LiuZhang2007,LiuZhang2008,ZhouChen2015}, nonlinear semidefinite programs
\cite{HuangYangTeo_Chapter,SunZhangWu2006,SunSunZhang2008,ZhaoSunToh2010,Sun2011,WenGoldfarbYin2010,LuoWuChen2012,
WuLuoDingChen2013,WuLuoYang2014,YamashitaYabe2015}, and semi-infinite programs
\cite{RuckmannShapiro,HuyKim2012,SonKimTam2012,BurachikYangZhou2017}, as well as extensions of the exponential penalty
function and the modified barrier functions to the case of nonlinear second order cone programs
\cite{ZhangGuXiao2011}, and nonlinear semidefinite programs \cite{Stingl2006,Noll2007,LiZhang2009,ZhangLiWu2014,
YamashitaYabe2015,LuoWuLiu2015}.

The first part of the paper is devoted to the problem of the existence of global saddle points of augmented Lagrangian
functions, which is important for convergence analysis of augmented Lagrangian methods, since the existence of a global
saddle point is, usually, necessary for the global convergence of these methods. The problem of the existence of global
saddle points was studied for general cone constrained optimization problems in \cite{ShapiroSun,ZhouZhouYang2014}, for
mathematical programming problems in \cite{LiuYang,LiuTangYang2009,LuoMastroeniWu2010,ZhouXiuWang,WuLuo2012b,
WangZhouXu2009,SunLi,WangLiuQu}, for nonlinear second order cone programming problems in \cite{ZhouChen2015}, for
nonlinear semidefinite programming problems in \cite{WuLuoYang2014,LuoWuLiu2015}, and for semi-infinite programming
problems in \cite{RuckmannShapiro,BurachikYangZhou2017}. The analysis of the known results on the existence of global
saddle points of augmented Lagrangian functions indicates that the same, in essence, results are proved and reproved
multiple times in different settings.

In this article, we propose a unified approach to the study of global saddle points of augmented Lagrangian functions,
called \textit{the localization principle}. The localization principle was first developed by the author for the study
of exact linear penalty functions \cite{Dolgopolik,DolgopolikII} and augmented Lagrange multipliers
\cite{Dolgopolik_AugmLagrMult}. A modification of this principle presented in this paper provides first simple
necessary and sufficient conditions for the existence of global saddle points that unify, generalize and sharpen almost
all known results in this area (see Remark~\ref{Remark_LocPrinciple_Unification}). Furthermore, the localization
principle reduces the study of global saddle points to a local analysis of sufficient optimality conditions. With the
use of this principle we obtained simple necessary and sufficient conditions for the existence of global saddle points
for cone constrained \textit{minimax} problems via both second order and first order (the so-called \textit{alternance}
conditions; see \cite{MalozemovPevnyi,DaugavetMalozemov75,Daugavet,DaugavetMalozemov81,DemyanovMalozemov_Alternance,
DemyanovMalozemov_Collect}) sufficient optimality conditions. To the best of authors knowledge, the problem of the
existence of saddle points for cone constrained minimax problems has never been studied before. We also provide simple
sufficient conditions for the existence of local saddle points that unify and sharpen many existing result
(see~Remark~\ref{Rmrk_MangasarianAL}).

As it is well known, standard augmented Lagrangian methods converge only linearly (see, e.g., \cite{BirginMartinez,
Stingl2006,Noll2007,LiuZhang2008,SunSunZhang2008,ZhangGuXiao2011,ZhangLiWu2014} and references therein), and in order to
apply them one has to minimize an augmented Lagrangian function numerous times. In order to overcome this difficulty, 
Di Pillo and Grippo \cite{DiPilloGrippo1979} proposed to consider the so-called \textit{exact} augmented Lagrangian
functions. Exact augmented Lagrangian functions are constructed in such a way that it is necessary to minimize them only
once (but \textit{simultaneously} in primal and dual variables) in order to recover KKT-points corresponding to globally
optimal solutions of the initial optimization problem. Furthermore, one can design superlinearly and even quadratically
convergent methods for minimizing exact augmented Lagrangian functions. Thus, the use of the exact augmented Lagrangians
allows one to overcome main disadvantages of standard augmented Lagrangian methods. Exact augmented Lagrangian functions
and numerical methods based on the use of these functions were studied in \cite{DiPilloGrippo1980,DiPilloGrippo1982,
Lucidi1988,DiPilloLucidi1996, DiPilloLucidi2001,DiPilloEtAl2002,DiPilloLiuzzi2003,DuZhangGao2006,DuLiangZhang2006,
LuoWuLiu2013,DiPilloLucidiPalagi1993,DiPilloGrippoLucidi1993,DiPilloLucidiPalagi2000,DiPilloLucidi2005,DiPilloLiuzzi,
DiPilloLiuzzi2011}). However, it should be noted that all existing exact augmented Lagrangian functions were constructed
from the Hestenes-Powell-Rockafellar augmented Lagrangian, and were only considered for mathematical programming
problems, except for the exact augmented Lagrangian function for nonlinear semidefinite programming problem from a
recent paper \cite{FukudaLourenco}.

In the second part of the paper, we develop a general theory of globally exact augmented Lagrangian functions. We
present a simple method for constructing exact augmented Lagrangian functions, and provide first simple necessary and
sufficient conditions for the global exactness of these functions in the form of the localization principle. 
We also demonstrate that globally exact augmented Lagrangians for mathematical programming problems can be
constructed not only from the Hestenes-Powell-Rockafellar augmented Lagrangian, but also from many other augmented
Lagrangian functions. Moreover, we, for the first time, propose globally exact augmented Lagrangian functions for
equality constrained problems and for nonlinear second order cone programs. We also propose new globally exact augmented
Lagrangian functions for nonlinear semidefinite programming problems. These exact augmented Lagrangian functions can be
utilized in order to design new efficient superlinearly (or even quadratically) convergent optimization methods.

The paper is organized as follows. In Section~\ref{Section_DefinitionAndExamples}, we introduce a general augmented
Lagrangian function for a cone constrained optimization problem, and present main assumptions on this function that are
utilized throughout the article. In Section~\ref{Section_Examples}, we provide many particular examples of augmented
Lagrangian functions for general cone constrained, mathematical programming, nonlinear second order cone programming,
nonlinear semidefinite programming and semi-infinite programming problems, and point out that all these augmetned
Lagrangian functions satisfy the main assumptions of this article. Some general results on the existence of global
sadde points are presented in Section~\ref{Section_SaddlePoints}, while in
Section~\ref{Section_LocalizationPrinciple_GSP} we study the localization principle. Applications of the localization
principle to cone constrained minimax problems are given in Section~\ref{Section_ApplLocPrinciple_GSP}. The general
theory of globally exact augmented Lagrangian functions is developed in Section~\ref{Section_ExactAL}, while some
applications of this theory as well as particular examples of globally exact augmented Lagrangians are presented in
Section~\ref{Section_ApplLocPrinciple_ExactAL}.

\section{An Augmented Lagrangian Function for Cone Constrained Optimization Problems}
\label{Section_DefinitionAndExamples}

Let $X$ be a finite dimensional normed space, and $A \subset X$ be a nonempty set. Let also $Y$ be a normed space, and
$K \subset Y$ be a nonempty closed convex cone. As usual, denote by $Y^*$ the topological dual of $Y$, and by 
$\langle \cdot, \cdot \rangle$ the standard coupling function between $Y$ and its dual or the inner product in
$\mathbb{R}^s$, $s \in \mathbb{N}$, depending on the context.

Throughout this article, we study the following cone constrained optimization problem
$$
  \min f(x) \quad \text{subject to} \quad G(x) \in K, \quad x \in A,		\eqno{(\mathcal{P})}
$$
where $f \colon X \to \mathbb{R} \cup \{ + \infty \}$ and $G \colon X \to Y$ are given functions. Hereinafter, we
suppose that there exists a feasible point $x$ of $(\mathcal{P})$ such that $f(x) < + \infty$, and there exists a
globally optimal solution of the problem $(\mathcal{P})$.

Let us introduce an \textit{augmented Lagrangian function} for the problem $(\mathcal{P})$. Choose
a function $\Phi \colon Y \times Y^* \times (0, + \infty) \to \mathbb{R} \cup \{ + \infty \}$, and define
$$
  \mathscr{L}(x, \lambda, c) = f(x) + \Phi(G(x), \lambda, c)
$$
where $\lambda \in Y^*$ is a Lagrange multiplier, and $c > 0$ is a penalty parameter. 

\begin{remark}
Note that only the constraint $G(x) \in K$ is incorporated into the augmented Lagrangian function, while the constraint
$x \in A$ has to be taken into account directly. This approach allows one to choose what constraints of an
optimization problem under consideration are handled via the augmented Lagrangian function, and what constraints are
handled by other methods.
\end{remark}

Our main goal is to obtain simple necessary and sufficient conditions for the existence of global saddle points of 
the augmented Lagrangian function $\mathscr{L}(x, \lambda, c)$ that can be easily applied to various cone constrained
optimization problems, and various augmented Lagrangians for these problems.

Before we proceed to the study of saddle points, let us list main assumptions on the function $\Phi$ that we
utilise throughout the article. In order to include several particular cases into the general theory, we formulate our 
assumptions (as well as all definitions and results below) with respect to a given closed convex 
cone $\Lambda \subseteq Y^*$. In particular, one can define $\Lambda = Y^*$ or $\Lambda = K^*$, 
where $K^* = \{ y^* \in Y^* \mid \langle y^*, y \rangle \le 0 \text{~for all~} y \in K \}$ is the polar cone of $K$.
Let us point out that a proper choice of the cone $\Lambda$ is necessary to ensure that the augmented Lagrangian
function $\mathscr{L}(x, \lambda, c)$ has desirable properties. Some remarks on how to choose the cone $\Lambda$ are
given throughout the text.

For any $y \in Y$ and $r > 0$ denote $B(y, r) = \{ z \in Y \mid \| z - y \| \le r \}$ and 
$\dist(y, K) = \inf_{z \in K} \| y - z \|$. In the subsequent sections, the following assumptions on the function $\Phi$
are utilised:
\begin{itemize}
\item[(A1)]{$\forall y \in K$ $\forall c > 0$ one has $\Phi(y, 0, c) \ge 0$;
}

\item[(A2)]{$\forall y \in K$ $\forall \lambda \in \Lambda$ $\forall c > 0$ one has $\Phi(y, \lambda, c) \le 0$;
}

\item[(A3)]{$\forall y \notin K$ $\forall c > 0$ $\exists \lambda \in \Lambda$ such 
that $\lim_{t \to + \infty} \Phi(y, t \lambda, c) = + \infty$;
}

\item[(A4)]{$\forall y \in Y$ $\forall \lambda \in \Lambda$ the function $\Phi(y, \lambda, c)$ is non-decreasing in $c$;
}

\item[(A5)]{$\forall \lambda \in \Lambda$ $\forall c > 0$ the function $\Phi(\cdot, \lambda, c)$ is monotone with
respect to the order generated by $K$, i.e. $\Phi(y_1, \lambda, c) \le \Phi(y_2, \lambda, c)$ if $y_1 - y_2 \in K$;
}

\item[(A6)]{$\forall y \notin K$ $\forall \lambda \in \Lambda$ $\forall c_0 > 0$ $\exists r > 0$ such that 
$$
  \lim_{c \to +\infty} 
  \inf\{ \Phi(z, \lambda, c) - \Phi(z, \lambda, c_0) \mid z \in B(y, r) \colon \Phi(z, \lambda, c_0) < + \infty \} 
  = + \infty;
$$
\vspace{-0.7cm}
}

\item[(A7)]{$\forall \lambda \in \Lambda$ $\forall c_0 > 0$ $\forall r > 0$ one has
\begin{multline*}
  \lim_{c \to +\infty} 
  \inf\Big\{ \Phi(y, \lambda, c) - \Phi(y, \lambda, c_0) \Bigm| \\
  y \in Y \colon \dist(y, K) \ge r, \: \Phi(y, \lambda, c_0) < + \infty \Big\} 
  = + \infty.
\end{multline*}
\vspace{-0.8cm}
}

\item[(A8)]{$\forall y \in K$ $\forall \lambda \in \Lambda \setminus K^*$ $\forall c > 0$ one has 
$\Phi(y, \lambda, c) < 0$;
}

\item[(A9)]{$\forall y \in K$ $\forall \lambda \in \Lambda$ $\forall c > 0$ such that $\langle \lambda, y \rangle \ne 0$
one has $\Phi(y, \lambda, c) < 0$;
}

\item[(A10)]{$\forall y \in K$ $\forall \lambda \in K^*$ $\forall c > 0$ such that $\langle \lambda, y \rangle = 0$ one
has $\Phi(y, \lambda, c) = 0$;
}

\item[(A11)]{$\forall y \in K$ $\forall \lambda \in K^*$ $\forall c > 0$ such that $\langle \lambda, y \rangle = 0$ the
function $\Phi(\cdot, \lambda, c)$ is Fr\'echet differentiable at $y$, and $D_y \Phi(y, \lambda, c) = \Phi_0(\lambda)$,
where $D_y$ stands for the Fr\'echet derivative in $y$, and $\Phi_0 \colon K^* \to K^*$ is a surjective mapping that
does not depend on $y$ and $c$, and such that $\langle \Phi_0(\lambda), y \rangle = 0$ iff 
$\langle \lambda, y \rangle = 0$;
}

\item[(A12)]{$\forall y \in K$ $\forall \lambda \in \Lambda$ one has $\Phi(y, \lambda, c) \to 0$ as $c \to \infty$.
}
\end{itemize}

\begin{remark}
Since $K$ is closed, $(A7)$ implies $(A6)$. Furthermore, note that the bigger is the cone $\Lambda$, the more
restrictive are the above assumptions (apart from assumption $(A3)$).
\end{remark}

\section{Examples of Augmented Lagrangian Functions}
\label{Section_Examples}

Below, we provide many particular examples of the augmented Lagrangian function $\mathscr{L}(x, \lambda, c)$ for
various cone constrained optimization problems, and point out whether assumptions above are satisfied in these examples.

\subsection{General Cone Constrained Problems}

We start with an augmented Lagrangian function introduced by Rockafellar and Wets (\cite{RockafellarWets},
Section~11.K), which is the only augmented Lagrangian function for the general cone constrained optimization problem
known to the author (apart from its direct generalizations, as in \cite{BurachikIusemMelo,WangYangYang,WangLiuQu}). For
more details on this augmented Lagrangian, see \cite{HuangYang2003,ShapiroSun,HuangYang2005,Dolgopolik_AugmLagrMult} and
references therein.

\begin{example} \label{Example_RockafellarWetsAL}
Let a function $\sigma \colon Y \to [0, + \infty]$ be such that $\sigma(0) = 0$ and $\sigma(y) > 0$ for any $y \ne 0$.
Define
$$
  \Phi(y, \lambda, c) = \inf_{p \in K - y} \big( - \langle \lambda, p \rangle + c \sigma(p) \big).
$$
Here we suppose that the function $\sigma$ is such that $\Phi(y, \lambda, c) > - \infty$ for all $y \in Y$, 
$\lambda \in Y^*$ and $c > 0$. In particular, one can set $\sigma(y) = \| y \|^2 / 2$.

Let $\Lambda = Y^*$. Then assumptions $(A1)$--$(A5)$ and $(A10)$ are satisfied in the general case. Assumptions $(A6)$
and $(A7)$ hold true, if the function $\sigma$ has \textit{a valley at zero} (i.e. for any neighbourhood $U$ of zero
there exists $\delta > 0$ such that $\sigma(y) \ge \delta$ for all $y \in Y \setminus U$). Assumptions $(A8)$ and $(A9)$
are valid, if $\sigma(t y) / t \to 0$ as $t \to 0$ for any $y \in Y$. Assumption $(A11)$ is satisfied with 
$\Phi_0(\lambda) \equiv \lambda$, in particular, if $Y$ is a Hilbert space and $\sigma(y) = \| y \|^2 / 2$
(see~\cite{ShapiroSun}). Finally, assumption $(A12)$ is satisfied, provided $\sigma(y) \ge \omega(\| y \|)$ for some
non-negative continuous function $\omega$ such that $\omega(t) = 0$ iff $t = 0$, and 
$\liminf_{t \to + \infty} \omega(t) / t > 0$.
\end{example}

\subsection{Mathematical Programming}

Consider the following mathematical programming problem:
\begin{equation} \label{MathProg}
  \min f(x) \quad \text{subject to} 
  \quad g_i(x) \le 0, \quad i \in I, \quad g_j(x) = 0, \quad j \in J, \quad x \in A,
\end{equation}
where $g_i \colon X \to \mathbb{R}$ are given functions, $I = \{ 1, \ldots, l \}$ and $J = \{ l + 1, \ldots, l + s \}$.
Denote $Y = \mathbb{R}^{l + s}$, $G(\cdot) = (g_1(\cdot), \ldots, g_{l + s}(\cdot))$ and
$K = \mathbb{R}_-^l \times \{ 0_s \}$, where $\mathbb{R}_- = (- \infty, 0]$. Then problem \eqref{MathProg} is
equivalent to the problem $(\mathcal{P})$. Note that in this case $K^* = \mathbb{R}_+^l \times \mathbb{R}^s$ and 
$Y^* = \mathbb{R}^{l + s}$, where $\mathbb{R}_+ = [0, + \infty)$.

Below, we only provide examples of \textit{separable} augmented Lagrangian functions for problem~\eqref{MathProg}, i.e.
such augmented Lagrangians that
$$
  \Phi(y, \lambda, c) = \sum_{i = 1}^{l + s} \Phi_i(y_i, \lambda_i, c), \quad
  y = \big( y_1, \ldots, y_{l + s} \big), \quad \lambda = \big( \lambda_1, \ldots, \lambda_{l + s} \big),
$$
where $\Phi_i \colon \mathbb{R}^2 \times (0, + \infty) \to \mathbb{R} \cup \{ + \infty \}$ are some functions. Let
us note that most of (if not all) augmented Lagrangian functions for mathematical programming problems appearing in
applications are indeed separable.

\begin{example} \label{Example_EssentiallyQuadraticAL}
Suppose that $J = \emptyset$, i.e. suppose that there are no equality constraints. Let 
$\phi \colon \mathbb{R} \to \mathbb{R}$ be a twice continuously differentiable and strictly convex function such that
$\phi(0) = 0$, $\phi'(0) = 0$ and the derivative $\phi'(\cdot)$ is surjective. For any $s, t \in \mathbb{R}$ define
$$
  P(t, \lambda) = \begin{cases}
    \lambda t + \phi(t), & \text{if } \lambda + \phi'(t) \ge 0, \\
    \min_{\tau \in \mathbb{R}} [ \lambda \tau + \phi(\tau) ], & \text{otherwise},
  \end{cases}
$$
and set
$$
  \Phi_i(y_i, \lambda_i, c) = \frac{1}{c} P(c y_i, \lambda_i) \quad \forall i \in I
$$
(see~\cite{Bertsekas}, Section~5.1.2, Example~1). In this case, the function $\mathscr{L}(x, \lambda, c)$ is called the
\textit{essentially quadratic} augmented Lagrangian function for problem \eqref{MathProg}
\cite{SunLi,LiuYang,WangLi2009}. If $\phi(t) = t^2 / 2$, then the essentially quadratic augmented Lagrangian function
coincides with the well-known Hestenes-Powell-Rockafellar Lagrangian function
\cite{Rockafellar1974,Rockafellar1993,BirginMartinez},
which is a particular case of the augmented Lagrangian from Example~\ref{Example_RockafellarWetsAL}.

Let $\Lambda = Y^* = \mathbb{R}^l$. Then one can verify that all assumptions $(A1)$--$(A12)$ are satisfied, and
$\Phi_0(\lambda) \equiv \lambda$.
\end{example}

\begin{example} \label{Example_CubilAL}
Suppose that $J = \emptyset$. Define
$$
  \Phi_i(y_i, \lambda_i, c) = 
  \frac{1}{3c} \Big[ \max\big\{ \sign(\lambda_i) \sqrt{|\lambda_i|} + c y_i, 0 \big\}^3 - |\lambda_i|^{3/2} \Big] 
  \quad \forall i \in I.
$$
Then $\mathscr{L}(x, \lambda, c)$ coincides with \textit{the cubic augmented Lagrangian} \cite{Kiwiel}.

Let $\Lambda = Y^* = \mathbb{R}^l$. Then it is not difficult to check that all assumptions $(A1)$--$(A12)$ are valid,
and $\Phi_0(\lambda) \equiv \lambda$.
\end{example}

\begin{example} \label{Example_Mangasarian}
Let $\phi \colon \mathbb{R} \to \mathbb{R}$ be a twice differentiable and strictly convex function such that 
$\phi(0) = \phi'(0) = 0$, and $\phi'(\cdot)$ is surjective. Define
\begin{gather*}
  \Phi_i(y_i, \lambda_i, c) = 
  \frac{1}{c} \Big[ \phi\Big( \max\{c y_i + \lambda_i, 0 \} \Big) - \phi(\lambda_i) \Big] \quad \forall i \in I, \\
  \Phi_j(y_j, \lambda_j, c) = 
  \frac{1}{c} \big[ \phi(c y_j + \lambda_j) - \phi(\lambda_j) \big] \quad \forall j \in J.
\end{gather*}
Then $\mathscr{L}(x, \lambda, c)$ coincides with the augmented Lagrangian function proposed by Mangasarian
\cite{Mangasarian} (see also~\cite{WuLuo2012b}).

Let $\Lambda = Y^* = \mathbb{R}^{l + s}$. Then one can check that all assumptions $(A1)$--$(A12)$ are
satisfied, and $\Phi_0(\lambda) \equiv \phi'(\lambda) = (\phi'(\lambda_1), \ldots, \phi'(\lambda_{l + s}))$.
\end{example}

\begin{example} \label{Example_ExpPenFunc}
Suppose that $J = \emptyset$. Let $\phi \colon \mathbb{R} \to \mathbb{R}$ be a twice differentiable and strictly
increasing function such that $\phi(0) = 0$. Define
$$
  \Phi_i(y_i, \lambda_i, c) = \frac{\lambda_i}{c} \phi(c y_i) \quad \forall i \in I.
$$
If $\phi(t) = e^t - 1$, then the augmented Lagrangian $\mathscr{L}(x, \lambda, c)$ coincides with 
\textit{the exponential penalty function} \cite{Bertsekas,TsengBertsekas,SunLi,LiuYang,WangLi2009}. In the case
$\phi(t) = 2(\ln(1 + e^t) - \ln 2)$ it coincides with \textit{the Log-Sigmoid Lagrangian} \cite{Polyak2001,Polyak2002}

Let $\Lambda = K^* = \mathbb{R}_+^l$. Then one can check that assumptions $(A1)$--$(A3)$, $(A5)$ and $(A8)$--$(A10)$ are
satisfied. Assumption $(A4)$ is valid, provided $\phi$ is convex. Assumption $(A11)$ is satisfied iff $\phi'(0) \ne 0$,
and in this case $\Phi_0(\lambda) \equiv \phi'(0) \lambda$. Assumption $(A12)$ is satisfied iff $\phi(t) / t \to 0$ as
$t \to - \infty$, while assumptions $(A6)$ and $(A7)$ are not valid, regardless of the choice of $\Lambda$ 
(set $\lambda = 0$). Note also that assumptions $(A2)$, $(A4)$ (if $\phi$ is convex), $(A5)$, $(A8)$ and $(A9)$ are
satisfied iff $\Lambda \subseteq K^*$.
\end{example}

\begin{example} \label{Example_PenalizedExpPenFunc}
Suppose that $J = \emptyset$. Let the function $\phi$ be as in the previous example, and let 
$\xi \colon \mathbb{R} \to \mathbb{R}_+$ be a twice continuously differentiable non-decreasing function such that 
$\xi(t) = 0$ for all $t \le 0$, and $\xi(t) > 0$ for all $t > 0$ (in particular, one can choose 
$\xi(t) = \max\{ 0, t \}^3$). Define
$$
  \Phi_i(y_i, \lambda_i, c) = \frac{\lambda_i}{c} \phi(c y_i) + \frac{1}{c} \xi(c y_i)  \quad \forall i \in I
$$
(see~\cite{Bertsekas}, Section~5.1.2, Example~2). In this case, the augmented Lagrangian function
$\mathscr{L}(x, \lambda, c)$ is called \textit{the penalized exponential-type augmented Lagrangian function}
\cite{SunLi,LiuYang,WangLi2009}.

Let $\Lambda = K^* = \mathbb{R}_+^l$. Then, as in the previous example, assumptions $(A1)$--$(A3)$, $(A5)$ and
$(A8)$--$(A10)$ are satisfied in the general case, assumption $(A4)$ is valid, provided both $\phi$ and $\xi$ are
convex, assumption $(A11)$ is satisfied if and only if $\phi'(0) \ne 0$ (in this case 
$\Phi_0(\lambda) \equiv \phi'(0) \lambda$), and assumption $(A12)$ is satisfied iff $\phi(t) / t \to 0$ as 
$t \to - \infty$. In contrast to the previous example, assumptions $(A6)$ and $(A7)$ are satisfied, provided
$\xi(t) / t \to + \infty$ as $t \to + \infty$. Let us also note that for assumptions $(A2)$ and $(A4)$--$(A9)$ to hold
true it is necessary that $\Lambda \subseteq K^*$.
\end{example}

\begin{remark}
As we will demonstrate below (see~Example~\ref{Example_SeparationCondition} and
Remark~\ref{Remark_SeparationCondition}), in the case when assumption $(A6)$ is not satisfied, it is necessary to
impose some rather restrictive assumptions on the problem $(\mathcal{P})$ in order to guarantee the existence of a
global saddle point of the augmented Lagrangian $\mathscr{L}(x, \lambda, c)$. In order to avoid this drawback one can
introduce an additional penalty term into the definition of the function $\Phi(y, \lambda, c)$ in the same way as in
example above (cf.~\cite{SunLi}).
\end{remark}

\begin{example} \label{Example_ModBarrierFunc}
Suppose that $J = \emptyset$. Let $\phi \colon (- \infty, 1) \to \mathbb{R}$ be a twice differentiable and strictly
increasing function such that $\phi(0) = 0$. Define
$$
  \Phi_i(y_i, \lambda_i, c) = \begin{cases}
    \dfrac{\lambda_i}{c} \phi(c y_i), & \text{if} \quad c y_i < 1, \\
    + \infty, & \text{otherwise}.
  \end{cases}
   \quad \forall i \in I
$$
In this case, the augmented Lagrangian function $\mathscr{L}(x, \lambda, c)$ coincides with \textit{the modified barrier
function} introduced by Polyak \cite{Polyak}. In particular, if $\phi(s) = - \ln(1 - s)$ or $\phi(s) = 1/(1 - s) - 1$,
then $\mathscr{L}(x, \lambda, c)$ coincides with the modified Frisch function or the modified Carroll function,
respectively \cite{Polyak} (see also \cite{SunLi,LiuYang,WangLi2009}).

Let $\Lambda = K^* = \mathbb{R}^l_+$. Then assumptions $(A1)$--$(A3)$ and $(A5)$--$(A10)$ are satisfied. Assumption 
$(A4)$ is satisfied, if $\phi$ is convex; assumption $(A11)$ is valid iff $\phi'(0) \ne 0$ 
($\Phi_0(\lambda) \equiv \phi'(0) \lambda$), while assumption $(A12)$ is satisfied, provided $\phi(t) / t \to 0$ as 
$t \to - \infty$. Furthermore, note that the function $\Phi_i(\cdot, \lambda_i, c)$ is l.s.c. iff $\phi(t) \to + \infty$
as $t \to 1$ and $\Lambda \subseteq K^*$. Thus, in particular, the modified Frisch and the modified Carroll functions
satisfy assumptions $(A1)$--$(A12)$ and are l.s.c. in $y$. Note, finally, that for assumptions $(A2)$, $(A4)$, $(A5)$,
$(A8)$ and $(A9)$ to hold true it is necessary that $\Lambda \subseteq K^*$.
\end{example}

\begin{example} \label{Example_pthPowerAugmLagr}
Suppose that $J = \emptyset$. Choose $b \ge 0$ and a non-decreasing function 
$\phi \colon \mathbb{R} \to \mathbb{R}_+$ such that $\phi(t) > \phi(b) > 0$ for all $t > b$ (in particular,
one can set $\phi(t) = \exp(t)$ or $\phi(t) = \max\{ 0, t \}$ with $b > 0$). Note that the inequality 
$g_i(x) \le 0$ is satisfied iff $\phi(g_i(x) + b) / \phi(b) \le 1$. Furthermore, 
$\phi(g_i(x) + b) \ge 0$ for all $x \in X$ by the definition of $\phi$. Define
$$
  \Phi_i(y_i, \lambda_i, c) = 
  \frac{\lambda_i}{c} \left[ \left(\frac{\phi(y_i + b)}{\phi(b)} \right)^c - 1 \right]
  \quad \forall i \in I
$$
(see~\cite{LiSun2001b}). Then $\mathscr{L}(x, \lambda, c)$ coincides with \textit{the p-th power augmented Lagrangian
function} \cite{Li1995,Li1997,Xu,LiSun2001a,LiSun2001b,WuLuo2012a,LiuYang}. 

Let $\Lambda = K^* = \mathbb{R}^l_+$. Then assumptions $(A1)$--$(A5)$, $(A8)$--$(A10)$ and $(A12)$ are satisfied.
Assumption $(A11)$ is satisfied provided $\phi'(b) \ne 0$ ($\Phi_0(\lambda) \equiv \phi'(b) \lambda$), while assumptions
$(A6)$ and $(A7)$ are not valid, regardless of the choice of $\Lambda$. Note also that, as in the above examples,
assumptions $(A2)$, $(A4)$, $(A5)$, $(A8)$ and $(A9)$ are satisfied iff $\Lambda \subseteq K^*$.
\end{example}

\begin{example} \label{Example_HeWuMengLagrangian}
Let $J = \emptyset$. Following the ideas of \cite{HeWuMeng}, define
$$
  \Phi_i(y_i, \lambda_i, c) = 
  \frac{1}{c} \int_0^{c y_i} \Big( \sqrt{t^2 + \lambda_i^2} + t \Big) \, d t \quad \forall i \in I.
$$
For any $\lambda_i \ne 0$ the function $\Phi_i$ has the form
$$
  \Phi_i(y_i, \lambda_i, c) = \frac{y_i}{2} \sqrt{(c y_i)^2 + \lambda_i^2} + \frac{c y_i^2}{2} + 
  \frac{\lambda_i^2}{2c} \ln \big( \sqrt{(c y_i)^2 + \lambda_i^2} + c y_i \big) - \frac{\lambda_i^2}{2c} \ln|\lambda_i|,
$$
while $\Phi_i(y_i, 0, c) = c y_i ( |y_i| + y_i ) / 2$. In this case, we refer to the function 
$\mathscr{L}(x, \lambda, c)$ as He-Wu-Meng's augmented Lagrangian.

Let $\Lambda = Y^* = \mathbb{R}^l$. Then one can verify that assumptions $(A1)$--$(A7)$ and $(A9)$--$(A12)$ are
satisfied, and $\Phi_0(\lambda) \equiv \lambda$. Assumption $(A8)$ holds true iff $\Lambda \subseteq K^*$, since, in
essence, He-Wu-Meng's augmented Lagrangian function does not distinguish multipliers $\lambda$ and $-\lambda$.
\end{example}

\subsection{Nonlinear Second Order Cone Programming}

Consider the following nonlinear second order cone programming problem:
\begin{equation} \label{SOCProg}
  \min f(x) \quad \text{subject to} 
  \quad g_i(x) \in Q_{l_i + 1}, \quad i \in I, \quad h(x) = 0, \quad x \in A,
\end{equation}
where $g_i \colon X \to \mathbb{R}^{l_i + 1}$, $I = \{ 1, \ldots, r \}$, and $h \colon X \to \mathbb{R}^s$ are given
functions, and 
$$
  Q_{l_i + 1} = \big\{ y = (y^0, \overline{y}) \in \mathbb{R} \times \mathbb{R}^{l_i} \bigm| 
  y^0 \ge \| \overline{y} \| \big\}
$$
is the second order (Lorentz) cone of dimension $l_i + 1$ (here $\| \cdot \|$ is the Euclidean norm). Denote 
$$
  Y = \mathbb{R}^{l_1 + 1} \times \ldots \times \mathbb{R}^{l_r + 1} \times \mathbb{R}^s, \quad
  K = Q_{l_1 + 1} \times \ldots \times Q_{l_r + 1} \times \{ 0_s \},
$$
and $G(\cdot) = (g_1(\cdot), \ldots, g_r(\cdot), h(\cdot))$. Then problem \eqref{SOCProg} is
equivalent to the problem $(\mathcal{P})$. Note that in this case 
$K^* = ( - Q_{l_1 + 1} ) \times \ldots \times ( - Q_{l_r + 1} ) \times \mathbb{R}^s$
and $Y^* = Y$.

\begin{example} \label{Example_RockWetsAL_SOC}
For any $y = (y_1, \ldots, y_r, z) \in Y$ and $\lambda = (\lambda_1, \ldots, \lambda_r, \mu) \in Y$ define
$$
  \Phi(y, \lambda, c) = \frac{c}{2} \sum_{i = 1}^r 
  \Big[ \dist^2\Big( y_i + \frac{1}{c} \lambda_i, Q_{l_i + 1} \Big) - \frac{1}{c^2} \| \lambda_i \|^2 \Big] +
  \langle \mu, z \rangle + \frac{c}{2} \| z \|^2.
$$
(see \cite{LiuZhang2007,LiuZhang2008,ZhouChen2015}). The function $\Phi(y, \lambda, c)$ defined above is a particular
case of the function $\Phi(y, \lambda, c)$ for the general cone constrained optimization problem from
Example~\ref{Example_RockafellarWetsAL} with $\sigma(y) = \| y \|^2 / 2$. Therefore it satisfies all assumptions
$(A1)$--$(A12)$ with 
$\Lambda = Y^* = \mathbb{R}^{l_1 + 1} \times \ldots \times \mathbb{R}^{l_r + 1} \times \mathbb{R}^s$, and
$\Phi_0(\lambda) \equiv \lambda$.
\end{example}

Following the ideas of \cite{ZhangGuXiao2011} we can define another augmented Lagrangian function for problem
\eqref{SOCProg}, which is a generalization of the augmented Lagrangians from Examples~\ref{Example_ExpPenFunc} and
\ref{Example_ModBarrierFunc} to the case of nonlinear second order cone programming problems. To this end, recall that
for any function $\psi \colon \mathbb{R} \to \mathbb{R}$ and for any $y = (y^0, \overline{y}) \in \mathbb{R}^{l + 1}$
\textit{L\"owner's operator} associated with $\psi$ is defined as
$$
  \Psi(y) = \frac{1}{2}
  \begin{pmatrix}
  \psi( y^0 + \| \overline{y} \| ) + \psi( y^0 - \| \overline{y} \| ) \\
  \Big( \psi( y^0 + \| \overline{y} \| ) - \psi( y^0 - \| \overline{y} \| ) \Big) 
  \dfrac{\overline{y}}{\| \overline{y} \|}
  \end{pmatrix},
$$
if $\overline{y} \ne 0$, and $\Psi(y) = ( \psi(y^0), 0_l )$ otherwise
(see~\cite{ZhangGuXiao2011,ChenChenTseng2004,SunSun2008} for more details). It is easy to check that if $\psi$ is
strictly increasing and $\psi(0) = 0$, then $y \in K \implies -\Psi(-y) \in K$, while 
$y \notin K \implies -\Psi(-y) \notin K$.

\begin{example} \label{Example_NonlinearRescale_SOC}
Suppose, for the sake of simplicity, that there are no equality constraints. Let 
$\psi \colon \mathbb{R} \to \mathbb{R} \cup \{ + \infty \}$ be a non-decreasing convex function such that
$\dom \psi = (- \infty, \varepsilon_0)$ for some $\varepsilon_0 \in (0, + \infty]$, $\psi(t) \to + \infty$ as 
$t \to \varepsilon_0$, $\psi(t) / t \to + \infty$ as $t \to + \infty$ if $\varepsilon_0 = + \infty$, 
$\psi$ is twice continuously differentiable on $\dom \psi$, $\psi(0) = 0$ and $\psi'(0) = 1$.
Then for any $y = (y_1, \ldots, y_r) \in Y$ and $\lambda = (\lambda_1, \ldots, \lambda_r) \in Y$ define
\begin{equation} \label{Ex_SOC_NonlinRescaling}
  \Phi(y, \lambda, c) = - \frac{1}{c} \sum_{i = 1}^r \big\langle \lambda_i, \Psi(- c y_i ) \big\rangle,
\end{equation}
if $y_i^0 + \| \overline{y}_i \| < \varepsilon_0$ for all $i \in I$, and $\Phi(y, \lambda, c) = + \infty$ otherwise.
It is easy to see that the function $\Phi(y, \lambda, c)$ is lower semicontinuous.

Let $\Lambda = K^*$. Then one can easily verify that assumptions $(A1)$--$(A4)$, $(A8)$--$(A10)$ are satisfied in the
general case. Assumption $(A11)$ is satisfied with $\Phi_0(\lambda) \equiv \lambda$ by \cite{ZhangGuXiao2011},
Lemma~3.1. Assumptions $(A6)$ and $(A7)$ are valid, if $\varepsilon_0 < + \infty$, while assumption $(A12)$ is valid,
provided $\psi(t) / t \to 0$ as $t \to - \infty$. Finally, assumption $(A5)$ is never satisfied for the function $\Phi$
defined above, since for $(A5)$ to hold true it is necessary that L\"owner's operator $y \to \Psi(-y)$ is non-increasing
with respect to the order generated by the second order cone, which is not the case when $\psi$ is non-decreasing. Let
us also note that for assumptions $(A2)$, $(A4)$, $(A8)$ and $(A9)$ to hold true it is necessary 
that $\Lambda \subseteq K^*$.
\end{example}

\begin{remark}
Note that one can easily incorporate equality constraints into the augmented Lagrangian function from the previous
example by simply adding terms corresponding to these constraints into the right-hand side of
\eqref{Ex_SOC_NonlinRescaling}. One can define these additional terms in the same way as in
Examples~\ref{Example_Mangasarian} or \ref{Example_RockWetsAL_SOC}.
\end{remark}

\subsection{Nonlinear Semidefinite Programming}

Consider the following nonlinear semidefinite programming problem:
\begin{equation} \label{SemiDefProg}
  \min f(x) \quad \text{subject to} 
  \quad G_0(x) \preceq 0, \quad h(x) = 0, \quad x \in A,
\end{equation}
where $G_0 \colon X \to \mathbb{S}^l$ and $h \colon X \to \mathbb{R}^s$ are given functions, $\mathbb{S}^l$ denotes the
set of all $l \times l$ real symmetric matrices, and the relation $G_0(x) \preceq 0$ means that the matrix $G_0(x)$ is
negative semidefinite. Hereinafter, we suppose that the linear space $\mathbb{S}^l$ is equipped with the inner product
$\langle A, B \rangle = \trace(AB)$, and the corresponding norm $\| A \|_F = \sqrt{\trace(A^2)}$, which is called the
Frobenius norm of a matrix $A \in \mathbb{S}^l$. Here $\trace(\cdot)$ is the trace operator.

Denote by $\mathbb{S}^l_+$ the cone of $l \times l$ positive semidefinite matrices, and denote $\mathbb{S}^l_-$ by the
cone of $l \times l$ negative semidefinite matrices. Define $Y = \mathbb{S}^l \times \mathbb{R}^s$, 
$K = \mathbb{S}^l_- \times \{ 0_s \}$ and $G(\cdot) = ( G_0(\cdot), h(\cdot) )$. Then problem \eqref{SemiDefProg} is
equivalent to the problem $(\mathcal{P})$. Note that in this case $K^* = \mathbb{S}^l_+ \times \mathbb{R}^s$ and 
$Y^* = Y$.

\begin{example} \label{Example_RockWetsAL_SemiDefProg}
For any $y = (y_0, z) \in Y = \mathbb{S}^l \times \mathbb{R}^s$ and $\lambda = (\lambda_0, \mu) \in Y$ define
$$
  \Phi(y, \lambda, c) = \frac{1}{2c} \Big( \trace\big( [c y_0 + \lambda_0]_+^2 \big) - \trace(\lambda_0^2) \Big) + 
  \langle \mu, z \rangle + \frac{c}{2} \| z \|^2,
$$
where $[\cdot]_+$ denotes the projection of a matrix onto the cone $\mathbb{S}^l_+$
(see \cite{SunZhangWu2006,SunSunZhang2008,ZhaoSunToh2010,Sun2011,WenGoldfarbYin2010,LuoWuChen2012,WuLuoDingChen2013,
WuLuoYang2014,YamashitaYabe2015} for more details on this augmented Lagrangian). One can check that the function 
$\Phi(y, \lambda, c)$ defined above is a particular case of the function $\Phi(y, \lambda, c)$ 
from Example~\ref{Example_RockafellarWetsAL} with $\sigma(y) = (\| y_0 \|_F^2 + \| z \|^2) / 2$. Therefore it satisfies
all assumptions $(A1)$--$(A12)$ with $\Lambda = Y^* = \mathbb{S}^l \times \mathbb{R}^s$ and 
$\Phi_0(\lambda) \equiv \lambda$.
\end{example}

As in the case of the second order cone programs, one can extend the augmented Lagrangians from
Examples~\ref{Example_ExpPenFunc} and \ref{Example_ModBarrierFunc} to the case of nonlinear semidefinite programming
problems with the use of L\"owner's operator (see \cite{SunSun2008,Shapiro2002,ChenQiTseng2003}). For any function
$\psi \colon \mathbb{R} \to \mathbb{R}$ and any $y \in \mathbb{S}^l$ the matrix function (L\"owner's operator)
associated with $\psi$ is defined as
$$
  \Psi(y) = E \diag\Big( \psi(\rho_1(y)), \ldots, \psi(\rho_l(y)) \Big) E^T,
$$
where $y = E \diag(\rho_1(y), \ldots, \rho_l(y)) E^T$ is a spectral decomposition of $y \in \mathbb{S}^l$, and 
$\rho_1(y), \ldots, \rho_l(y)$ are the eigenvalues of $y$ listed in the decreasing order. Note that 
the projection operator $[\cdot]_+$ is simply the matrix function associated with the function 
$\psi(t) = \max\{ 0, t \}$.

\begin{example} \label{Example_NonlinearRescale_SemiDefProg}
Let $\psi \colon \mathbb{R} \to \mathbb{R} \cup \{ + \infty \}$ be a non-decreasing convex function such 
that $\dom \psi = (- \infty, \varepsilon_0)$ for some $\varepsilon_0 \in (0, + \infty]$, $\psi(t) \to + \infty$ as 
$t \to \varepsilon_0$, $\psi(t) / t \to + \infty$ as $t \to + \infty$ if $\varepsilon_0 = + \infty$, 
$\psi$ is twice continuously differentiable on $\dom \psi$, $\psi(0) = 0$ and $\psi'(0) = 1$. For
any $y = (y_0, z) \in Y = \mathbb{S}^l \times \mathbb{R}^s$ and $\lambda = (\lambda_0, \mu) \in Y$ define
$$
  \Phi(y, \lambda, c) = \frac{1}{c} \big\langle \lambda_0, \Psi(c y_0) \big\rangle + 
  \langle \mu, z \rangle + \frac{c}{2} \| z \|^2,
$$
if $\sigma_1(y_0) < \varepsilon_0$, and $\Phi(y, \lambda, c) = + \infty$ otherwise (see
\cite{Stingl2006,Noll2007,LiZhang2009,ZhangLiWu2014,YamashitaYabe2015, LuoWuLiu2015}). Note that the function
$\Phi(y, \lambda, c)$ is lower semicontinuous.

Let $\Lambda = K^* = \mathbb{S}^l_+ \times \mathbb{R}^s$. Then assumptions $(A1)$--$(A4)$ and $(A8)$--$(A10)$ hold
true. Furthermore, assumption $(A11)$ is satisfied with $\Phi_0(\lambda) \equiv \lambda$ by \cite{LuoWuLiu2015},
Proposition~4.2. Assumption $(A5)$ is satisfied, if the matrix function $\Psi(y)$ is monotone (see, e.g.,
\cite{HornJohnson}, Def.~6.6.33). Assumptions $(A6)$ and $(A7)$ hold true, provided $\varepsilon_0 < + \infty$, while
assumption $(A12)$ is valid iff $\psi(t) / t \to 0$ as $t \to -\infty$. Note also that for assumptions $(A2)$,
$(A4)$, $(A5)$, $(A8)$ and $(A9)$ to hold true it is necessary that $\Lambda \subseteq K^*$.
\end{example}

Following the ideas of \cite{LuoWuLiu2015} we can also extend the penalized exponential-type augmented Lagrangian
function from Example~\ref{Example_PenalizedExpPenFunc} to the case of nonlinear semidefinite programming problems.

\begin{example} \label{Example_NonlinearRescalePenalized_SemiDefProg}
Let the function $\psi$ be as in the previous example with $\varepsilon_0 = + \infty$, and let 
$\xi \colon \mathbb{R} \to \mathbb{R}$ be a twice continuously differentiable non-decreasing convex function such that
$\xi(t) = 0$ for all $t \le 0$, and $\xi(t) > 0$ for all $t > 0$. Denote by $\Xi(\cdot)$ the matrix function associated
with $\xi(\cdot)$. 

For any $y = (y_0, z) \in Y$ and $\lambda = (\lambda_0, \mu) \in Y$ define
$$
  \Phi(y, \lambda, c) = \frac{1}{c} \big\langle \lambda_0, \Psi(c y_0) \big\rangle + 
  \frac{1}{c} \trace\big( \Xi( c y_0) \big) + \langle \mu, z \rangle + \frac{c}{2} \| z \|^2
$$
(see~\cite{LuoWuLiu2015}). Let $\Lambda = K^* = \mathbb{S}^l_+ \times \mathbb{R}^s$. Then assumptions $(A1)$--$(A4)$ and
$(A8)$--$(A11)$ are satisfied. Assumption $(A5)$ hold true, provided both matrix functions $\Psi$ and $\Xi$ are
monotone, assumptions $(A6)$ and $(A7)$ are satisfied if and, in the case $l > 1$, only if $\xi(t) / t \to + \infty$ as
$t \to +\infty$, and $\psi$ is bounded below, while assumption $(A12)$ is valid iff $\psi(t) / t \to 0$ as $t \to -
\infty$. 
\end{example}

\subsection{Semi-Infinite Programming}

Consider the following semi-infinite programming problem:
\begin{equation} \label{SemiInfProg}
  \min f(x) \quad \text{subject to} 
  \quad g_i(x, t) \le 0, \quad t \in T, \quad i \in I, \quad h(x) = 0, \quad x \in A,
\end{equation}
where $T$ is a compact metric space, the functions $g_i \colon X \times T \to \mathbb{R}$, $I = \{ 1, \ldots, l \}$, are
continuous, and $h \colon X \to \mathbb{R}^s$. Let $C(T)$ be the space of all real-valued continuous functions defined
on $T$ endowed with the uniform norm, and denote by $C_+(T)$ the closed convex subcone of $C(T)$ consisting of all
nonnegative functions. As it is well-known, the topological dual space of $C(T)$ is isometrically isomorphic to the
space of signed (i.e. real-valued) regular Borel measures on $T$, which we denote by $rca(T)$, while the space of
regular Borel measures on $T$ is denoted as $rca_+(T)$.

Define $Y = (C(T))^l \times \mathbb{R}^s$ and $K = (- C_+(T))^l \times \{ 0_s \}$, and introduce the function 
$x \to G(x) = (g_1(x, \cdot), \ldots, g_l(x, \cdot), h(x))$ mapping $X$ to $Y$. Then problem \eqref{SemiInfProg} is
equivalent to the problem $(\mathcal{P})$. Note that in this case $Y^*$ is isometrically isomorphic to (and, thus, can
be identified with) the space $(rca(T))^l \times \mathbb{R}^s$, while $K^*$ can be identified with 
the cone $(rca_+(T))^l \times \mathbb{R}^s$.

To the best of author's knowledge, the only augmented Lagrangian function for semi-infinite programming problems studied
in the literature is a particular case of the Rockafellar-Wets augmented Lagrangian from
Example~\ref{Example_RockafellarWetsAL} (see~\cite{RuckmannShapiro,HuyKim2012,SonKimTam2012,BurachikYangZhou2017}). 

\begin{example} \label{Example_RockafellarWetsAL_SemiInfProg}
Let a function $\sigma \colon C(T) \to \mathbb{R}_+$ be such that $\sigma(0) = 0$ and $\sigma(y) > 0$ for all $y \ne 0$.
Suppose also that $\sigma$ has a valley at zero. For any $y = (y_1, \ldots, y_l, z) \in Y$ (i.e. $y_i \in C(T)$ and 
$z \in \mathbb{R}^s$) and $\lambda = (\lambda_1, \ldots, \lambda_l, \mu) \in Y^*$ (i.e. $\lambda_i \in rca(T)$ and 
$\mu \in \mathbb{R}^s$) define
\begin{multline*}
  \Phi(y, \lambda, c) \\ 
  = \sum_{i = 1}^l 
  \inf\Big\{ - \int_T p d \lambda_i + c \sigma(p) \Bigm| p \in C(T), \: p(\cdot) + y_i(\cdot) \le 0 \Big\} 
  + \langle \mu, z \rangle + \frac{c}{2} \| z \|^2.
\end{multline*}
Let $\Lambda \subseteq Y^*$ (in particular, one can define $\Lambda$ as the set of all measures $\lambda \in rca(T)$
with finite support, see \cite{RuckmannShapiro,BurachikYangZhou2017}). Then, as it was pointed out in
Example~\ref{Example_RockafellarWetsAL}, assumptions $(A1)$--$(A7)$ and $(A10)$ are satisfied. Assumptions $(A8)$ and
$(A9)$ are valid, if $\sigma(t y) / t \to 0$ as $t \to 0$ for any $y \in Y$, while assumption $(A12)$ is satisfied,
provided $\sigma(y) \ge \omega( \| y \| )$ for all $y \in Y$, where $\omega$ is a non-negative continuous function such
that $\omega(t) = 0$ iff $t = 0$ and $\liminf_{t \to + \infty} \omega(t) / t > 0$. Finally, it seems that there are no
natural general assumptions on the function $\sigma$ and the cone $\Lambda$ which can guarantee the validity of $(A11)$.
\end{example}

Let us also present a new augmented Lagrangian function for problem \eqref{SemiInfProg}, which is a simple extension of
the augmented Lagrangian functions from Examples~\ref{Example_ExpPenFunc} and \ref{Example_ModBarrierFunc} to the case
of semi-infinite programming problems.

\begin{example} \label{Example_NonlinearRescale_SemiInfProg}
Suppose, for the sake of simplicity, that there are no equality constraints. Let 
$\phi \colon \mathbb{R} \to \mathbb{R} \cup \{ + \infty \}$ be a non-decreasing convex function such that 
$\dom \phi = (- \infty, \varepsilon_0)$ for some $\varepsilon_0 \in (0, + \infty]$, $\phi(t) \to + \infty$ as 
$t \to \varepsilon_0$, $\phi(t) / t \to + \infty$ as $t \to + \infty$ if $\varepsilon_0 = + \infty$, $\phi$ is twice
continuously differentiable on $\dom \phi$, $\phi(0) = 0$ and $\phi'(0) = 1$. For any $y = (y_1, \ldots, y_l) \in Y$ and
$\lambda = (\lambda_1, \ldots, \lambda_l) \in Y^*$ define
\begin{equation} \label{ModBarrierFunc_SemiInfProg}
  \Phi(y, \lambda, c) = \frac{1}{c} \sum_{i = 1}^l \int_T \phi(c y_i) \, d \lambda_i,
\end{equation}
if $\| y_i \| < \varepsilon_0$ for all $i \in I$, and $\Phi(y, \lambda, c) = + \infty$ otherwise.

Let $\Lambda \subseteq K^* = (rca_+(T))^l$. Then assumptions $(A1)$--$(A5)$ and $(A8)$--$(A11)$ are satisfied
($\Phi_0(\lambda) \equiv \lambda$). Assumptions $(A6)$ and $(A7)$ are valid, provided $\varepsilon_0 < + \infty$, while
assumption $(A12)$ is valid iff $\phi(t) / t \to 0$ as $t \to -\infty$.
\end{example}

\begin{remark}
{(i)~Note that if one defines function \eqref{ModBarrierFunc_SemiInfProg} for $\lambda \in K^*$, then there is
no need to separate the case when $\| y_i \| \ge \varepsilon_0$ for some $i \in I$.
}

\noindent{(ii)~If $T \subset \mathbb{R}^q$ for some $q \in \mathbb{N}$, then one can also extend
the penalized exponential-type augmented Lagrangian function from Example~\ref{Example_PenalizedExpPenFunc} to the case
of semi-infinite programming problems by simply adding the term $c^{-1} \int_T \xi(c y_i(t)) \, dt$ to the right-hand
side of \eqref{ModBarrierFunc_SemiInfProg}, where the function $\xi$ is the same as in
Example~\ref{Example_PenalizedExpPenFunc}.
}
\end{remark}

\section{Saddle Points of Augmented Lagrangian Functions}
\label{Section_SaddlePoints}

Let us turn to the study of saddle points of the augmented Lagrangian function $\mathscr{L}(x, \lambda, c)$. Recall that
we formulate all definitions and results with respect to a given closed convex cone $\Lambda \subseteq Y^*$. 

\begin{definition}
A pair $(x_*, \lambda_*) \in A \times \Lambda$ is called a \textit{global saddle point} of 
the augmented Lagrangian $\mathscr{L}(x, \lambda, c)$ if there exists $c_0 > 0$ such that 
$\mathscr{L}(x_*, \lambda_*, c) < + \infty$ for all $c \ge c_0$ and
$$
  \sup_{\lambda \in \Lambda} \mathscr{L}(x_*, \lambda, c) \le \mathscr{L}(x_*, \lambda_*, c) \le 
  \inf_{x \in A} \mathscr{L}(x, \lambda_*, c)	\qquad \forall c \ge c_0.
$$
The greatest lower bound of all such $c_0$ is denoted by $c^*(x_*, \lambda_*)$ and is referred to as \textit{the least
exact penalty parameter} for the global saddle point $(x_*, \lambda_*)$.
\end{definition}

\begin{definition}
A pair $(x_*, \lambda_*) \in A \times \Lambda$ is called a \textit{local saddle point} of 
the augmented Lagrangian $\mathscr{L}(x, \lambda, c)$ if there exist $c_0 > 0$ and a neighbourhood $U$ of $x_*$ such
that $\mathscr{L}(x_*, \lambda_*, c) < + \infty$ for all $c \ge c_0$ and
\begin{equation} \label{LocalSaddlePoint_FDef}
  \sup_{\lambda \in \Lambda} \mathscr{L}(x_*, \lambda, c) \le \mathscr{L}(x_*, \lambda_*, c) \le 
  \inf_{x \in U \cap A} \mathscr{L}(x, \lambda_*, c) \qquad \forall c \ge c_0.
\end{equation}
The greatest lower bound of all such $c_0$ is denoted by $c^*_{loc}(x_*, \lambda_*)$ and is referred to as 
\textit{the least local exact penalty parameter} for the local saddle point $(x_*, \lambda_*)$.
\end{definition}

Observe a direct connection between saddle points of the augmented Lagrangian $\mathscr{L}(x, \lambda, c)$ and optimal
solutions
of the problem $(\mathcal{P})$. 

\begin{proposition} \label{Prp_SaddlePointImpliesOptSol}
Suppose that assumptions $(A1)$--$(A3)$ are satisfied. If $(x_*, \lambda_*)$ is a local (global) saddle point of
the augmented Lagrangian $\mathscr{L}(x, \lambda, c)$, then $x_*$ is a locally (globally) optimal solution of 
the problem $(\mathcal{P})$.
\end{proposition}

\begin{proof}
Let $(x_*, \lambda_*)$ be a local saddle point of $\mathscr{L}(x, \lambda, c)$. Then $\mathscr{L}(x_*, \lambda_*, c)$ is
finite for all $c > 0$, and there exist $c_0 > 0$ and a neighbourhood $U$ of $x_*$ such that
\eqref{LocalSaddlePoint_FDef} holds true. Let us prove, at first, that $x_*$ is a feasible point of the problem
$(\mathcal{P})$. Indeed, suppose that $x_*$ is infeasible, i.e. that $G(x_*) \notin K$. Then by $(A3)$ for any $c > 0$
there exists $\lambda_0 \in \Lambda$ such that $\lim_{t \to +\infty} \Phi(G(x_*), t \lambda_0, c) = + \infty$. From
\eqref{LocalSaddlePoint_FDef} it follows that for any $c \ge c_0$ one has
$$
  \mathscr{L}(x_*, \lambda_*, c) \ge \mathscr{L}(x_*, t \lambda_0, c) = f(x_*) + \Phi(G(x_*), t \lambda_0, c)
  \quad \forall t > 0.
$$
Passing to the limit as $t \to +\infty$ one obtains that $\mathscr{L}(x_*, \lambda_*, c) = + \infty$, which contradicts
the definition of local saddle point. Thus, $G(x_*) \in K$.

Let $x \in U$ be a feasible point of the problem $(\mathcal{P})$. Then applying the second inequality in
\eqref{LocalSaddlePoint_FDef}, and $(A2)$ one obtains that
\begin{equation} \label{LocalSaddlePoint_Crlr}
  \mathscr{L}(x_*, \lambda_*, c) \le \mathscr{L}(x, \lambda_*, c) \le f(x) \quad \forall c \ge c_0.
\end{equation}
Applying $(A2)$ again one gets that $\Phi(G(x_*), \lambda_*, c) \le 0$. On the other hand, from the first
inequality in \eqref{LocalSaddlePoint_FDef} and $(A1)$ it follows that
$\Phi(G(x_*), \lambda_*, c) \ge \Phi(G(x_*), 0, c) \ge 0$ for all $c \ge c_0$, which yields that 
$\Phi(G(x_*), \lambda_*, c) = 0$ for all $c \ge c_0$. Hence and from~\eqref{LocalSaddlePoint_Crlr} one gets that 
$f(x_*) \le f(x)$ for any feasible point $x \in U$, which implies that $x_*$ is a locally optimal solution of 
the problem $(\mathcal{P})$. 

Repeating the same argument as above with $U$ replaced by $A$ one obtains that if $(x_*, \lambda_*)$ is a global saddle
point of $\mathscr{L}(x, \lambda, c)$, then $x_*$ is a globally optimal solution of the problem $(\mathcal{P})$.  
\end{proof}

\begin{remark} \label{Rmrk_ALValueAtSP}
Let assumptions $(A1)$--$(A3)$ be satisfied. From the proof of the proposition above it follows that if
$(x_*, \lambda_*) \in A \times \Lambda$ is such that
$$
  \sup_{\lambda \in \Lambda} \mathscr{L}(x_*, \lambda, c_0) \le \mathscr{L}(x_*, \lambda_*, c_0) < + \infty
$$
for some $c_0 > 0$, then $x_*$ is feasible point of the problem $(\mathcal{P})$, and
$\mathscr{L}(x_*, \lambda_*, c_0) = f(x_*)$. If, additionally, assumption $(A4)$ holds true, then 
$\mathscr{L}(x_*, \lambda_*, c) = f(x_*)$ for all $c \ge c_0$. In particular, if $(x_*, \lambda_*)$ is a local saddle
point of $\mathscr{L}(x, \lambda, c)$, and assumptions $(A1)$--$(A3)$ are valid, then
$\mathscr{L}(x_*, \lambda_*, c) = f(x_*)$ for all $c > c^*_{loc}(x_*, \lambda_*)$.
\end{remark}

Below we will show that there is a direct connection between local (global) saddle points of the augmented Lagrangian
function $\mathscr{L}(x, \lambda, c)$ and KKT-pairs of the problem $(\mathcal{P})$ corresponding to locally (globally)
optimal solutions of this problem. Since for any KKT-pair $(x_*, \lambda_*)$ one has $\lambda_* \in K^*$ and 
$\langle \lambda_*, G(x_*) \rangle = 0$, it is natural to ask when these conditions are satisfied for local or
global saddle points. The two following propositions give an answer to this question.

\begin{proposition} \label{Prp_NonnegativeMultiplierAtSP}
Let assumptions $(A1)$, $(A3)$ and $(A8)$ be valid, and let a pair $(x_*, \lambda_*) \in A \times \Lambda$
satisfy the inequalities
\begin{equation} \label{HalfSaddlePoint}
  \sup_{\lambda \in \Lambda} \mathscr{L}(x_*, \lambda, c) \le \mathscr{L}(x_*, \lambda_*, c) < + \infty
\end{equation}
for some $c > 0$. Then $\lambda_* \in K^*$. 
\end{proposition}

\begin{proof}
Arguing by reductio ad absurdum, suppose that $\lambda_* \notin K^*$. Repeating the first part of the proof of
Proposition~\ref{Prp_SaddlePointImpliesOptSol} one can easily verify that $x_*$ is a feasible point of the problem
$(\mathcal{P})$. Then applying $(A1)$ and $(A8)$ one gets that 
$\Phi( G(x_*), \lambda_*, c) < 0 \le \Phi( G(x_*), 0, c)$ for all $c > 0$. Hence for any $c > 0$ one has
$\mathscr{L}(x_*, \lambda_*, c) < \mathscr{L}(x_*, 0, c)$, which contradicts \eqref{HalfSaddlePoint}.	 
\end{proof}

Let us demonstrate that the proposition above does not hold true in the case when assumption $(A8)$ is not valid. To
this end, let us consider the exponential penalty function from Example~\ref{Example_ExpPenFunc} with 
$\Lambda = Y^* = \mathbb{R}^l$ (note that in this case assumption $(A8)$ is not satisfied).

\begin{example} \label{Exmpl}
Let $X = A = \mathbb{R}^2$. Consider the following optimization problem
\begin{equation*}
  \begin{split}
  &\min f(x_1, x_2) = x_1^2 + x_2^2 \qquad \text{subject to} \\
  &g_1(x_1, x_2) = x_1 + x_2 + 2 \le 0, \quad
  g_2(x_1, x_2) = \frac12 (x_1 + 2)^2 + \frac12 (x_2 + 2)^2 - 1 \le 0.
  \end{split}
\end{equation*}
Let $x_* = (-1, -1)$ and $\lambda_* = (- 1, 3)$. Let us show that the pair $(x_*, \lambda_*)$ is a global saddle
point of the exponential penalty function $\mathscr{L}(x, \lambda, c)$ despite the fact that 
$\lambda_* \notin K^* = \mathbb{R}^2_+$. Indeed, by definition
$$
  \mathscr{L}(x, \lambda, c) = f(x) + 
  \frac{\lambda_1}{c} \big( e^{c g_1(x)} - 1 \big) + \frac{\lambda_2}{c} \big( e^{c g_2(x)} - 1 \big).
$$
The equalities $g_1(x_*) = g_2(x_*) = 0$ imply that $\mathscr{L}(x_*, \cdot, c) \equiv f(x_*) = 2$, which yields that
$$
  \sup_{\lambda \in \mathbb{R}^2} \mathscr{L}(x_*, \lambda, c) = \mathscr{L}(x_*, \lambda_*, c) = 2 
  \quad \forall c > 0
$$
Let us prove that $x_*$ is a point of global minimum of the function $\mathscr{L}(\cdot, \lambda_*, c)$ for any $c > 0$.
Then one gets that $(x_*, \lambda_*)$ is a global saddle point of the exponential penalty function for the above
problem.

Note that $g_2(x) \ge g_1(x)$ for all $x \in \mathbb{R}^2$. Therefore taking into account the fact 
that $\lambda^* = (-1, 3)$ one gets that for any $x \in \mathbb{R}^2$ and $c > 0$ one has
\begin{equation} \label{ExpPenFuncExample_LowerEst}
  \mathscr{L}(x, \lambda_*, c) = f(x) + \frac{2}{c}\big( e^{c g_2(x)} - 1 \big) +
  \frac{1}{c} \big( e^{c g_2(x)} - e^{c g_1(x)} \big) \ge f_c(x),
\end{equation}
where $f_c(x) = f(x) + 2( e^{c g_2(x)} - 1) / c$. It is easy to see that for any $c > 0$ the function $f_c(x)$ is
convex, and $\nabla f_c(x_*) = 0$, which implies that $x_*$ is a point of global minimum of $f_c$ for all $c > 0$. Hence
taking into account inequality~\eqref{ExpPenFuncExample_LowerEst}, and the fact that 
$\mathscr{L}(x_*, \lambda_*, c) = f_c(x_*) = 2$ one obtains the required result.
\end{example}

\begin{remark}
As the previous example shows, one needs to properly choose the cone $\Lambda$ in order to ensure the validity of 
the inclusion $\lambda_* \in K^*$. Alternatively, one can guarantee the validity of the inclusion $\lambda_* \in K^*$
with the use of a suitable constraint qualification. In particular, it is easy to check that if $(x_*, \lambda_*)$ is a
local saddle point of the exponential penalty function with $\Lambda = \mathbb{R}^l$, and LICQ holds true at $x_*$, then
$\lambda_* \in K^*$. On the other hand, note that if at least one of the inequality constraints is not active at $x_*$,
say $g_i(x_*) < 0$, then for $(x_*, \lambda_*)$ to be a local saddle point of the exponential penalty function it is
necessary that $\lambda_i \ge 0$ for all $\lambda \in \Lambda$.
\end{remark}

Arguing in a similar way to the proof of Proposition~\ref{Prp_NonnegativeMultiplierAtSP} one can check that the
following result holds true.

\begin{proposition} \label{Prp_ComplementarySlackAtSP}
Let assumptions $(A1)$, $(A3)$ and $(A9)$ hold true, and let a pair $(x_*, \lambda_*) \in A \times \Lambda$ satisfy
inequalities \eqref{HalfSaddlePoint} for some $c > 0$. Then $\langle \lambda_*, G(x_*) \rangle = 0$.
\end{proposition}

\begin{remark}
Let us note that if 
\begin{equation} \label{GenALF_ExpPenTypeStructure}
  \Phi(y, \lambda, c) = \langle \lambda, \Phi_1(y, c) \rangle + \Phi_2(y, c)
\end{equation}
for some functions $\Phi_1$ and $\Phi_2$, and $(x_*, \lambda_*)$ is a local saddle point of 
$\mathscr{L}(x, \lambda, c)$ with $\Lambda = Y^*$, then
$$
  \big\langle \lambda, \Phi_1(G(x_*), c) \big\rangle \le 
  \big\langle \lambda_*, \Phi_1(G(x_*), c) \big\rangle \quad \forall \lambda \in Y^* \quad 
  \forall c > c^*_{loc}(x_*, \lambda_*),
$$
which in all particular examples presented above implies that $\langle \lambda_*, G(x_*) \rangle = 0$. Note also that
assumption $(A9)$ is satisfied with $\Lambda = Y^*$ for all particular augmented Lagrangian functions that cannot be
represented in the form \eqref{GenALF_ExpPenTypeStructure}.
\end{remark}

Let us now provide several useful characterizations of global saddle points of the augmented Lagrangian function
$\mathscr{L}(x, \lambda, c)$. Global saddle points of $\mathscr{L}(x, \lambda, c)$ can be characterized in terms of
solutions of the \textit{augmented dual problem}. Recall that the augmented dual problem of $(\mathcal{P})$ associated
with the augmented Lagrangian $\mathscr{L}(x, \lambda, c)$ has the form
$$
  \max_{(\lambda, c)} \Theta(\lambda, c) \quad \text{subject to} \quad \lambda \in \Lambda, \quad c > 0,
$$
where $\Theta(\lambda, c) = \inf_{x \in A} \mathscr{L}(x, \lambda, c)$. Note that if assumption $(A2)$ is satisfied,
then \textit{the weak duality} between $(\mathcal{P})$ and the augmented dual problem holds, i.e.
$$
  \Theta(\lambda, c) \le f(x) \quad 
  \forall (\lambda, c) \in \Lambda \times (0, + \infty) \quad \forall x \in \Omega,
$$
where $\Omega$ is the feasible set of the problem $(\mathcal{P})$. One says that \textit{the zero duality gap
property} holds true for the augmented Lagrangian $\mathscr{L}(x, \lambda, c)$, if
$$
  \inf_{x \in \Omega} f(x) = 
  \sup\big\{ \Theta(\lambda, c) \bigm| \lambda \in \Lambda, \: c > 0 \big\}.
$$
Now, we can obtain a simple and well-known characterization of global saddle points in terms of globally optimal
solutions of the augmented dual problem.

\begin{proposition} \label{Prp_SaddlePointsVSOptSolDualProb}
Suppose that assumptions $(A1)$--$(A4)$ are satisfied. If $(x_*, \lambda_*)$ is a global saddle point of the augmented
Lagrangian $\mathscr{L}(x, \lambda, c)$, then for any $c > c^*(x_*, \lambda_*)$ the pair $(\lambda_*, c)$ is a globally
optimal solution of the augmented dual problem, and the zero duality gap property holds true. Conversely, if
$(\lambda_*, c_*)$ is a globally optimal solution of the augmented dual problem, and the zero duality gap property holds
true, then for any globally optimal solution $x_*$ of the problem $(\mathcal{P})$ the pair $(x_*, \lambda_*)$ is a
global saddle point of $\mathscr{L}(x, \lambda, c)$ and $c_* \ge c^*(x_*, \lambda_*)$.
\end{proposition}

\begin{proof}
Let $(x_*, \lambda_*)$ be a global saddle point of $\mathscr{L}(x, \lambda, c)$. Applying Remark~\ref{Rmrk_ALValueAtSP}
one
obtains that for any $c > c^*(x_*, \lambda_*)$ the following inequalities hold true
$$
  \sup_{\lambda \in \Lambda} \mathscr{L}(x_*, \lambda, c) \le \mathscr{L}(x_*, \lambda_*, c) = f(x_*) 
  \le \inf_{x \in A} \mathscr{L}(x, \lambda_*, c).
$$
Consequently, taking into account the definition of $\Theta$ one gets that 
$$
  \Theta(\lambda, c) \le f(x_*) \le \Theta(\lambda_*, c)
  \quad \forall \lambda \in \Lambda \quad \forall c > c^*(x_*, \lambda_*).
$$
Assumption $(A4)$ implies that $\Theta(\lambda, c)$ is non-decreasing in $c$ for any $\lambda \in \Lambda$. Therefore 
$\Theta(\lambda, c) \le f(x_*)$ for all $(\lambda, c) \in \Lambda \times (0, + \infty)$. Hence taking the supremum over
all $(\lambda, c) \in \Lambda \times (0, + \infty)$ one obtains that
$$
  \sup\big\{ \Theta(\lambda, c) \bigm| (\lambda, c) \in \Lambda \times (0, + \infty) \big\} \le f(x_*) \le
  \Theta(\lambda_*, c) \quad \forall c > c^*(x_*, \lambda_*).
$$
Consequently, $(\lambda_*, c)$ is a globally optimal solution of the augmented dual problem for any 
$c > c^*(x_*, \lambda_*)$, and the zero duality gap property holds true, since $x_*$ is a globally optimal solution of 
the problem $(\mathcal{P})$ by virtue of Proposition~\ref{Prp_SaddlePointImpliesOptSol}. 

Let, now, $(\lambda_*, c_*)$ be a globally optimal solution of the augmented dual problem, $x_*$ be a globally optimal
solution of the problem $(\mathcal{P})$, and suppose that the zero duality gap property holds true. Then with the use of
the fact that $\Theta$ is non-decreasing in $c$ due to $(A4)$ one gets that $\Theta(\lambda_*, c) = f(x_*)$ for 
all $c \ge c_*$. Applying assumption $(A2)$ (recall that $x_*$ is a feasible) one obtains that
$\mathscr{L}(x_*, \lambda_*, c) \le f(x_*)$ for any $c > 0$. Therefore taking into account the definition of $\Theta$
one gets that
\begin{equation} \label{SaddlePoint_RightInequal}
  \mathscr{L}(x_*, \lambda_*, c) \le f(x_*) = \Theta(\lambda_*, c) := \inf_{x \in A} \mathscr{L}(x, \lambda_*, c)
  \quad \forall c \ge c_*,
\end{equation}
which implies that $\mathscr{L}(x_*, \lambda_*, c) = f(x_*)$. Applying $(A2)$ again one obtains that
\begin{equation} \label{SaddlePoint_LeftInequal}
  \mathscr{L}(x_*, \lambda, c) \le f(x_*) = \mathscr{L}(x_*, \lambda_*, c) \quad 
  \forall (\lambda, c) \in \Lambda \times [c_*, + \infty).  
\end{equation}
Combining \eqref{SaddlePoint_LeftInequal} and \eqref{SaddlePoint_RightInequal} one gets that $(x_*, \lambda_*)$ is 
a global saddle point of $\mathscr{L}(x, \lambda, c)$ and $c_* \ge c^*(x_*, \lambda_*)$.	 
\end{proof}

\begin{corollary}
Let assumptions $(A1)$--$(A4)$ be satisfied, and suppose that the zero duality gap property holds true for
$\mathscr{L}(x, \lambda, c)$. Then a global saddle point of $\mathscr{L}(x, \lambda, c)$ exists if and only if there
exists a globally optimal solution of the augmented dual problem.
\end{corollary}

\begin{corollary} \label{Crlr_IndependenceGSPofOptSol}
Let assumptions $(A1)$--$(A4)$ be satisfied, and let $(x_*, \lambda_*)$ be a global saddle point 
of $\mathscr{L}(x, \lambda, c)$. Then for any globally optimal solution $z_*$ of $(\mathcal{P})$ the pair 
$(z_*, \lambda_*)$ is a global saddle point of $\mathscr{L}(x, \lambda, c)$ as well.
\end{corollary}

Let us also obtain a characterization of global saddle points of the augmented Lagrangian $\mathscr{L}(x, \lambda, c)$
in terms of the behaviour of \textit{the optimal value function} (or \textit{the perturbation function}) 
$$
  \beta(p) = \inf\big\{ f(x) \bigm| x \in A, \: G(x) - p \in K \big\}	\quad \forall p \in Y
$$
of the problem $(\mathcal{P})$ (cf.~\cite{RockafellarWets}, Section~11.K, and \cite{ShapiroSun,WangYangYang}).

\begin{proposition} \label{Prp_SaddlePointsAsGenerSubgradients}
Let assumptions $(A1)$--$(A5)$ be satisfied. Then a pair $(x_*, \lambda_*) \in A \times \Lambda$ is a global saddle
point of $\mathscr{L}(x, \lambda, c)$ if and only if $x_*$ is a globally optimal solution of the problem
$(\mathcal{P})$,
and there exists $c_0 > 0$ such that
\begin{equation} \label{LagrMultAsGlobalGenerSubgradient}
  \beta(p) \ge \beta(0) - \Phi(p, \lambda_*, c) \quad \forall p \in Y \quad \forall c \ge c_0.
\end{equation}
\end{proposition}

\begin{proof}
Let $(x_*, \lambda_*)$ be a global saddle point of $\mathscr{L}(x, \lambda, c)$. By 
Proposition~\ref{Prp_SaddlePointImpliesOptSol} the point $x_*$ is a globally optimal solution of the problem
$(\mathcal{P})$, which implies that $f(x_*) = \beta(0)$.

Fix an arbitrary $p \in Y$. If there is no $x \in A$ such that $G(x) - p \in K$, then $\beta(p) = + \infty$, and
inequality \eqref{LagrMultAsGlobalGenerSubgradient} is valid. Otherwise, for any such $x$ one has
$$
  \mathscr{L}(x, \lambda_*, c) \ge \mathscr{L}(x_*, \lambda_*, c) = f(x_*) = \beta(0)
  \quad \forall c > c^*(x_*, \lambda_*)
$$
by virtue of the fact that $(x_*, \lambda_*)$ is a global saddle point of $\mathscr{L}(x, \lambda, c)$ and
Remark~\ref{Rmrk_ALValueAtSP}. Hence and from $(A5)$ it follows that $f(x) \ge \beta(0) - \Phi(p, \lambda_*, c)$
for all $c > c^*(x_*, \lambda_*)$. Taking the infimum over all $x \in A$ for which $G(x) - p \in K$ one obtains that
\eqref{LagrMultAsGlobalGenerSubgradient} holds true.

Let, now, $x_*$ be a globally optimal solution of $(\mathcal{P})$, and let \eqref{LagrMultAsGlobalGenerSubgradient} be
valid. Fix arbitrary $c \ge c_0$ and $x \in A$. Define $p = G(x)$. Then \eqref{LagrMultAsGlobalGenerSubgradient} implies
that 
$$
  f(x) \ge \beta(p) \ge f(x_*) - \Phi(G(x), \lambda_*, c)
$$
or, equivalently, $\mathscr{L}(x, \lambda_*, c) \ge f(x_*)$. Taking the infimum over all $x \in A$ one gets that
$\Theta(\lambda_*, c) \ge f(x_*) = \inf_{x \in \Omega} f(x)$. Hence and from the fact that the weak duality between
$(\mathcal{P})$ and the augmented dual problem holds by $(A2)$, it follows that the zero duality gap property holds true
for $\mathscr{L}(x, \lambda, c)$, and $(\lambda_*, c)$ is a globally optimal solution of the augmented dual problem.
Consequently, applying Proposition~\ref{Prp_SaddlePointsVSOptSolDualProb} one gets that $(x_*, \lambda_*)$ is a global
saddle point.	 
\end{proof}

\begin{remark}
It should be noted that under some additional assumptions inequality \eqref{LagrMultAsGlobalGenerSubgradient} holds
true if and only if there exist $c_0 > 0$ and a neighbourhood of zero, $U \subset Y$, such that
$\beta(p) \ge \beta(0) - \Phi(p, \lambda_*, c)$ for all $p \in U$ and $c \ge c_0$ 
(cf.~\cite{RockafellarWets}, Theorem~11.61, and \cite{ShapiroSun}, Lemma~3.1). However, we do not present an exact
formulation of this result here, and leave it to the interested reader.
\end{remark}

\section{The Localization Principle for Saddle Points}
\label{Section_LocalizationPrinciple_GSP}

In this section, we develop a new general method for proving the existence of global saddle points of augmented
Lagrangian functions. According to this method, one has to verify that there exists a multiplier $\lambda_*$ such that
for any globally optimal solution $x_*$ of the problem $(\mathcal{P})$ the pair $(x_*, \lambda_*)$ is a local saddle
point of $\mathscr{L}(x, \lambda, c)$ in order to prove the existence of a global saddle point. In turn, the existence
of such multiplier $\lambda_*$ can usually be proved with the use of sufficient optimality conditions. Thus, the general
method for proving the existence of global saddle points of augmented Lagrangian functions, that we discuss in this
section, allows one to reduce the problem of the existence of global saddle points to a local analysis of sufficient
optimality conditions. That is why we call this method \textit{the localization principle}.

Let $x_*$ be a globally optimal solution of the problem $(\mathcal{P})$. Denote by $\Lambda_{loc}(x_*)$ 
(resp. $\Lambda(x_*)$) the set of all $\lambda_* \in \Lambda$ for which the pair $(x_*, \lambda_*)$ is a local (resp.
global) saddle point of the augmented Lagrangian function $\mathscr{L}(x, \lambda, c)$. Also, denote by $\Omega_*$ the
set of all globally optimal solutions of the problem $(\mathcal{P})$. Define
$$
  \Lambda_{loc}(\mathcal{P}) = \bigcap_{x_* \in \Omega_*} \Lambda_{loc}(x_*), \quad
  \Lambda(\mathcal{P}) = \bigcap_{x_* \in \Omega_*} \Lambda(x_*)
$$
Corollary~\ref{Crlr_IndependenceGSPofOptSol} implies that, in actuality, $\Lambda(\mathcal{P}) = \Lambda(x_*)$ for 
any $x_* \in \Omega_*$, provided assumptions $(A1)$--$(A4)$ are valid.

Observe that $\Lambda(\mathcal{P}) \subseteq \Lambda_{loc}(\mathcal{P})$ by virtue of the fact that every global saddle
point of $\mathscr{L}(x, \lambda, c)$ is a local one. Therefore for the existence of a global saddle point of 
$\mathscr{L}(x, \lambda, c)$ it is necessary that $\Lambda_{loc}(\mathcal{P}) \ne \emptyset$. The localization principle
states that under some additional assumptions the condition $\Lambda_{loc}(\mathcal{P}) \ne \emptyset$ is also
sufficient for the existence of a global saddle point.

We need an auxiliary result in order to prove the localization principle. 

\begin{lemma} \label{Lemma_MinimizingSeq}
Let $\lambda_* \in \Lambda_{loc}(\mathcal{P})$, $A$ be closed, and assumptions $(A2)$, $(A4)$ and $(A6)$ be valid.
Suppose also that $G$ is continuous on $A$, $\mathscr{L}(\cdot, \lambda_*, c)$ is l.s.c. on $A$ for all $c > 0$, and
$\mathscr{L}(\cdot, \lambda_*, c_0)$ is bounded from below on $A$ for some $c_0 > 0$. Let, finally, sequences 
$\{ c_n \} \subset [c_0, + \infty)$ and $\{ \varepsilon_n \} \subset (0, + \infty)$ be such that 
$c_n \to + \infty$ and $\varepsilon_n \to 0$ as $n \to \infty$. Then any cluster point of a sequence 
$\{ x_n \} \subset A$ such that
\begin{equation} \label{MinimizingSequence}
  \mathscr{L}(x_n, \lambda_*, c_n) \le \inf_{x \in A} \mathscr{L}(x, \lambda_*, c_n) + \varepsilon_n 
  \quad \forall n \in \mathbb{N}.
\end{equation}
is a globally optimal solution of the problem $(\mathcal{P})$.
\end{lemma}

\begin{proof}
Let $x_*$ be a cluster point of a sequence $\{ x_n \} \subset A$ satisfying \eqref{MinimizingSequence}. Replacing, if
necessary, the sequences $\{ c_n \}$, $\{ \varepsilon_n \}$ and $\{ x_n \}$ by their subsequences, one can suppose that
$\{ x_n \}$ converges to $x_*$. Note that $x_* \in A$ due to the fact that $A$ is closed.

Fix an arbitrary $x_0 \in \Omega_*$. Note that $\lambda_* \in \Lambda$, since 
$\lambda_* \in \Lambda_{loc}(\mathcal{P})$, which implies that $\mathscr{L}(x_0, \lambda_*, c) \le f(x_0)$ for all 
$c > 0$ by assumption $(A2)$. Hence and from \eqref{MinimizingSequence} it follows that
\begin{equation} \label{BoundednesAlongTheNet}
  \mathscr{L}(x_n, \lambda_*, c_n) \le f(x_0) + \varepsilon_n \quad \forall n \in \mathbb{N},
\end{equation}
which yields that 
\begin{equation} \label{ConstrBoundAlongTheNet}
  f(x_n) < + \infty, \quad \Phi( G(x_n), \lambda_*, c_n ) < + \infty \quad \forall n \in \mathbb{N}.
\end{equation}
Choose an arbitrary $c > 0$. Since $\lim c_n = + \infty$, there exists $n_0 \in \mathbb{N}$ such that $c_n \ge c$
for all $n \ge n_0$. Hence with the use of $(A4)$ one obtains that 
$\mathscr{L}(x_n, \lambda_*, c) \le f(x_0) + \varepsilon_n$ for all $n \ge n_0$. Passing to the limit inferior as 
$n \to \infty$ one gets that $\mathscr{L}(x_*, \lambda_*, c) \le f(x_0)$ for all $c > 0$. Let us verify that $x_*$ is a
feasible point of the problem $(\mathcal{P})$. Then applying $(A2)$ one obtains that $f(x_*) \le f(x_0)$, which implies
that $x_* \in \Omega_*$ due to the definition of $x_0$.

Arguing by reductio ad absurdum, suppose that $x_*$ is infeasible, i.e. that $y_* := G(x_*) \notin K$. Then by $(A6)$
there exists $r > 0$ such that
\begin{equation} \label{MinimSeqAssumpA8}
  \lim_{c \to \infty} \inf\{ \Phi(y, \lambda_*, c) - \Phi(y, \lambda_*, c_0) \mid 
  y \in B(y_*, r) \colon \Phi(y, \lambda_*, c_0) < + \infty \} = + \infty.
\end{equation}
Since $G$ is continuous and $x_n \to x_*$ as $n \to \infty$, there exists $k \in \mathbb{N}$ such that 
$G(x_n) \in B(y_*, r)$ for all $n \ge k$. Consequently, taking into account \eqref{ConstrBoundAlongTheNet} and
assumption $(A4)$ one gets that for all $n \ge k$ the following inequalities hold true
\begin{multline*}
  \mathscr{L}(x_n, \lambda_*, c_n) = \mathscr{L}(x_n, \lambda_*, c_0) +
  \Phi(G(x_n), \lambda_*, c_n) - \Phi(G(x_n), \lambda_*, c_0) \\
  \ge \gamma + \inf\{ \Phi(y, \lambda_*, c_n) - \Phi(y, \lambda_*, c_0) \mid 
  y \in B(y_*, r) \colon \Phi(y, \lambda_*, c_0) < + \infty \},
\end{multline*}
where $\gamma = \inf_{x \in A} \mathscr{L}(x, \lambda_*, c_0) > - \infty$. Hence applying \eqref{MinimSeqAssumpA8}
one obtains that $\mathscr{L}(x_n, \lambda_*, c_n) \to + \infty$ as $n \to \infty$, which contradicts
\eqref{BoundednesAlongTheNet}. Thus, $x_*$ is a feasible point of the problem $(\mathcal{P})$, which completes 
the proof.	 
\end{proof}

Now, we are ready to derive the localization principle for global saddle points of the augmented Lagrangian function
$\mathscr{L}(x, \lambda, c)$. Denote by $f_* = \inf_{x \in \Omega} f(x)$ the optimal value of the problem
$(\mathcal{P})$.

\begin{theorem}[Localization Principle] \label{Thrm_LocalizationPrinciple}
Let $A$ be closed, $G$ be continuous on $A$, and $\mathscr{L}(\cdot, \lambda, c)$ be l.s.c. on $A$ for all 
$\lambda \in \Lambda$ and $c > 0$. Suppose also that assumptions $(A1)$--$(A4)$ and $(A6)$ are satisfied. Then a global
saddle point of $\mathscr{L}(x, \lambda, c)$ exists if and only if $\Lambda_{loc}(\mathcal{P}) \ne \emptyset$, and there
exist $\lambda_* \in \Lambda_{loc}(\mathcal{P})$ and $c_0 > 0$ such that the set
$$
  S(\lambda_*, c_0) := \big\{ x \in A \bigm| \mathscr{L}(x, \lambda_*, c_0) < f_* \big\}
$$
is either bounded or empty. Furthermore, $\Lambda(\mathcal{P})$ consists of all those 
$\lambda_* \in \Lambda_{loc}(\mathcal{P})$ for which there exists $c_0 > 0$ such that the set $S(\lambda_*, c_0)$ is
either bounded or empty.
\end{theorem}

\begin{proof}
Suppose that there exists a global saddle point $(x_*, \lambda_*)$ of $\mathscr{L}(x, \lambda, c)$.
By Corollary~\ref{Crlr_IndependenceGSPofOptSol}, for any globally optimal solution $x_0$ of $(\mathcal{P})$ the pair
$(x_0, \lambda_*)$ is a global saddle point of $\mathscr{L}(x, \lambda, c)$, which yields that 
$\lambda_* \in \Lambda_{loc}(\mathcal{P})$. Furthermore, from the definition of global saddle point and
Remark~\ref{Rmrk_ALValueAtSP} it follows that for any $c > c^*(x_*, \lambda_*)$ the set $S(\lambda_*, c)$ is empty. 

Suppose, now, that $\lambda_* \in \Lambda_{loc}(\mathcal{P})$, and there exists $c_0 > 0$ such that the set
$S(\lambda_*, c_0)$ is either bounded or empty. From $(A4)$ it follows that 
\begin{equation} \label{ALFSublevelSetsIncl}
  S(\lambda_*, c) \subseteq S(\lambda_*, c_0) \quad \forall c \ge c_0,
\end{equation}
which implies, in particular, that $S(\lambda_*, c)$ is bounded (or empty) for any $c \ge c_0$. 

If there exists $\overline{c} \ge c_0$ such that $S(\lambda_*, \overline{c}) = \emptyset$, then
$\inf_{x \in A} \mathscr{L}(x, \lambda_*, c) \ge f_*$ for all $c \ge \overline{c}$ due to $(A4)$. Hence applying 
the definition of local saddle point and Remark~\ref{Rmrk_ALValueAtSP} one gets that for any globally optimal solution
$x_*$ of $(\mathcal{P})$, and for all $c > \max\{ \overline{c}, c_{loc}^*(x_*, \lambda_*) \}$ the following inequalities
hold true
$$
  \sup_{\lambda \in \Lambda} \mathscr{L}(x_*, \lambda, c) \le \mathscr{L}(x_*, \lambda_*, c) = f(x_*) = f_* \le 
  \inf_{x \in A} \mathscr{L}(x, \lambda_*, c),
$$
i.e. $(x_*, \lambda_*)$ is a global saddle point of $\mathscr{L}(x, \lambda, c)$.

Thus, it remains to consider the case when $S(\lambda_*, c) \ne \emptyset$ for all $c \ge c_0$. As it was noted above
(see inclusion~\eqref{ALFSublevelSetsIncl}), the set $S(\lambda_*, c)$ is bounded for all $c \ge c_0$. Therefore, taking
into account the fact that $\mathscr{L}(\cdot, \lambda_*, c)$ is l.s.c. on $A$, one obtains that 
$\mathscr{L}(\cdot, \lambda_*, c)$ attains a global minimum on $A$ at a point $x(c)$ for all $c \ge c_0$. Choose an
increasing unbounded sequence $\{ c_n \} \subset [c_0, + \infty)$, and denote $x_n = x(c_n)$. Inclusion
\eqref{ALFSublevelSetsIncl} implies that the sequence $\{ x_n \}$ is contained in the bounded set $S(\lambda_*, c_0)$.
Therefore there exists a subsequence $\{ x_{n_k} \}$ converging to a point $x_*$ that belongs to the set $A$ due to the
fact that $A$ is closed. 

By the definition of $x_n$ one has 
$\mathscr{L}( x_n, \lambda_*, c_n ) = \inf_{x \in A} \mathscr{L}(x, \lambda_*, c_n)$. Consequently, applying
Lemma~\ref{Lemma_MinimizingSeq} one obtains that $x_*$ is a globally optimal solution of the problem $(\mathcal{P})$.
Hence taking into account the fact that $\lambda_* \in \Lambda_{loc}(\mathcal{P})$ and applying
Remark~\ref{Rmrk_ALValueAtSP} one gets that there exist $\overline{c} > 0$ and a neighbourhood $U$ of $x_*$ such that
\begin{equation} \label{ClusterPointIsLocSP}
  \sup_{\lambda \in \Lambda} \mathscr{L}(x_*, \lambda, c) \le \mathscr{L}(x_*, \lambda_*, c) = f(x_*) \le 
  \inf_{x \in U \cap A} \mathscr{L}(x, \lambda_*, c) \quad \forall c \ge \overline{c}.
\end{equation}
Since $\{ x_{n_k} \}$ converges to $x_*$, and $\{ c_n \}$ is an increasing unbounded sequence, there exists 
$k_0 \in \mathbb{N}$ such that $x_{n_k} \in U$ and $c_{n_k} \ge \overline{c}$ for all $k \ge k_0$. Hence applying
\eqref{ClusterPointIsLocSP} and the fact that $x_* \in \Omega_*$ one obtains that 
$\mathscr{L}( x_{n_k}, \lambda_*, c_{n_k} ) \ge f(x_*) = f_*$ for all $k \ge k_0$, which contradicts our assumption that
$S(\lambda_*, c) \ne \emptyset$ for all $c \ge c_0$ due to the fact that $x_n$ is a global minimizer of 
$\mathscr{L}(\cdot, \lambda_*, c_n)$ on the set $A$.	 
\end{proof}

\begin{corollary}[Localization Principle]
Let all assumptions of the theorem above be satisfied, and suppose that the set $A$ is compact. 
Then $\Lambda(\mathcal{P}) = \Lambda_{loc}(\mathcal{P})$.
\end{corollary}

\begin{remark} \label{Remark_LocPrinciple_Unification}
The localization principle for global saddle points of augmented Lagrangian functions for cone constrained programming
problems (Theorem~\ref{Thrm_LocalizationPrinciple}) unifies and sharpens most of the known results on the existence of
global saddle points of augmented Lagrangians for mathematical programming problems (\cite{LiuYang}, Thm.~3.3;
\cite{LiuTangYang2009}, Thm.~3; \cite{LuoMastroeniWu2010}, Thm.~4.1; \cite{ZhouXiuWang}, Thm.~3.3; 
\cite{WuLuo2012b}, Thm.~2; \cite{WangZhouXu2009}, Thms.~3.3 and 3.4; \cite{SunLi}, Thms.~3.1, 3.3 and 3.4;
\cite{WangLiuQu}, part~3 of Cor.~4.1, 4.3 and 4.4, part 2 of Cor.~4.2 and 4.5, and parts 3 and 4 of
Cor.~4.6), nonlinear second order cone programming problems (\cite{ZhouChen2015}, Thm.~3.1) and nonlinear
semidefinite programming (\cite{WuLuoYang2014}, Thm.~4; \cite{LuoWuLiu2015}, Thm.~4.4). Furthermore, the localization
principle provides first simple \textit{necessary and sufficient} conditions for the existence of global saddle points.
\end{remark}

Let us note that if, instead of assumption $(A6)$, stronger assumption $(A7)$ holds true, then instead of 
$S(\lambda_*, c)$ one can use a smaller set.

\begin{proposition}
Let assumptions $(A4)$ and $(A7)$ be valid, and let the function $\mathscr{L}(\cdot, \lambda_*, c_0)$ be bounded from
below on $A$ for some $\lambda_* \in \Lambda$ and $c_0 > 0$. Then the set $S(\lambda_*, c_1)$ is bounded for some 
$c_1 > 0$ iff there exist $c_2 > 0$ and $\alpha > 0$ such that the set 
$Q(\lambda_*, c_2, \alpha) := \{ x \in A \mid \mathscr{L}(x, \lambda_*, c_2) < f_*, \: d(G(x), K) < \alpha \}$
is bounded.
\end{proposition}

\begin{proof}
The validity of the ``only if'' part of the proposition follows directly from the obvious inclusion
$Q(\lambda_*, c, \alpha) \subseteq S(\lambda_*, c)$ that holds true for all $c > 0$ and $\alpha > 0$. Therefore, let us
prove the ``if'' part. Namely, let us show that for any $\alpha > 0$ there exists $\overline{c} > 0$ such that
$S(\lambda_*, c) \subseteq Q(\lambda_*, c, \alpha)$ for all $c \ge \overline{c}$. Then taking into account the inclusion
$Q(\lambda_*, c', \alpha) \subseteq Q(\lambda_*, c'', \alpha)$, that holds true for all $c' \ge c''$ due to $(A4)$, one
obtains the required result.

Fix $\alpha > 0$. Let $x \in A$ be such that $d(G(x), K) \ge \alpha$ and $\Phi(G(x), \lambda_*, c_0) < + \infty$. Then
for all $c \ge c_0$ one has
\begin{multline*}
  \mathscr{L}(x, \lambda_*, c) = 
  \mathscr{L}(x, \lambda_*, c_0) + \Phi(G(x), \lambda_*, c) - \Phi(G(x), \lambda_*, c_0) \\
  \ge \gamma + \inf\big\{ \Phi(y, \lambda_*, c) - \Phi(y, \lambda_*, c_0) \bigm|
  y \in Y \colon d(y, K) \ge \alpha, \: \Phi(y, \lambda_*, c_0) < + \infty \big\},
\end{multline*}
where $\gamma = \inf_{x \in A} \mathscr{L}(x, \lambda_*, c_0) > - \infty$. Applying $(A7)$ one gets that there
exists $\overline{c} \ge c_0$ (that depends only on $\lambda_*$, $c_0$ and $\alpha > 0$) such 
that $\mathscr{L}(x, \lambda_*, c) \ge f_*$ for all $c \ge \overline{c}$. In other words, 
$S(\lambda_*, c) = Q(\lambda_*, c, \alpha)$ for all $c \ge \overline{c}$.	 
\end{proof}

\begin{remark}
Let the following assumption be valid:
\begin{itemize}
\item[(A13)]{$\forall \lambda \in \Lambda$ $\forall \varepsilon > 0$ $\forall c_0 > 0$ $\exists \, c \ge c_0$ 
$\exists \, \alpha > 0$ such that $\Phi(G(x), \lambda, c) > - \varepsilon$ for all 
$x \in A \colon \dist(G(x), K) < \alpha$;
}
\end{itemize}
Then one can easily verify that the set $Q(\lambda, c, \alpha)$ from the proposition above is bounded for some $c > 0$
and $\alpha > 0$, provided there exists $\gamma > 0$ such that the set
\begin{equation} \label{ProblemSublevelSet}
  \big\{ x \in A \bigm| f(x) < f_* + \gamma, \: d(G(x), K) < \gamma \big\}
\end{equation}
is bounded. Note that the assumption on the boundedness of set \eqref{ProblemSublevelSet} is utilized in most of 
the known results on existence of global saddle points of augmented Lagrangian functions (cf., e.g., \cite{SunLi},
Thms.~3.1--3.4; \cite{LiuYang}, Thm.~3.3; \cite{WuLuo2012b}, Thm.~2; \cite{ZhouXiuWang}, Thm.~3.3;
\cite{ZhouZhouYang2014}, Thms.~3.1 and 3.2; \cite{ZhouChen2015}, Thm.~3.1, etc.). Observe also that
assumption $(A13)$ is satisfied in Examples~\ref{Example_RockafellarWetsAL} and
\ref{Example_RockafellarWetsAL_SemiInfProg}, provided $\sigma(y) \ge \omega(\| y \|)$ for some non-negative continuous
function $\omega$ such that $\liminf_{t \to +\infty} \omega(t) / t > 0$. This assumptions is always valid in 
Examples~\ref{Example_EssentiallyQuadraticAL}--\ref{Example_Mangasarian}, 
\ref{Example_pthPowerAugmLagr}, \ref{Example_RockWetsAL_SOC} and \ref{Example_RockWetsAL_SemiDefProg}. Finally, $(A13)$
is satisfied in Examples~\ref{Example_ExpPenFunc}--\ref{Example_ModBarrierFunc}, \ref{Example_HeWuMengLagrangian},
\ref{Example_NonlinearRescale_SOC}, \ref{Example_NonlinearRescale_SemiDefProg},
\ref{Example_NonlinearRescalePenalized_SemiDefProg} and \ref{Example_NonlinearRescale_SemiInfProg} if either the
function $\phi$ (or $\psi$) is bounded from below or the sets
$g_i(A) \cap K_i$, $i \in I$, are bounded, where $K_i$ is either $\mathbb{R}_-$ or the second order cone or the cone of
negative semidefinite matrices or the cone of non-positive continuous functions, depending on the context. Let us
finally note that the assumption on the boundedness of the sets $g_i(A) \cap K_i$ was utilized in various paper on
augmented Lagrangian functions (see, e.g., \cite{WangLi2009,LiuYang,WangLiuQu})
\end{remark}

\begin{remark} \label{Remark_BarrierTerms}
Note that if assumption $(A7)$ is valid, and the set $Q(\lambda_*, c, \alpha)$ is bounded for some $c > 0$ and 
$\alpha > 0$, but the function $\mathscr{L}(\cdot, \lambda_*, c)$ is not bounded from below for any $c > 0$, then one
can redefine the function $\Phi(y, \lambda, c)$ in order to guarantee the existence of global saddle points. Namely,
define
$$
  \widehat{\Phi}(y, \lambda, c) = \begin{cases}
  \big( \alpha - \dist(y, K)^{\varkappa} \big) \cdot 
  \Phi\left( \dfrac{1}{\alpha - \dist(y, K)^{\varkappa}} y, \lambda, c \right), 
  & \text{if } d(y, K)^{\varkappa} < \alpha, \\
  + \infty, & \text{otherwise}
  \end{cases}
$$
for some $\alpha > 0$ and $\varkappa > 0$. One can verify that such transformation not only guarantees the boundedness
from below of the function $\mathscr{L}(\cdot, \lambda, c)$, but also preserves local and global saddle points, 
the boundedness of the set $Q(\lambda_*, c, \alpha)$, and general properties of the augmented Lagrangian 
$\mathscr{L}(x, \lambda, c)$ (such as continuity, differentiability if $\varkappa \ge 2$, etc.) in all particular
examples presented above, except for Example~\ref{Example_NonlinearRescale_SOC}. Finally, it should be noted that there
is no need for such transformation of the function $\Phi(y, \lambda, c)$ in Example~\ref{Example_ModBarrierFunc}, and
Examples~\ref{Example_NonlinearRescale_SOC}, \ref{Example_NonlinearRescale_SemiDefProg} and
\ref{Example_NonlinearRescale_SemiInfProg} in the case when $\varepsilon_0 < + \infty$, since the augmented Lagrangian
functions from these examples are bounded from below iff the set $Q(\lambda_*, c, \alpha)$ is bounded for some $c > 0$
and $\alpha > 0$.
\end{remark}

The following proposition contains a simple reformulation of the boundedness condition on the set $S(\lambda_*, c)$
that will allow us to give an illuminating interpretation of the localization principle.

\begin{proposition} \label{Prp_GSP_LP_Nondegeneracy}
Let $\lambda_* \in \Lambda_{loc}(\mathcal{P})$, $A$ be closed, $\mathscr{L}(\cdot, \lambda_*, c)$ be l.s.c. on $A$ for
all $c > 0$, and $G$ be continuous on $A$. Suppose also that assumptions $(A2)$, $(A4)$ and $(A6)$ are satisfied. Then
for the existence of $c_0 > 0$ such that the set $S(\lambda_*, c_0)$ is either bounded or empty it is necessary and
sufficient that there exist $c_0 > 0$ and $R > 0$ such that for any $c \ge c_0$ there exists 
$x(c) \in \argmin_{x \in A} \mathscr{L}(x, \lambda_*, c)$ with  $\| x(c) \| \le R$.
\end{proposition}

\begin{proof}
Suppose that exists $c_0 > 0$ such that the set $S(\lambda_*, c_0)$ is either bounded or empty. Note that $(A4)$ implies
that \eqref{ALFSublevelSetsIncl} holds true. If there exists $c_1 \ge c_0$ such that $S(\lambda_*, c_1) = \emptyset$,
then taking into account $(A2)$ and $(A4)$ one gets that $x_* \in \argmin_{x \in A} \mathscr{L}(x, \lambda_*, c)$ for
all $c \ge c_1$, where $x_*$ is a globally optimal solution of $(\mathcal{P})$. On the other hand, if $S(\lambda_*, c)
\ne \emptyset$ for all $c \ge c_0$, then taking into account \eqref{ALFSublevelSetsIncl} and the fact that
$\mathscr{L}(\cdot, \lambda_*, c)$ is l.s.c. on $A$ one obtains that for any $c \ge c_0$ there exists 
$x(c) \in \argmin_{x \in A} \mathscr{L}(x, \lambda_*, c)$ such that $x(c) \in S(\lambda_*, c_0)$. Since 
$S(\lambda_*, c_0)$ is bounded, there exists $R > 0$ such that $\| x(c) \| \le R$ for all $c \ge c_0$.

Suppose, now, that there exist $c_0 > 0$ and $R > 0$ such that for any $c \ge c_0$ there exists 
$x(c) \in \argmin_{x \in A} \mathscr{L}(x, \lambda_*, c)$ with $\| x(c) \| \le R$.  Choose an increasing unbounded
sequence $\{ c_n \} \subset [c_0, + \infty)$, and define $x_n = x(c_n)$. Without loss of generality one can suppose
that the sequence $\{ x_n \}$ converges to a point $x_*$. By Lemma~\ref{Lemma_MinimizingSeq}, the point $x_*$ is a
globally optimal solution of $(\mathcal{P})$.

Taking into account Remark~\ref{Rmrk_ALValueAtSP} and the fact that $\lambda_* \in \Lambda_{loc}(\mathcal{P})$ one gets
that there exist $r > 0$ and a neighbourhood $U$ of $x_*$ such that 
$$
  \sup_{\lambda \in \Lambda} \mathscr{L}(x_*, \lambda, c) \le \mathscr{L}(x_*, \lambda_*, c) = f(x_*) \le 
  \inf_{x \in U \cap A} \mathscr{L}(x, \lambda_*, c) \quad \forall c \ge r.
$$
Since $x_n \to x_*$ and $c_n \to + \infty$ as $n \to \infty$, there exists $n \in \mathbb{N}$ for which $x_n \in U$
and $c_n \ge r$. For any such $n$ one has $\mathscr{L}(x_n, \lambda_*, c_n) \ge f(x_*) = f_*$, which implies that
$S(\lambda_*, c_n) = \emptyset$ due to the definition of $x_n$.	 
\end{proof}

\begin{remark}
With the use of the proposition above we can give the following interpretation of the localization principle.
Roughly speaking, according to the localization principle a global saddle point of the augmented Lagrangian function
$\mathscr{L}(x, \lambda, c)$ exists if and only if there exists a multiplier $\lambda_*$ such that for any globally
optimal solution $x_*$ of $(\mathcal{P})$ the pair $(x_*, \lambda_*)$ is a local saddle point of 
$\mathscr{L}(x, \lambda, c)$, the function $\mathscr{L}(\cdot, \lambda_*, c)$ attains a global minimum on $A$ for any
sufficiently large $c > 0$, and global minimizers of $\mathscr{L}(\cdot, \lambda_*, c)$ on $A$ do not escape to infinity
as $c \to + \infty$.
\end{remark}

Note that, in particular, the exponential penalty function (Example~\ref{Example_ExpPenFunc}) and the p-th power
augmented Lagrangian (Example~\ref{Example_pthPowerAugmLagr}) do not satisfy assumption $(A6)$. Therefore one cannot
apply the localization principle in order to obtain necessary and sufficient conditions for the existence of global
saddle points of these augmented Lagrangian functions. Let us show that global saddle points of these augmented
Lagrangian functions exist only under a rather restrictive assumption on the problem $(\mathcal{P})$. 

\begin{example} \label{Example_SeparationCondition}
Consider the following nonlinear programming problem
\begin{equation} \label{ExProblem_ExpPenFunc}
  \min f(x) \quad \text{subject to} \quad g_1(x) \le 0, \quad g_2(x) \le 0.
\end{equation}
In this case $K = \mathbb{R}_-^2$. Let $\Lambda = K^* = \mathbb{R}^2_+$, and let
$$
  \mathscr{L}(x, \lambda, c) = f(x) + \frac{\lambda_1}{c} \big(e^{c g_1(x)} - 1 \big) +
  \frac{\lambda_2}{c} \big(e^{c g_2(x)} - 1 \big)
$$
be the exponential penalty function for this problem. Suppose that a pair $(x_*, \lambda_*)$ with 
$\lambda_* \in \mathbb{R}^2_+$ is a global saddle point of $\mathscr{L}(x, \lambda, c)$, i.e.
\begin{equation} \label{GSPofExpPenFuncDef}
  \sup_{\lambda \in \mathbb{R}^2_+} \mathscr{L}(x_*, \lambda, c) \le \mathscr{L}(x_*, \lambda_*, c)
  \le \inf_{x \in X} \mathscr{L}(x, \lambda_*, c)
\end{equation}
for any sufficiently large $c > 0$. Suppose, in addition, that $g_2(x_*) < 0$. Then $(\lambda_*)_2 = 0$ by
Proposition~\ref{Prp_ComplementarySlackAtSP}, which implies that $\mathscr{L}(x, \lambda_*, c)$ does not depend on
$g_2(x)$. Consequently, taking the supremum in \eqref{GSPofExpPenFuncDef} over only those $\lambda \in \mathbb{R}^2_+$
for which $\lambda_2 = 0$ one gets that $(x_*, (\lambda_*)_1)$ is a global saddle point of the reduced augmented
Lagrangian function $\mathscr{L}_0(x, \lambda_1, c) = \mathscr{L}(x, (\lambda_1, 0), c)$. Hence applying
Proposition~\ref{Prp_SaddlePointImpliesOptSol} one gets that $x_*$ is a globally optimal solution of 
problem~\eqref{ExProblem_ExpPenFunc}, and of the problem of minimizing $f(x)$ subject to $g_1(x) \le 0$. 
\end{example}

Thus, if a global saddle point of the exponential penalty function exists, then the removal of inequality constraints
that are inactive at a globally optimal solution must not change the global optimality of the given solution. After
inactive inequality constraints have been removed, the exponential penalty function will satisfy assumptions
$(A6)$ and $(A7)$ for a given multiplier $\lambda_* \in \Lambda_{loc}(\mathcal{P})$, provided all components of
$\lambda_*$ are positive (or, equivalently, if $\lambda_*$ satisfies the strict complementarity condition at
all globally optimal solutions of the problem $(\mathcal{P})$). Then arguing in the same way as in the proof of the
localization principle one can obtain simple necessary and sufficient conditions for the existence of a global saddle
point of the exponential penalty function.

Namely, one can verify that the following result holds true

\begin{proposition} \label{Prp_ExpPenFuncGSP}
Let $A$ be closed, and $\mathscr{L}(x, \lambda, c)$ be the exponential penalty function for the problem
$$
  \min f(x) \quad \mbox{subject to} \quad g_i(x) \le 0, \quad i \in I, \quad x \in A.	\eqno{(\mathcal{P}_1)}
$$
Let also $f$ be l.s.c. on $A$, and $g_i$, $i \in I$, be continuous on $A$. Then a global saddle point of 
$\mathscr{L}(x, \lambda, c)$ exists if and only if there exist $\lambda_* \in \Lambda_{loc}(\mathcal{P}_1)$ and 
$c_0 > 0$ such that the set $S(\lambda_*, c_0)$ is either bounded or empty, every globally optimal solutions of the
problem $(\mathcal{P}_1)$ is a globally optimal solution of the problem
$$
  \min f(x) \quad \text{subject to} \quad g_i(x) \le 0, \quad i \in I(\lambda_*), \quad x \in A,
  \eqno{(\mathcal{P}_2)}
$$
and $\lambda_* \in \Lambda_{loc}(\mathcal{P}_2)$, where $I(\lambda_*) = \{ i \in I \mid (\lambda_*)_i > 0 \}$.
\end{proposition}

\begin{remark} \label{Remark_SeparationCondition}
{(i)~It should be noted that the example and the proposition above were inspired by ``the separation condition'' (3.11)
from \cite{SunLi}, where it was used in order to obtain sufficient conditions for the existence of a global saddle point
of the exponential penalty function (\cite{SunLi}, Theorem~3.2). Let us note that Proposition~\ref{Prp_ExpPenFuncGSP}
significantly sharpens Theorem~3.2 from \cite{SunLi}, since we do not assume that $A$ is compact or that a globally
optimal solution of the problem $(\mathcal{P}_1)$ is unique, and obtain \textit{necessary and sufficient}
conditions for the existence of a global saddle point, in contrast to only sufficient conditions in \cite{SunLi}.
}

\noindent{(ii)~Observe that every globally optimal solution of the problem $(\mathcal{P}_1)$ is a globally optimal
solution of the problem $(\mathcal{P}_2)$ and $\Lambda_{loc}(\mathcal{P}_1) = \Lambda_{loc}(\mathcal{P}_2)$ in the case
when the problem $(\mathcal{P}_1)$ is convex. However, note also that in the convex case one obviously has 
$\Lambda(\mathcal{P}) = \Lambda_{loc}(\mathcal{P}) = \Lambda_{loc}(x_*)$ for any $x_* \in \Omega_*$, and the
localization principle holds trivially.
}

\noindent{(iii)~The interested reader can extend Proposition~\ref{Prp_ExpPenFuncGSP} to case of other
augmented Lagrangians, such as the ones from Examples~\ref{Example_NonlinearRescale_SOC},
\ref{Example_NonlinearRescale_SemiDefProg} and \ref{Example_NonlinearRescale_SemiInfProg} 
with $\varepsilon_0 = + \infty$.
}
\end{remark}

\section{Applications of the Localization Principle for Global Saddle Points}
\label{Section_ApplLocPrinciple_GSP}

The main goal of this section is to demonstrate that the localization principle allows one to easily prove the existence
of global saddle points of augmented Lagrangian functions with the use of sufficient optimality conditions. Below, we
suppose that $X = \mathbb{R}^d$. For the sake of simplicity, in this section we also suppose that the set $A$ is convex.

At first, we study the existence of local saddle points. Our aim is to establish a connection between KKT-pairs of the
problem $(\mathcal{P})$ and local saddle points of the augmented Lagrangian function $\mathscr{L}(x, \lambda, c)$.

Let $(\mathcal{P})$ be a cone constrained minimax problem, i.e. let $f(x)$ have the form $f = \max_{k \in M} f_k$,
where $f_k \colon \mathbb{R}^d \to \mathbb{R} \cup \{ + \infty \}$ are given functions, and $M = \{ 1, \ldots, m \}$.
Denote by $L(x, \lambda) = f(x) + \langle \lambda, G(x) \rangle$ the standard Lagrangian function for the problem
$(\mathcal{P})$. For any $\alpha = (\alpha_1, \ldots, \alpha_m) \in \mathbb{R}^m_+$ denote
$L_0(x, \lambda, \alpha) = \sum_{k = 1}^m \alpha_k f_k(x) + \langle \lambda, G(x) \rangle$. Finally, 
denote $M(x) = \{ k \in M \mid f_k(x) = f(x) \}$. 

Suppose that the functions $f_k$, $k \in M$, and $G$ are twice differentiable at a point $x_* \in A$. Recall that a pair
$(x_*, \lambda_*)$ is called a \textit{KKT-pair} of the problem $(\mathcal{P})$ if $G(x_*) \in K$, $\lambda_* \in K^*$,
$\langle \lambda_*,  G(x_*) \rangle = 0$ and
$$
  \big[ L(\cdot, \lambda_*) \big]'(x_*, h) \ge 0 \quad \forall h \in T_A(x_*).
$$
Here $[ L(\cdot, \lambda_*) ]'(x_*, h)$ is the directional derivative of the function $L(\cdot, \lambda_*)$ at the point
$x_*$ in a direction $h$, and $T_A(x_*)$ is the contingent cone to $A$ at $x_*$. Any $\lambda_* \in K^*$ such that
$(x_*, \lambda_*)$ is a KKT-pair of $(\mathcal{P})$ is called \textit{a Lagrange multiplier} of $(\mathcal{P})$ at
$x_*$.

One can easily verify that a pair $(x_*, \lambda_*)$ is a KKT-pair of $(\mathcal{P})$ iff $G(x_*) \in K$, 
$\lambda_* \in K^*$, $\langle \lambda_*,  G(x_*) \rangle = 0$, and there exists a vector $\alpha \in \mathbb{R}^m_+$
(that is sometimes called \textit{a Danskin-Demyanov multiplier}) such that
$\alpha_k = 0$ for any $k \notin M(x_*)$, and
\begin{equation} \label{KKTdef}
  \sum_{k = 1}^m \alpha_k = 1, \quad
  \big\langle D_x L_0(x_*, \lambda_*, \alpha), h \big\rangle \ge 0 \quad \forall h \in T_A(x_*).
\end{equation}
The set of all such $\alpha$ is denoted by $\alpha(x_*, \lambda_*)$.

Let us also recall the second order necessary optimality conditions (cf.~\cite{BonnansShapiro}, Theorem~3.45 and
Proposition~3.46). One says that a KKT-pair $(x_*, \lambda_*)$ satisfies \textit{the second order necessary optimality
condition} if
\begin{multline} \label{KKTSecondOrderDef}
  \sup_{\alpha \in \alpha(x_*, \lambda_*)} 
  \big\langle h, D^2_{xx} L_0(x_*, \lambda_*, \alpha) h \big\rangle - \sigma(\lambda_*, \mathcal{T}(h)) \ge 0 \\
  \forall h \in C(x_*, \lambda_*) \colon [L(\cdot, \lambda_*)]'(x_*, h) = 0.
\end{multline}
Here $\sigma(\lambda_*, \mathcal{T}(h)) = \sup_{y \in \mathcal{T}(h)} \langle \lambda_*, y \rangle$,
$\mathcal{T}(h) = T_K^2(G(x_*), DG(x_*) h)$ is \textit{the outer second order tangent set} to the set $K$ at the point
$G(x_*)$ in the direction $DG(x_*) h$ (see~\cite{BonnansShapiro}, Definition~3.28), and
$$
  C(x_*, \lambda_*) = \Big\{ h \in T_A(x_*) \Bigm| D G(x_*) h \in T_K \big( G(x_*) \big), \:
  \langle \lambda_*, D G(x_*) h \rangle = 0 \Big\}.
$$
is \textit{the critical cone} at the point $x_*$. Note that if the cone $K$ is polyhedral, then 
$\sigma(\lambda_*, \mathcal{T}(h)) = 0$ for any $h \in C(x_*, \lambda_*)$, and optimality condition
\eqref{KKTSecondOrderDef} is reduced to the standard optimality condition. On the other hand, in the general case
$\sigma(\lambda_*, \mathcal{T}(h)) \le 0$ for all $h \in C(x_*, \lambda_*)$, which means that optimality condition
\eqref{KKTSecondOrderDef} is weaker than the standard one. It should be mentioned that the term 
$\sigma(\lambda_*, \mathcal{T}(h))$, in a sense, represents the contribution of the curvature of the cone $K$ at the
point $G(x_*)$.

In order to study a connection between local saddle points and KKT-pairs, we need to introduce an additional
assumption on differentiability properties of the function $\Phi$. 

\begin{definition} \label{Def_SeconOrdeExpansion}
Let assumption $(A11)$ hold true, and $G$ be twice Fr\'echet differentiable at a feasible point $x_* \in A$. Let also
$\lambda_* \in K^*$ be such that $\langle \lambda_*, G(x_*) \rangle = 0$. One says that the function 
$\Phi(G(x), \lambda, c)$ admits \textit{the second order expansion} in $x$ at $(x_*, \lambda_*)$, if for all $c > 0$
there exists a positively homogeneous of degree $2$ function $\omega_c \colon X \to \mathbb{R}$ such that
for any $h$ in a neighbourhood of zero and $c > 0$ one has
\begin{multline*}
  \Phi(G(x_* + h), \lambda_*, c) - \Phi(G(x_*), \lambda_*, c) = \langle \mu_*, D G(x_*) h \rangle \\
  + \frac{1}{2} \langle \mu_*, D^2 G(x_*)(h, h) \rangle + \frac{1}{2} \omega_c(h) 
  + o(\|h\|^2),
\end{multline*}
where $o(\|h\|^2) / \| h \|^2 \to 0$ as $h \to 0$, and 
\begin{enumerate}
\item{$\mu_* = \Phi_0(\lambda_*)$ and $\Phi_0(\lambda) = D_y \Phi(y, \lambda, c)$ (see $(A11)$);
}

\item{$\omega_c(h) \to - \sigma(\mu_*, \mathcal{T}(h))$ as $c \to +\infty$ for any $h \in C(x_*, \mu_*)$;
}

\item{if $\limsup_{[h, c] \to [h_*, + \infty]} \omega_c(h)$ is finite for some $h_* \in T_A(x_*)$, then
$h_* \in C(x_*, \mu_*)$, and $\limsup_{[h, c] \to [h_*, + \infty]} \omega_c(h) \ge - \sigma(\mu_*, \mathcal{T}(h_*))$.
}
\end{enumerate}
\end{definition}

Note that the function $\Phi(G(\cdot), \lambda_*, c)$ need not be twice differentiable at $x_*$ to admit
the second order expansion in $x$ at $(x_*, \lambda_*)$.

\begin{remark} \label{Rmrk_2OrderExpansInX}
Let us discuss when the assumption that the function $\Phi(G(x), \lambda, c)$ admits the second order expansion in
$x$ at $(x_*, \lambda_*)$ is satisfied for the previously analysed examples. This assumption is satisfied in
Example~\ref{Example_RockafellarWetsAL}, provided $Y$ is finite dimensional, $\sigma(y) = \| y \|^2 / 2$, 
$(x_*, \lambda_*)$ is a KKT-pair, and the restriction of the function $\sigma(\lambda_*, T^2_K(G(x_*), \cdot))$ to its
effective domain is u.s.c. (see~\cite{ShapiroSun}, formulae (3.23) and (3.25)). In this case one has
$$
  \omega_c(h) = \min_{z \in \mathscr{C}(x_*, \lambda_*)}
  \Big( c \big\| D G(x_*) h - z \big\|^2 - \sigma\big(\lambda_*, T_K^2(G(x_*), z) \big) \Big),
$$
where $\mathscr{C}(x_*, \lambda_*) = \{ z \in T_K(G(x_*)) \mid \langle \lambda_*, z \rangle = 0 \}$. In particular, the
assumption holds true in Examples~\ref{Example_RockWetsAL_SOC} and \ref{Example_RockWetsAL_SemiDefProg}, 
if $(x_*, \lambda_*)$ is a KKT-pair.

The assumption is satisfied in Example~\ref{Example_EssentiallyQuadraticAL} with
$$
  \omega_c(h) = \sum_{i \in I_+(x_*, \lambda_*)} c \phi''(0) \| \nabla g_i(x_*) h \|^2 +
  \sum_{i \in I_0(x_*, \lambda_*)} c \phi''(0) \max\{ 0, \langle \nabla g_i(x_*), h \rangle \}^2,
$$
provided $\phi''(0) > 0$, where $I_+(x_*, \lambda_*) = \{ i \in I(x_*) \mid (\lambda_*)_i > 0 \}$, 
$I_0(x_*, \lambda_*) = \{ i \in I(x_*) \mid (\lambda_*)_i = 0 \}$ and $I(x_*) = \{ i \in I \mid g_i(x_*) = 0 \}$. 
The assumption is satisfied in Example~\ref{Example_CubilAL} iff the strict complementarity (s.c.) condition holds true,
and it is satisfied in Example~\ref{Example_Mangasarian} iff $\phi''(t) > 0$ for all $t \in \mathbb{R}$. This
assumptions is valid in Examples~\ref{Example_ExpPenFunc}--\ref{Example_ModBarrierFunc} iff $\phi'(0) \ne 0$,
$\phi''(0) > 0$ and s.c. condition holds true. The assumptions is valid in Example~\ref{Example_pthPowerAugmLagr} iff
$\phi'(b) \ne 0$, $\phi''(b) > - \phi'(b)^2$ and s.c. condition holds true, and it is always satisfied in
Example~\ref{Example_HeWuMengLagrangian}.

The assumption is valid in Example~\ref{Example_NonlinearRescale_SOC}, provided the constraint nondegeneracy and s.c.
conditions are satisfied (\cite{ZhangGuXiao2011}, Cor.~3.1 and Prp.~3.1). Similarly, this assumption is
satisfied in Examples~\ref{Example_NonlinearRescale_SemiDefProg} and
\ref{Example_NonlinearRescalePenalized_SemiDefProg}, if s.c. condition holds true
(\cite{Stingl2006}, Thm.~5.1, and \cite{LuoWuLiu2015}, Prp.~4.2).

Finally, one can check that the function $\Phi(x, \lambda, c)$ from Example~\ref{Example_NonlinearRescale_SemiInfProg}
never admits the second order expansion in $x$ due to the fact that in this example $\omega_c(h) \to 0$ as 
$c \to + \infty$ for any $h \in C(x_*, \lambda_*)$, while the sigma term $\sigma(\lambda_*, \mathcal{T}(h))$ for 
semi-infinite programming problems is not identically equal to zero in the general case (see~\cite{BonnansShapiro},
Section~5.4.3). Apparently, the same conclusion can be drawn for Example~\ref{Example_RockafellarWetsAL_SemiInfProg}.
\end{remark}

Let us show that a local saddle point of the augmented Lagrangian $\mathscr{L}(x, \lambda, c)$ must be a KKT-pair
of the problem $(\mathcal{P})$ that under some additional assumptions satisfies the second order necessary optimality
condition.

\begin{proposition} \label{Prp_LSPisKKTpair}
Let $(x_*, \lambda_*)$ be a local saddle point of $\mathscr{L}(x, \lambda, c)$, and the functions $f_k$, $k \in M$, and
$G$ be differentiable at $x_*$. Suppose also that assumptions $(A1)$--$(A3)$, $(A8)$, $(A9)$ and $(A11)$ are
satisfied. Then $(x_*, \mu_*)$ is a KKT-pair of the problem $(\mathcal{P})$, where $\mu_* = \Phi_0(\lambda_*)$.
If, in addition, $x_* \in \interior A$, the functions $f_k$, $k \in M$, and $G$ are twice continuously
differentiable at $x_*$, the function $\Phi(G(x), \lambda, c)$ admits the second order expansion in $x$ at 
$(x_*, \lambda_*)$, and either $m = 1$ or $\Phi(G(\cdot), \lambda_*, c)$ is twice continuously differentiable at $x_*$,
then the KTT-pair $(x_*, \mu_*)$ satisfies the second order necessary optimality condition.
\end{proposition}

\begin{proof}
Proposition~\ref{Prp_SaddlePointImpliesOptSol} implies that $x_*$ is a feasible point of the problem $(\mathcal{P})$,
i.e. $G(x_*) \in K$, while Propositions~\ref{Prp_NonnegativeMultiplierAtSP} and \ref{Prp_ComplementarySlackAtSP}, and
$(A11)$ imply that $\mu_* \in K^*$ and $\langle \mu_*, G(x_*) \rangle = 0$. 

Taking into account $(A11)$ one gets that the function $\mathscr{L}(\cdot, \lambda_*, c)$ is Hada\-mard directionally
differentiable at $x_*$, and $[\mathscr{L}(\cdot, \lambda_*, c)]'(x_*, h) = [L(\cdot, \mu_*)]'(x_*, h)$ for all  $h \in
X$ and $c > 0$. Note that by the definition of local saddle point $x_*$ is a local minimizer of $\mathscr{L}(\cdot,
\lambda_*, c)$ on the set $A$ for any $c > c^*_{loc}(x_*, \lambda_*)$. Therefore
\begin{equation} \label{FirstOrderOptCond}
  [\mathscr{L}(\cdot, \lambda_*, c)]'(x_*, h) = [L(\cdot, \mu_*)]'(x_*, h) \ge 0 
  \quad \forall h \in T_A(x_*)
\end{equation}
for any $c > c^*_{loc}(x_*, \lambda_*)$. Thus, $(x_*, \mu_*)$ is a KKT-pair of the problem $(\mathcal{P})$. 

Let us now turn to second order necessary optimality conditions. Suppose, at first, that $m = 1$. In this case, the
function $L(\cdot, \mu_*)$ is twice differentiable at $x_*$, and $D_x L(x_*, \mu_*) = 0$ due to the fact that 
$x_* \in \interior A$. Fix arbitrary $c > c^*_{loc}(x_*, \lambda_*)$ and $h \in C(x_*, \mu_*)$, and choose a sequence 
$\{ \gamma_n \} \subset (0, 1)$ such that $\gamma_n \to 0$ as $n \to \infty$. Since $x_* \in \interior A$, one can
suppose that $x_n := x_* + \gamma_n h \in A$ for all $n \in \mathbb{N}$.

As it was noted above, $x_*$ is a local minimizer of $\mathscr{L}(\cdot, \lambda_*, c)$ on the set $A$. 
Therefore $\mathscr{L}(x_n, \lambda_*, c) \ge \mathscr{L}(x_*, \lambda_*, c)$ for any sufficiently large $n$.
Hence taking into account \eqref{FirstOrderOptCond}, the equality $D_x L(x_*, \mu_*) = 0$, and the fact that 
$\Phi(G(x), \lambda, c)$ admits the second order expansion in $x$ at $(x_*, \lambda_*)$ one obtains that
\begin{multline*}
  0 \le \mathscr{L}(x_n, \lambda_*, c) - \mathscr{L}(x_*, \lambda_*, c) \\ 
  = \frac{\gamma_n^2}{2} \langle h, D^2_{xx} L(x_*, \mu_*) h \rangle +
  \frac{\gamma_n^2}{2} \omega_c(h) + o(\| x_n - x_* \|^2)
\end{multline*}
for any $n$ large enough. Dividing this inequality by $\gamma_n^2$, passing to the limit as $n \to + \infty$, and 
then passing to the limit as $c \to \infty$ with the use of Def.~\ref{Def_SeconOrdeExpansion} one obtains that 
the KKT-pair $(x_*, \mu_*)$ satisfies the second order necessary optimality condition.

Suppose, now, that $\Phi(G(\cdot), \lambda_*, c)$ is twice continuously differentiable at $x_*$. Fix an arbitrary 
$c > c^*_{loc}(x_*, \lambda_*)$. As it was pointed out above, $x_*$ is a point of local minimum 
of $\mathscr{L}(\cdot, \lambda_*, c)$. Consequently, the point $(x_*, 0) \in \mathbb{R}^{d + 1}$ is a locally optimal
solution of the problem
\begin{equation} \label{ConvertedMiniMax}
  \min_{(x, z)} z \quad \text{subject to} \quad g_k(x, z) \le 0, \quad k \in M,
\end{equation}
where $g_k(x, z) = f_k(x) + \Phi(G(x), \lambda_*, c) - z$. The Lagrangian function for this problem has the form 
$\mathcal{L}(x, z, \alpha) = z + \sum_{k = 1}^m \alpha_k g_k(x, z)$. Observe that the set of Lagrange multipliers of
problem \eqref{ConvertedMiniMax} at the point $(x_*, 0)$ coincides with $\alpha(x_*, \mu_*)$.

From the fact that $\Phi(G(\cdot), \lambda_*, c)$ is twice continuously differentiable at $x_*$ it follows that 
the functions $g_k(x, z)$ are twice continuously differentiable at the point $(x_*, 0)$. Furthermore, 
the Mangasarian-Fromovitz constraint qualification obviously holds at this point. Therefore, applying the second order
necessary optimality conditions (see, e.g., \cite{BonnansShapiro}, Theorem~3.45) one can easily verify that for 
any $h \in \mathbb{R}^d$ such that $\langle D_x g_k(x_*, 0), h \rangle \le 0$ for all $k \in M(x_*)$ one has
$$
  \sup_{\alpha \in \alpha(x_*, \mu_*)} \langle h, D_{xx} \mathcal{L}(x_*, 0, \alpha) h \rangle \ge 0.
$$
Hence applying $(A11)$ and Def.~\ref{Def_SeconOrdeExpansion} one gets that
$$
  \sup_{\alpha \in \alpha(x_*, \mu_*)} \langle h, D_{xx} L_0(x_*, \mu_*, \alpha) h \rangle
  + \omega_c(h) \ge 0 \quad \forall c > c^*_{loc}(x_*, \lambda_*)
$$
for any $h \in \mathbb{R}^d$ such that 
$[L(\cdot, \mu_*)]'(x_*, h) = \max_{k \in M(x_*)} \langle D_x g_k(x_*, 0), h \rangle \le 0$. Passing to 
the limit as $c \to \infty$, and taking into account Def.~\ref{Def_SeconOrdeExpansion} one obtains that the KKT-pair
$(x_*, \mu_*)$ satisfies optimality condition \eqref{KKTSecondOrderDef}.	 
\end{proof}

Let us prove that under some additional assumptions any KKT-pair satisfying the second order \textit{sufficient}
optimality condition is a local saddle point of the augmented Lagrangian function $\mathscr{L}(x, \lambda, c)$.

Suppose that the functions $f_k$, $k \in M$ and $G$ are twice Fr\'echet differentiable at a point $x_* \in A$. One
says that a KKT-pair $(x_*, \lambda_*)$ satisfies \textit{the second order sufficient optimality condition}
(cf.~\cite{BonnansShapiro}, Theorem~3.86) if
\begin{multline} \label{KKT_2OrderSuffCond}
  \sup_{\alpha \in \alpha(x_*, \lambda_*)} 
  \big\langle h, D^2_{xx} L_0(x_*, \lambda_*, \alpha) h \big\rangle - \sigma(\lambda_*, \mathcal{T}(h)) > 0 \\
  \forall h \in C(x_*, \lambda_*) \setminus \{ 0 \} \colon [L(\cdot, \lambda_*)]'(x_*, h) = 0.
\end{multline}
The following result holds true.

\begin{theorem} \label{Thrm_LSPvia2OrderSuffCond}
Let $x_*$ be a locally optimal solution of the problem $(\mathcal{P})$, the functions $f_k$, $k \in M$, and $G$ be twice
Fr\'echet differentiable at $x_*$, and $(x_*, \mu_*)$ be a KKT-pair of the problem $(\mathcal{P})$ satisfying the
second order sufficient optimality condition. Suppose also that assumptions $(A2)$, $(A4)$, $(A10)$ and $(A11)$ hold
true, and the function $\Phi(G(x), \lambda, c)$ admits the second order expansion in $x$ at $(x_*, \lambda_*)$ for some
$\lambda_* \in \Phi_0^{-1}(\mu_*)$. Then $(x_*, \lambda_*)$ is a local saddle point of $\mathscr{L}(x, \lambda, c)$.
\end{theorem}

\begin{proof}
Taking into account the fact that $(x_*, \mu_*)$ is a KKT-pair of the problem $(\mathcal{P})$ and utilizing assumption
$(A2)$, $(A10)$ and $(A11)$ one can easily verify that for all $c > 0$ one has
$\sup_{\lambda \in \Lambda} \mathscr{L}(x_*, \lambda, c) \le   \mathscr{L}(x_*, \lambda_*, c)$.

Applying $(A11)$ and Def.~\ref{Def_SeconOrdeExpansion} one obtains that for any $c > 0$ there exists a neighbourhood
$U_c$ of $x_*$ such that for all $x \in U_c$ one has
\begin{multline} \label{AugmLagr_TaylorExpnInX}
  \Big| \mathscr{L}(x, \lambda_*, c) - \mathscr{L}(x_*, \lambda_*, c) \\
  - \max_{k \in M(x_*)} \Big( \langle \nabla f_k(x_*), x - x_* \rangle + 
  \frac12 \big \langle x - x_*, \nabla^2 f_k(x_*) (x - x_*) \big \rangle \Big) \\
  - \Big\langle \mu_*, D G(x_*) (x - x_*) + \frac{1}{2} D^2 G(x_*)(x - x_*, x - x_*) \Big\rangle \\
  - \frac{1}{2} \omega_c(x - x_*) \Big| < \frac{1}{2c} \| x - x_* \|^2.
\end{multline}
Arguing by reductio ad absurdum, suppose that $(x_*, \lambda_*)$ is not a local saddle point 
of $\mathscr{L}(x, \lambda, c)$. Then for any $n \in \mathbb{N}$ there exists $x_n \in A \cap U_n$ such that 
$\mathscr{L}(x_n, \lambda_*, n) < \mathscr{L}(x_*, \lambda_*, n)$. Then taking into account
\eqref{AugmLagr_TaylorExpnInX} one obtains that 
\begin{multline} \label{AugmLagr_TaylorExpnInX_Seq}
  0 > \mathscr{L}(x_n, \lambda_*, n) - \mathscr{L}(x_*, \lambda_*, n) 
  \ge \max_{k \in M(x_*)} \Big( \langle \nabla f_k(x_*), z_n \rangle + 
  \frac12 \langle z_n, \nabla^2 f_k(x_*) z_n \rangle \Big) \\
  + \left\langle \mu_*, D G(x_*) z_n + \frac{1}{2} D^2 G(x_*)(z_n, z_n) \right\rangle 
  + \frac{1}{2} \omega_n(z_n) - \frac{1}{2n} \| z_n \|^2
\end{multline}
for any $n \in \mathbb{N}$, where $z_n = x_n - x_*$. 

Recall that if $\alpha \in \alpha(x_*, \mu_*)$, then $\alpha \in \mathbb{R}^m_+$, $\alpha_k = 0$ for all 
$k \notin M(x_*)$, and $\sum_{k = 1}^m \alpha_k = 1$. Therefore with the use of \eqref{AugmLagr_TaylorExpnInX_Seq} one
obtains that for any $\alpha \in \alpha(x_*, \mu_*)$ and $n \in \mathbb{N}$ one has
\begin{multline} \label{NonLSP2OrderExpans}
  0 > \mathscr{L}(x_n, \lambda_*, n) - \mathscr{L}(x_*, \lambda_*, n) \ge
  \big\langle D_x L_0(x_*, \mu_*, \alpha), z_n \big\rangle \\
  + \frac12 \big\langle z_n , D^2_{xx} L_0(x_*, \mu_*, \alpha) z_n \big\rangle 
  + \frac{1}{2} \omega_n(z_n) - \frac{1}{2n} \| z_n \|^2.
\end{multline}
Denote $h_n = z_n / \| z_n \|$. Since $X$ is a finite dimensional space, without loss of generality one can suppose
that the sequence $h_n$ converges to a point $h_*$ with $\| h_* \| = 1$. Furthermore, since $x_n \in A$ for all 
$n \in \mathbb{N}$ and $A$ is convex, $h_* \in T_A(x_*)$ and $z_n \in T_A(x_*)$ for all $n \in \mathbb{N}$.

Let us check that $[L(\cdot, \mu_*)]'(x_*, h_*) = 0$. Indeed, from the facts that $(x_*, \mu_*)$ is a KKT-pair and 
$h_* \in T_A(x_*)$ it follows that $[L(\cdot, \mu_*)]'(x_*, h_*) \ge 0$. Suppose 
that $[L(\cdot, \mu_*)]'(x_*, h_*) > 0$. Then by virtue of $(A11)$ for any $c > 0$ one has
$$
  \lim_{n \to \infty} \frac{\mathscr{L}(x_* + \gamma_n h_n, \lambda_*, c) - \mathscr{L}(x_*, \lambda_*, c)}{\gamma_n} =
  [L(\cdot, \mu_*)]'(x_*, h_*) > 0,
$$
where $\gamma_n = \| x_n - x_* \|$. Note that $x_* + \gamma_n h_n = x_n$. Therefore, in particular, there 
exists $n_0 \in \mathbb{N}$ such that $\mathscr{L}(x_n, \lambda_*, 1) > \mathscr{L}(x_*, \lambda_*, 1)$ for all 
$n \ge n_0$. Observe that $\mathscr{L}(x_*, \lambda_*, c) = f(x_*)$ for all $c > 0$ due to $(A10)$ and the fact that
$(x_*, \mu_*)$ is a KKT-pair, while $\mathscr{L}(x_n, \lambda_*, c) \ge \mathscr{L}(x_n, \lambda_*, 1)$ for all 
$c \ge 1$ by $(A4)$. Consequently, $\mathscr{L}(x_n, \lambda_*, n) > \mathscr{L}(x_*, \lambda_*, n)$ for
any $n \ge n_0$, which contradicts the definition of $x_n$. Thus, $[L(\cdot, \mu_*)]'(x_*, h_*) = 0$.

From \eqref{NonLSP2OrderExpans} and the fact that $(x_*, \mu_*)$ is a KKT-pair (see \eqref{KKTdef}) it follows that
$$
  0 > \big\langle z_n, D^2_{xx} L_0(x_*, \mu_*, \alpha) z_n \big\rangle
  + \omega_n(z_n) - \frac{1}{n} \| z_n \|^2.
$$
Dividing this inequality by $\| z_n \|^2$, passing to the limit superior as $n \to \infty$ with the use of
Def.~\ref{Def_SeconOrdeExpansion}, and taking the supremum over all $\alpha \in \alpha(x_*, \mu_*)$ one obtains that
$$
  0 \ge 
  \sup_{\alpha \in \alpha(x_*, \mu_*)} \big\langle h_*, D^2_{xx} L_0(x_*, \mu_*, \alpha) h_* \big\rangle 
  - \sigma(\mu_*, \mathcal{T}(h_*)), \quad h_* \in C(x_*, \mu_*),
$$
which contradicts the fact that the KKT-pair $(x_*, \mu_*)$ satisfies the second order sufficient optimality
condition due to the fact that $[L(\cdot, \mu_*)]'(x_*, h_*) = 0$.	 
\end{proof}

\begin{remark} \label{Rmrk_MangasarianAL}
{(i)~Let us note that $\Phi_0(\lambda) \equiv \lambda$ for most particular augmented Lagrangian functions appearing in
applications. Therefore, usually, there is a direct connection between local saddle points and KKT-pairs corresponding
to locally optimal solutions, i.e. if $(x_*, \lambda_*)$ is a local saddle point, then $x_*$ is a locally optimal
solution, and $(x_*, \lambda_*)$ is a KKT-pair satisfying the second order necessary optimality condition
(Proposition~\ref{Prp_LSPisKKTpair}), and, conversely, if $(x_*, \lambda_*)$ is a KKT-pair satisfying the second order
sufficient optimality condition, then $(x_*, \lambda_*)$ is a local saddle point
(Theorem~\ref{Thrm_LSPvia2OrderSuffCond}).
}

\noindent{(ii)~As it was noted above, existing augmented Lagrangian functions for semi-infinite programming
problems do not admit the second order expansion in $x$ in the sense of Def.~\ref{Def_SeconOrdeExpansion} due to the
absence of the sigma term $\sigma(\lambda_*, \mathcal{T}(h_*))$ in their second order expansions. Therefore, the second
order sufficient optimality condition \eqref{KKT_2OrderSuffCond} cannot be utilized in order to prove the existence of
a local saddle point in the case of semi-infinite programming problems. More important, one can easily provide an
example of a semi-infinite programming problem such that there exists a KKT-pair of this problem satisfying
\eqref{KKT_2OrderSuffCond}, but which is not a local saddle point of augmented Lagrangian functions from
Example~\ref{Example_NonlinearRescale_SemiInfProg} and Example \ref{Example_RockafellarWetsAL_SemiInfProg} with
$\sigma(y)$ defined as in \cite{RuckmannShapiro,BurachikYangZhou2017} (such KKT-pair must not satisfy optimality
conditions without the sigma term).
}

\noindent{(iii)~Theorem~\ref{Thrm_LSPvia2OrderSuffCond} unifies and sharpens many known results on existence of local
saddle points of augmented Lagrangian for mathematical programming problems (see, e.g., 
\cite{ZhaoZhangZhou2010}, Thm.~3.3; \cite{ZhouXiuWang}, Thm.~2.8; \cite{WuLuo2012b}, Thm.~2; 
\cite{LiuTangYang2009}, Thm.~2; \cite{LiuYang}, Thm.~3.2; \cite{SunLi}, Thms.~2.1, 2.3 and 2.4), nonlinear second order
cone programming problems (\cite{ZhouChen2015}, Thm.~2.3) and nonlinear semidefinite programming problems
(\cite{LuoWuLiu2015}, Thm.~4.2). Furthermore, Theorem~\ref{Thrm_LSPvia2OrderSuffCond} also extends the aforementioned
results to the case of minimax cone constrained optimization problem. To the best of author's knowledge,
Theorem~\ref{Thrm_LSPvia2OrderSuffCond} contains first simple sufficient conditions for the existence of local saddle
points in the case of cone constrained \textit{minimax} problems. Let us also note that
Theorem~\ref{Thrm_LSPvia2OrderSuffCond} provides a correct proof of Theorem 4 in \cite{Dolgopolik_AugmLagrMult}.
}
\end{remark}

In the theorem above, we demonstrated that the existence of a local saddle point of the augmented Lagrangian function
$\mathscr{L}(x, \lambda, c)$ can be easily proved via second order sufficient optimality conditions. Let us show
that one can utilize \textit{first} order sufficient optimality conditions for constrained minimax problems
\cite{MalozemovPevnyi,DaugavetMalozemov75,Daugavet,DaugavetMalozemov81,DemyanovMalozemov_Alternance,
DemyanovMalozemov_Collect} in order to obtain a similar result.

Suppose that the functions $f_k$, $k \in M$ and $G$ are differentiable at a point $x_* \in A$, and let 
$(x_*, \lambda_*)$ be a KKT-pair of the problem $(\mathcal{P})$. Then $[L(\cdot, \lambda_*)]'(x_*, h) \ge 0$
for any $h \in T_A(x_*)$. The natural ``no gap'' first order \textit{sufficient} optimality condition for the problem
$(\mathcal{P})$ has
the form
\begin{equation} \label{FirstOrderKKTSuffCond}
  [L(\cdot, \lambda_*)]'(x_*, h) > 0 \quad \forall h \in C(x_*, \lambda_*) \setminus \{ 0 \}
\end{equation}
(cf.~\cite{BonnansShapiro}, Section~3.1.4). The main drawbacks of this sufficient optimality condition consist in the
facts that this condition is difficult to verify, and it rarely holds true for smooth problems. However, sufficient
optimality condition \eqref{FirstOrderKKTSuffCond} often holds for constrained minimax problems, and, furthermore, it
can be reformulated in a more convenient form of the so-called \textit{alternance conditions} that are independent of
the Lagrange multiplier $\lambda_*$.

Let us introduce \textit{alternance optimality conditions} for the problem $(\mathcal{P})$. Let $Z \subset \mathbb{R}^d$
be a set consisting of $d$ linearly independent vectors. Denote 
by $N_A(x_*) = \{ z \in \mathbb{R}^d \mid \langle z, h \rangle \le 0 \: \forall h \in T_A(x_*) \}$ the normal cone to
$A$ at $x_*$. For any linear subspace $Y_0 \subset Y$ denote by 
$Y_0^{\perp} = \{ y^* \in Y^* \mid \langle y^*, y \rangle = 0 \: \forall y \in Y_0 \}$ the annihilator of $Y_0$.
For the sake of correctness, for any linear operator $U \colon \mathbb{R}^d \to Y$ denote by $[U]^*$ the composition of
the natural isomorphism between $(\mathbb{R}^d)^*$ and $\mathbb{R}^d$ and the adjoint operator 
$U^* \colon Y^* \to (\mathbb{R}^d)^*$.

One says that a $p$-\textit{point alternance} exists at $x_*$ with $p \in \{ 1, \ldots, d + 1\}$, if there exist 
$k_0 \in \{ 1, \ldots, p \}$, $i_0 \in \{ k_0 + 1, \ldots, p \}$, vectors
\begin{gather*}
  V_1, \ldots, V_{k_0} \in \Big\{ \nabla f_k(x_*) \Bigm| k \in M(x_*) \Big\}, \\
  V_{k_0 + 1}, \ldots, V_{i_0} \in \big[ DG(x_*) \big]^* \Big( K^* \cap \linhull(G(x_*))^{\perp} \Big), \quad
  V_{i_0 + 1}, \ldots, V_p \in N_A(x_*),
\end{gather*}
and vectors $V_{p + 1}, \ldots, V_{d + 1} \in Z$ such that the $d$th-order determinants $\Delta_s$ of the matrices
composed of the columns $V_1, \ldots, V_{s - 1}, V_{s + 1}, \ldots V_{d + 1}$ satisfy the following conditions
\begin{gather*}
  \Delta_s \ne 0, \quad s \in \{ 1, \ldots, p \}, \quad
  \sign \Delta_s = - \sign \Delta_{s + 1}, \quad s \in \{ 1, \ldots, p - 1 \}, \\
  \Delta_s = 0, \quad s \in \{ p + 1, \ldots d + 1 \}.
\end{gather*}
One can verify that a $p$-point alternance exists at $x_*$ for some $p \in \{ 1, \ldots, d + 1 \}$ iff there exists
$\lambda_* \in K^*$ such that $(x_*, \lambda_*)$ is a KKT-pair of the problem $(\mathcal{P})$. Moreover, the existence
of a \textit{complete} (i.e. $d+1$-point) alternance is a first order sufficient optimality condition for the problem
$(\mathcal{P})$
(see~\cite{MalozemovPevnyi,DaugavetMalozemov75,Daugavet,DaugavetMalozemov81,DemyanovMalozemov_Alternance,
DemyanovMalozemov_Collect} for more details). Note that in the case of complete alternance one has
$$
  \Delta_s \ne 0 \quad s \in \{ 1, \ldots, d + 1 \}, \quad
  \sign \Delta_s = - \sign \Delta_{s + 1} \quad s \in \{ 1, \ldots, d \},
$$
i.e. the determinants $\Delta_s$, $s \in \{1, \ldots, d + 1 \}$ are not equal to zero and have \textit{alternating}
signs, which explains the term \textit{alternance optimality conditions}. Finally, it should be mentioned that usually
alternance optimality conditions can only be applied in the case when the cardinality of $M(x_*)$ is greater than $1$,
i.e. when the objective function $f(x)$ is nonsmooth at $x_*$.

\begin{remark}
Let us point out that there is a close connection between standard and alternance optimality conditions for minimax
optimization problems. In particular, if one considers the unconstrained problem (i.e. $G(x) \equiv 0$, $K = \{ 0 \}$
and $A = \mathbb{R}^d$), then the standard first order necessary optimality condition 
$0 \in \partial f(x_*) = \{ \nabla f_k(x_*) \mid k \in M(x_*) \}$ is equivalent to the existence of a $p$-point
alternance at $x_*$ for some $p \in \{ 1, \ldots, d + 1 \}$. Similarly, the natural first order sufficient optimality
condition $0 \in \interior \partial f(x_*)$, that is equivalent to the first order growth condition at $x_*$ (i.e. there
exists $\gamma > 0$ such that $f(x) \ge f(x_*) + \gamma \| x - x_* \|$ for any $x$ in a neighbourhood of $x_*$), is also
equivalent to the existence of a complete alternance. Let us also note that the existence of complete alternance is a
natural assumption for many particular minimax problems. For more details on alternance optimality conditions see
\cite{MalozemovPevnyi,DaugavetMalozemov75,Daugavet,DaugavetMalozemov81,DemyanovMalozemov_Alternance,
DemyanovMalozemov_Collect}.
\end{remark}

Our aim is to prove that the existence of a complete alternance at a locally optimal solution of the problem
$(\mathcal{P})$ guarantees the existence of a local saddle point of 
the augmented Lagrangian $\mathscr{L}(x, \lambda, c)$.

\begin{theorem} \label{Thrm_LSPviaFirstOrderSuffCond}
Let $x_*$ be a locally optimal solution of the problem $(\mathcal{P})$, the functions $f_k$, $k \in M$ and $G$, be twice
Fr\'echet differentiable at $x_*$, and let a complete alternance exists at $x_*$. Suppose also that assumptions $(A2)$,
$(A4)$, $(A10)$ and $(A11)$ are satisfied, and the function $\Phi(G(x), \lambda, c)$ admits the second order expansion
in $x$ at $(x_*, \lambda)$ for any $\lambda \in K^*$ such that $\langle \lambda, G(x_*) \rangle = 0$. Then there exists 
$\lambda_* \in K^*$ such that $(x_*, \lambda_*)$ is a local saddle point of the augmented Lagrangian 
$\mathscr{L}(x, \lambda, c)$. Furthermore, under the assumptions of the theorem for any KKT-pair $(x_*, \mu_*)$ of
the problem $(\mathcal{P})$, and for all $\lambda_* \in \Phi_0^{-1}(\mu_*)$ the pair $(x_*, \lambda_*)$ is a local
saddle point of $\mathscr{L}(x, \lambda, c)$.
\end{theorem}

\begin{proof}
Let us verify, at first, that there exists $\mu_* \in K^*$ such that $(x_*, \mu_*)$ is a KKT-pair of the problem
$(\mathcal{P})$. Indeed, let $k_0 \in \{ 1, \ldots, d + 1 \}$, $i_0 \in \{ k_0 + 1, \ldots, d + 1 \}$ and vectors 
$V_1, \ldots, V_{d + 1} \in \mathbb{R}^d$ be from the definition of complete alternance. Applying Cramer's rule to the
system $\sum_{s = 2}^{d + 1} \beta_s V_s = - V_1$ one obtains that there exist unique $\beta_s > 0$, 
$s \in \{ 2, \ldots, d + 1 \}$ such that
\begin{equation} \label{AlternanceCond_CramerRule}
  V_1 + \sum_{s = 2}^{d + 1} \beta_s V_s = 0, \quad 
  \beta_s = (-1)^{s - 1} \frac{\Delta_s}{\Delta_1} > 0 \quad s \in \{ 2, \ldots, d + 1 \}.
\end{equation}
Denote $\widehat{\beta} = 1 + \beta_2 + \ldots + \beta_{k_0}$, $\gamma_1 = 1 / \widehat{\beta}$, and 
$\gamma_s = \beta_s / \widehat{\beta} > 0$, $s \in \{ 2, \ldots, d + 1 \}$. 
Define $\alpha = (\alpha_1, \ldots, \alpha_m) \in \mathbb{R}^m_+$ as follows
$$
  \alpha_k = \begin{cases}
    \gamma_s, & \text{if } k \in M(x_*) \text{ and } \exists s \in \{ 1, \ldots, k_0 \} \colon V_s = \nabla f_k(x_*), \\
    0, & \text{otherwise}.
  \end{cases}
$$
Then one can easily see that $\alpha_1 + \ldots + \alpha_m = 1$, $\alpha_k = 0$ if $k \notin M(x_*)$, and
$$
  \sum_{s = 1}^{k_0} \gamma_i V_i = \sum_{k = 1}^m \alpha_k \nabla f_k(x_*).
$$
Denote $W = \sum_{k_0 + 1}^{i_0} \gamma_s V_s$. From the fact that $K^* \cap \linhull(G(x_*))^{\perp}$ is a convex cone
it follows that $W \in [DG(x_*)]^* (K^* \cap \linhull(G(x_*))^{\perp})$. Hence $W = [DG(x_*)]^* \mu_*$ for some 
$\mu_* \in K^*$ such that $\langle \mu_*, G(x_*) \rangle = 0$ (in the case $k_0 = d + 1$ one has $W = 0$ and 
$\mu_* = 0$). Therefore for any $h \in \mathbb{R}^d$ one has
$$
  \big\langle D_x L_0(x_*, \mu_*, \alpha), h \big\rangle =
  \Big\langle \sum_{s = 1}^{i_0} \gamma_i V_i, h \Big\rangle = 
  - \gamma_s \sum_{s = {i_0 + 1}}^{d + 1} \langle V_s, h \rangle.
$$
By definition, $V_s \in N_A(x_*)$ for any $s \in \{ i_0 + 1, \ldots, d + 1 \}$. Consequently, 
$$
  \big\langle D_x L_0(x_*, \mu_*, \alpha), h \big\rangle \ge 0 \quad \forall h \in T_A(x_*)
$$
(note that in the case $k_0 = d + 1$ or $i_0 = d + 1$ one has $D_x L_0(x_*, \mu_*, \alpha) = 0$, and the above
inequality holds trivially). Thus, $(x_*, \mu_*)$ is a KKT-pair of the problem $(\mathcal{P})$ and 
$\alpha \in \alpha(x_*, \mu_*)$.

Let, now, $\mu_* \in K^*$ be such that $(x_*, \mu_*)$ is a KKT-pair of the problem $(\mathcal{P})$, and 
$\lambda_* \in \Phi_0^{-1}(\mu_*)$ be arbitrary. With the use of $(A2)$, $(A10)$ and $(A11)$ one obtains that 
$\sup_{\lambda \in \Lambda} \mathscr{L}(x_*, \lambda, c) \le \mathscr{L}(x_*, \lambda_*, c)$ for all $c > 0$.

Applying $(A11)$ and Def.~\ref{Def_SeconOrdeExpansion} one gets that for any $c > 0$ there exists a neighbourhood $U_c$
of $x_*$ such that for any $x \in U_c$ one has
\begin{multline} \label{AugmLagrTaylorExpans_AC}
  \Big| \mathscr{L}(x, \lambda_*, c) - \mathscr{L}(x_*, \lambda_*, c) \\
  - \max_{k \in M(x_*)} \Big( \langle \nabla f_k(x_*), x - x_* \rangle + 
  \frac12 \big\langle x - x_*, \nabla^2 f_k(x_*) (x - x_*) \big\rangle \Big) \\
  - \big\langle \mu_*, DG(x_*)(x - x_*) \big\rangle 
  - \frac{1}{2} \big\langle \mu_*, D^2 G(x_*)(x - x_*, x - x_*) \big\rangle \\
  - \frac{1}{2} \omega_c(x - x_*) \Big| < \frac{1}{2 c} \| x - x_* \|^2.
\end{multline}
Recall that our aim is to show that $(x_*, \lambda_*)$ is a local saddle point of $\mathscr{L}(x, \lambda, c)$. Arguing
by reductio ad absurdum, suppose that this claim is false. Then for any $n \in \mathbb{N}$ there exists 
$x_n \in A \cap U_n$ such that $\mathscr{L}(x_n, \lambda_*, n) < \mathscr{L}(x_*, \lambda_*, n)$. Taking into account
\eqref{AugmLagrTaylorExpans_AC} one obtains that for any $n \in \mathbb{N}$ the following inequality holds true
\begin{multline} \label{TaylorExpansNonLSP_AC}
  0 > \max_{k \in M(x_*)} \Big( \langle \nabla f_k(x_*), z_n \rangle + 
  \frac12 \big\langle z_n, \nabla^2 f_k(x_*) z_n \big\rangle \Big) + \big\langle \mu_*, DG(x_*) z_n \big\rangle \\
  + \frac{1}{2} \big\langle \mu_*, D^2 G(x_*)(z_n, z_n) \big\rangle
  + \frac{1}{2} \omega_n(z_n) - \frac{1}{2 n} \| z_n \|^2,
\end{multline}
where $z_n = x_n - x_*$. For any $n \in \mathbb{N}$ denote $h_n = z_n / \| z_n \|$. Without loss of generality, one can
suppose that the sequence $\{ h_n \}$ converges to a vector $h_* \in T_A(x_*)$ such that $\| h_* \| = 1$. 

Since $A$ is convex, $z_n \in T_A(x_*)$ for all $n \in \mathbb{N}$. Hence taking into account
\eqref{TaylorExpansNonLSP_AC} and the fact that $(x_*, \mu_*)$ is a KKT-pair one obtains that 
$$
  0 > \langle z_n, D^2_{xx} L_0(x_*, \mu_*, \alpha) z_n \rangle + 
  \omega_n(z_n) - \frac{1}{n} \| z_n \|^2
$$
for any $n \in \mathbb{N}$ and $\alpha \in \alpha(x_*, \mu_*)$. Dividing this inequality by $\| z_n \|^2$ and
passing to the limit superior as $n \to \infty$ with the use of Def.~\ref{Def_SeconOrdeExpansion} one gets that
$h_* \in C(x_*, \mu_*)$, which implies that $\langle \mu_*, D G(x_*) h_* \rangle = 0$ due to the definition
of the cone $C(x_*, \mu_*)$.

Dividing \eqref{TaylorExpansNonLSP_AC} by $\| z_n \|$, passing to the limit superior as $n \to \infty$ with the use of
Def.~\ref{Def_SeconOrdeExpansion}, and taking into account the facts that 
$\omega_n(z_n) / \| z_n \| = \omega_n(z_n / \sqrt{\| z_n \|})$, $z_n / \sqrt{\| z_n \|} \to 0$ as $n \to \infty$, and
$\sigma(\mu_*, \mathcal{T}(0)) = 0$ one obtains
\begin{equation} \label{DegenLimitDirection}
  0 \ge \max_{k \in M(x_*)} \langle \nabla f_k(x_*), h_* \rangle.
\end{equation}
Here we used the equality $\langle \mu_*, D G(x_*) h_* \rangle = 0$.

Introduce the matrix $V = (V_1, \ldots, V_{d + 1})$, and define $\beta = (1, \beta_2, \ldots, \beta_{d + 1})^T$, where
the vectors $V_s$ are from the definition of complete alternance, and $\beta_s$ are from
\eqref{AlternanceCond_CramerRule}. The first equality in \eqref{AlternanceCond_CramerRule} implies that 
\begin{equation} \label{AlternanceCond_Contr}
  \langle V \beta, h_* \rangle = \langle \beta, V^T h_* \rangle = \langle V_1, h_* \rangle +
  \sum_{s = 2}^{d + 1} \beta_s \langle V_s, h_* \rangle = 0.
\end{equation}
As it was shown above, $h_* \in C(x_*, \mu_*)$. Therefore $h_* \in T_A(x_*)$ and $\langle V_s, h_* \rangle \le 0$
for any $s \in \{ i_0 + 1, \ldots, d + 1 \}$. Furthermore, from the facts that $h_* \in C(x_*, \mu_*)$ and $K$ is
convex it follows that $D G(x_*) h_* \in T_K(G(x_*)) = \cl \cone(K - G(x_*))$ and 
$\langle y^*, D G(x_*) h_* \rangle \le 0$ for any $y^* \in K^* \cap \linhull(G(x_*))^{\perp}$, which yields that
$\langle V_s, h_* \rangle \le 0$ for any $s \in \{ k_0 + 1, \ldots, i_0 \}$. Thus, $\langle V_s, h_* \rangle \le 0$ for
all $s \in \{ k_0 + 1, \ldots, d + 1 \}$.

The definition of complete alternance implies that the matrix $V$ has full rank. Consequently, $V^T h_* \ne 0$
due to the fact that $h_* \ne 0$. Hence taking into account \eqref{AlternanceCond_Contr} and the fact that $\beta_s > 0$
for all $s \in \{ 2, \ldots, d + 1 \}$ one obtains that there exists $s_0 \in \{ 1, \ldots, k_0 \}$ such that 
$\langle V_{s_0}, h_* \rangle > 0$. Recall that $V_s \in \{ \nabla f_k(x_*) \mid k \in M(x_*) \}$ for any 
$s \in \{ 1, \ldots, k_0 \}$. Therefore there exists $k \in M(x_*)$ such that $V_{s_0} = \nabla f_k(x_*)$. Hence one has
$\max_{k \in M(x_*)} \langle \nabla f_k(x_*), h_* \rangle > 0$, which contradicts \eqref{DegenLimitDirection}. Thus,
$(x_*, \lambda_*)$ is a local saddle point of $\mathscr{L}(x, \lambda, c)$.	 
\end{proof}

Now, we can easily obtain simple necessary and sufficient conditions for the existence of a global saddle point of the
augmented Lagrangian $\mathscr{L}(x, \lambda, c)$ with the use of the localization principle.

Recall that $\Omega_*$ is the set of globally optimal solutions of the problem $(\mathcal{P})$.
By Proposition~\ref{Prp_LSPisKKTpair} any local saddle point $(x_*, \lambda_*)$ of the augmented Lagrangian function
$\mathscr{L}(x, \lambda, c)$ must be a KKT-pair of the problem $(\mathcal{P})$. Hence taking into account
Corollary~\ref{Crlr_IndependenceGSPofOptSol} one obtains that for the existence of a global saddle point 
of $\mathscr{L}(x, \lambda, c)$ it is \textit{necessary} that there exists $\lambda_* \in K^*$ such that $(x_*,
\lambda_*)$ is a
KKT-pair of $(\mathcal{P})$ for any $x_* \in \Omega_*$. Therefore, if there are at least two globally optimal solutions
of the problem $(\mathcal{P})$ that have disjoint sets of Lagrange multipliers, then there are no global saddle points
of the augmented Lagrangian $\mathscr{L}(x, \lambda, c)$.

Under the additional assumption that the first or the second order sufficient optimality conditions hold at every 
$x_* \in \Omega_*$, one can demonstrate that the existence of $\lambda_* \in K^*$ such that $(x_*, \lambda_*)$ is a
KKT-pair of $(\mathcal{P})$ for any $x_* \in \Omega_*$ is also sufficient for the existence of a global saddle point of 
$\mathscr{L}(x, \lambda, c)$.

\begin{theorem} \label{Th_LocPrincipleSP_MiniMaxProblems}
Let $A$ be closed, $G$ be continuous on $A$, and $\mathscr{L}(\cdot, \lambda, c)$ be l.s.c. on $A$ for any 
$\lambda \in \Lambda$ and $c > 0$. Suppose that assumptions $(A1)$--$(A4)$, $(A6)$, $(A10)$ and $(A11)$ are satisfied.
Let also the following assumptions be valid:
\begin{enumerate}
\item{$f_k$, $k \in M$ and $G$ are twice Fr\'echet differentiable at every point $x_* \in \Omega_*$;
}

\item{there exists $\mu_* \in \Lambda$ such that $(x_*, \mu_*)$ is a KKT-pair of the problem $(\mathcal{P})$ 
for any $x_* \in \Omega_*$;
}

\item{for any $x_* \in \Omega_*$ either a complete alternance exists at $x_*$ or the KKT-pair $(x_*, \mu_*)$
satisfies the second order sufficient optimality condition;
}

\item{the function $\Phi(G(x), \lambda, c)$ admits the second order expansion in $x$ at every point 
$(x_*, \lambda_*)$ such that $x_* \in \Omega_*$ and $\lambda_* \in \Phi_0^{-1}(\mu_*)$.
}
\end{enumerate}
Then for any $\lambda_* \in \Phi_0^{-1}(\mu_*)$ one has $\lambda_* \in \Lambda(\mathcal{P})$ iff there exists
$c_0 > 0$ such that the set $S(\lambda_*, c_0) = \{ x \in A \mid \mathscr{L}(x, \lambda_*, c_0) < f_* \}$ is either
bounded or empty.
\end{theorem}

\begin{proof}
With the use of Theorems~\ref{Thrm_LSPvia2OrderSuffCond} and \ref{Thrm_LSPviaFirstOrderSuffCond} one obtains that
$\lambda_* \in \Lambda_{loc}(\mathcal{P})$. Then applying the localization principle we arrive at the required result.
 
\end{proof}

In the case when a complete alternance exists at every globally optimal solution of $(\mathcal{P})$ one can obtain 
a stronger result.

\begin{theorem} \label{Th_LocPrincipleSP_MiniMaxProblems_Alternance}
Let $A$ be closed, $G$ be continuous on $A$, $\mathscr{L}(\cdot, \lambda, c)$ be l.s.c. on $A$ for any 
$\lambda \in \Lambda$ and $c > 0$. Suppose that assumptions $(A1)$--$(A4)$, $(A6)$ and $(A8)$--$(A11)$ are satisfied.
Let also the following assumptions be valid:
\begin{enumerate}
\item{the functions $f_k$, $k \in M$, and $G$ are twice Fr\'echet differentiable at every $x_* \in \Omega_*$;
} 

\item{the function $\Phi(G(x), \lambda, c)$ admits the second order expansion in $x$ at every point 
$(x_*, \lambda_*) \in \Omega_* \times K^*$ such that $\langle \lambda_*, G(x_*) \rangle = 0$;
}

\item{a complete alternance exists at every $x_* \in \Omega_*$.
} 
\end{enumerate}
Then a global saddle point of $\mathscr{L}(x, \lambda, c)$ exists if and only if
there exist $\mu_* \in K^*$, $\lambda_* \in \Phi_0^{-1}(\mu_*)$ and $c_0 > 0$ such that for any $x_* \in \Omega_*$ 
the pair $(x_*, \mu_*)$ is a KKT-pair of $(\mathcal{P})$, and the set $S(\lambda_*, c_0)$ is either bounded or empty.
Furthermore, $\Lambda(\mathcal{P})$ consists of all those $\lambda_* \in \Lambda$ which satisfy the above assumptions.
\end{theorem}

\begin{proof}
Let $(x_*, \lambda_*)$ be a global saddle point of $\mathscr{L}(x, \lambda, c)$, and $\mu_* = \Phi_0(\lambda_*)$. Then
for any $z_* \in \Omega_*$ the pair $(z_*, \mu_*)$ is a KKT-pair of $(\mathcal{P})$ by
Corollary~\ref{Crlr_IndependenceGSPofOptSol} and Proposition~\ref{Prp_LSPisKKTpair}. Moreover, the set
$S(\lambda_*, c_0)$ is empty for any $c_0 \ge c^*(x_*, \lambda_*)$.

Let, now, $\mu_* \in K^*$ and $\lambda_* \in \Phi_0^{-1}(\mu_*)$ be such that for any $x_* \in \Omega_*$ the pair 
$(x_*, \mu_*)$ is a KKT-pair of $(\mathcal{P})$, and there exists $c_0 > 0$ for which the set
$S(\lambda_*, c_0)$ is either bounded or empty. Since a complete alternance exists at every $x_* \in \Omega_*$, 
$\lambda_* \in \Lambda_{loc}(\mathcal{P})$ by Theorem~\ref{Thrm_LSPviaFirstOrderSuffCond}. Then applying 
the localization principle one obtains the desired result.	 
\end{proof}

\begin{remark}
To the best of author's knowledge, Theorems~\ref{Th_LocPrincipleSP_MiniMaxProblems} and
\ref{Th_LocPrincipleSP_MiniMaxProblems_Alternance} are the first results on the existence of global saddle points of
augmented Lagrangian functions for cone constrained minimax problems.
\end{remark}

Let us demonstrate how one can apply the theorems above to semi-infinite and generalized semi-infinite min-max problems.

\begin{example}
Consider the following semi-infinite programming problem
\begin{equation} \label{SemiInfProblem_Example}
  \min f(x) \quad \text{subject to} \quad g(x, t) \le 0 \quad \forall t \in T,
\end{equation}
where $T$ is a compact metric space. Suppose that the functions $f(\cdot)$ and $g(\cdot, t)$, $t \in T$, are twice
continuously differentiable, and the functions $g(x, t)$, $\nabla_x g(x, t)$ and $\nabla^2_{xx} g(x, t)$ are continuous
(jointly in $x$ and $t$).

Let $x_*$ be a globally optimal solution of problem \eqref{SemiInfProblem_Example}, and $(x_*, \lambda_*)$ be a KKT-pair
of this problem. Then, without loss of generality (see, e.g., \cite{BonnansShapiro}, Lemma~5.110), one can suppose that
the support of the measure $\lambda_*$ consists of at most $d$ points, i.e. $\lambda_*$ has the form 
$\lambda_* = \sum_{i = 1}^l \lambda_i \delta(t_i)$ for some $l \le d$, $\lambda_i \ge 0$ and $t_i \in T$, where
$\lambda_i g(x, t_i) = 0$ for all $i \in I := \{ 1, \ldots, l \}$, and $\delta(t_i)$ denotes the Dirac measure of mass
one at the point $t_i$. Denote $\overline{\lambda} = (\lambda_1, \ldots, \lambda_m)$.

Let, finally, $\Phi(y, \lambda, c)$ be defined as in Example~\ref{Example_NonlinearRescale_SemiInfProg}. Then
\begin{equation} \label{SemiInf_DiscrEqualCont}
  \mathscr{L}(x, \lambda_*, c) = f(x) + \frac{1}{c} \sum_{i = 1}^l \lambda_i \phi( c g(x, t_i) ) =
  \mathcal{L}(x, \overline{\lambda}, c) \quad \forall x \in \mathbb{R}^d,
\end{equation}
where $\mathcal{L}(x, \lambda, c)$ is the exponential penalty function (Example~\ref{Example_ExpPenFunc}) for the
discretized problem
\begin{equation} \label{SemiInfDiscretized}
  \min f(x) \quad \text{subject to} \quad g_i(x) = g(x, t_i) \le 0 \quad \forall i \in I.
\end{equation}
With the use of \eqref{SemiInf_DiscrEqualCont} one can easily check that $(x_*, \lambda_*)$ is a global saddle point of
$\mathscr{L}(x, \lambda, c)$ iff $(x_*, \overline{\lambda})$ is a global saddle point of $\mathcal{L}(x, \lambda, c)$.
Therefore applying Proposition~\ref{Prp_SaddlePointImpliesOptSol} one obtains that for the pair $(x_*, \lambda_*)$ to
be a global saddle point of $\mathscr{L}(x, \lambda, c)$ it is necesssary that $x_*$ is a globally optimal solution of
the discretized problem \eqref{SemiInfDiscretized}.

One can apply Theorem~\ref{Th_LocPrincipleSP_MiniMaxProblems} in order to obtain necessary and sufficient
conditions for the existence of a global saddle point of the exponential penalty function $\mathcal{L}(x, \lambda, c)$
for the discretized problem \eqref{SemiInfDiscretized}, which, in turn, can be used as necessary and sufficient
conditions for the existence of a global saddle point of the augmented Lagrangian $\mathscr{L}(x, \lambda, c)$.
However, this approach forces one to use sufficient optimality conditions for the discretized problem
\eqref{SemiInfDiscretized} that are unnatural for semi-infinite programming problems due to the absence of the sigma
term $\sigma(\lambda_*, \mathcal{T}(h))$ (see~\cite{BonnansShapiro}, Section~5.4.3). Let us note that this drawback is
common for all existing results on global saddle points of augmented Lagrangian functions for semi-infinite programming
problems (cf. \cite{RuckmannShapiro,BurachikYangZhou2017}).
\end{example}

\begin{remark} \label{Remark_AugmLagr_SemiInfProblems}
It should be noted that all existing augmented Lagrangian functions for semi-infinite programming problems are
constructed as a straightforward generalization of corresponding augmented Lagrangian functions for mathematical
programming problems. This approach leads to unsatisfactory results, since one has to utilize unnatural optimality
conditions in order to prove the existence of global or local saddle points of augmented Lagrangian functions for
semi-infinite programming problems. Clearly, a completely different approach to the construction of augmented Lagrangian
functions for these problems is needed. The search of such an approach is a challenging topic of future research.
\end{remark}

\begin{example}
Consider the following generalized semi-infinite min-max problem
\begin{equation} \label{GenSemiInf_Prob}
  \min_{x \in \mathbb{R}^d} \max_{z \in Z(x)} f(x, z), \quad
  Z(x) = \big\{ z \in A \bigm| G(x, z) \in K \big\},
\end{equation}
where $A \subset \mathbb{R}^l$ is a nonempty set, while $f \colon \mathbb{R}^{d + l} \to \mathbb{R}$ and 
$G \colon \mathbb{R}^{d + l} \to Y$ are given functions. Denote $f_0(x) = \max_{z \in Z(x)} f(x, z)$, 
$Z^*(x) = \{ z \in Z(x) \mid f_0(x) = f(x, z) \}$, and introduce the following lower level cone constrained
optimization problem
\begin{equation} \label{GenSemiInf_AuxProb}
  \min_{z \in \mathbb{R}^l} (- f_x(z)) \quad \text{subject to} \quad G_x(z) \in K, \quad z \in A,
\end{equation}
where $f_x(z) = f(x, z)$ and $G_x(z) = G(x, z)$. Clearly, the set of globally optimal solutions of this problem coincide
with $Z^*(x)$.

Being inspired by the ideas of \cite{PolakRoyset2005}, define
$$
  h(x, \lambda, c) = \sup_{z \in A} \Big( f(x, z) - \Phi(G(x, z), \lambda, c) \Big)
  \quad \forall x \in \mathbb{R}^d, \: \lambda \in Y^*, \: c > 0,
$$
Let $\mathscr{L}_x(z, \lambda, c)$ be the augmented Lagrangian function for problem \eqref{GenSemiInf_AuxProb}. Then
$h(x, \lambda, c) = - \inf_{z \in A} \mathscr{L}_x (z, \lambda, c)$, i.e. $h(x, \lambda, c)$ is the negative of the
objective function of the augmented dual problem of problem \eqref{GenSemiInf_AuxProb}. Therefore applying
Proposition~\ref{Prp_SaddlePointsVSOptSolDualProb} one obtains that the following result holds true.

\begin{proposition} \label{Prp_GenSemiInf_WeakDuality}
Suppose that assumption $(A2)$ holds true. Then
$$
  h(x, \lambda, c) \ge f_0(x) \quad \forall x \in \mathbb{R}^d, \: \lambda \in \Lambda, \: c > 0.
$$
Suppose, additionally, that assumptions $(A1)$--$(A4)$ hold true and $Z^*(x) \ne \emptyset$. 
Then $h(x, \lambda, c) = f_0(x)$ for some $\lambda \in \Lambda$ and $c > 0$ iff there exists $z \in Z^*(x)$ such that
$(z, \lambda)$ is a global saddle point of $\mathscr{L}_x(z, \lambda, c)$, and $c \ge c^*(z, \lambda)$.
\end{proposition}

The proposition above motivates us to consider the following auxiliary minimax problem
\begin{equation} \label{GenSemiInf_MinMaxProb}
  \min_{(x, \lambda, c)} h(x, \lambda, c) \quad \text{subject to} \quad \lambda \in \Lambda, \quad c > 0.
\end{equation}
Let $x_*$ be a locally optimal solution of problems \eqref{GenSemiInf_Prob} such that $f_0(x) > f_0(x_*)$ for any
$x \in U \setminus \{ x_* \}$, where $U$ is a neighbourhood of $x_*$. From Proposition~\ref{Prp_GenSemiInf_WeakDuality}
it follows that if $h(x_*, \lambda_*, c_*) = f_0(x_*)$ for some $\lambda_* \in \Lambda$ and $c_* > 0$, then
$(x_*, \lambda_*, c_*)$ is a locally optimal solution of problem \eqref{GenSemiInf_MinMaxProb} such that
$h(x, \lambda, c) > h(x_*, \lambda_*, c_*)$ for any $x \in U \setminus \{ x_* \}$, $\lambda \in \Lambda$ and $c > 0$.
With the use of this result one can easily obtain necessary and sufficient conditions for the sets of all locally (and
hence globally) optimal solutions of problems \eqref{GenSemiInf_Prob} and \eqref{GenSemiInf_MinMaxProb} to coincide.
Here, we only provide such conditions for the case of globally optimal solutions.

We say that problems \eqref{GenSemiInf_Prob} and \eqref{GenSemiInf_MinMaxProb} are \textit{equivalent} if their optimal
values coincide, and $x_*$ is a globally optimal solution of problem \eqref{GenSemiInf_Prob} iff $(x_*, \lambda_*, c_*)$
with some $\lambda_* \in \Lambda$ and $c_* > 0$ is a globally optimal solution of problem
\eqref{GenSemiInf_MinMaxProb}. Applying Proposition \ref{Prp_GenSemiInf_WeakDuality} one can easily obtain the following
result.

\begin{proposition}
Suppose that assumptions $(A1)$--$(A4)$ hold true, and $Z^*(x_*) \ne \emptyset$ for any globally optimal
solution $x_*$ of problem \eqref{GenSemiInf_Prob}. Then problems \eqref{GenSemiInf_Prob} and
\eqref{GenSemiInf_MinMaxProb} are equivalent iff for any globally optimal solution $x_*$ of 
problem \eqref{GenSemiInf_Prob} there exists a global saddle point of 
the augmented Lagrangian $\mathscr{L}_{x_*} (z, \lambda, c)$.
\end{proposition}

With the use of the proposition above and Theorems~\ref{Th_LocPrincipleSP_MiniMaxProblems} and
\ref{Th_LocPrincipleSP_MiniMaxProblems_Alternance} (or the localization principle) one can obtain simple necessary and
sufficient conditions for the equivalence of problems \eqref{GenSemiInf_Prob} and \eqref{GenSemiInf_MinMaxProb} in terms
of sufficient optimality conditions for lower level problem \eqref{GenSemiInf_AuxProb}. The interested reader can
easily formulate these conditions, which unify and significantly generalize some existing results
(\cite{PolakRoyset2005}, Thrm.~3.1; \cite{WangZhouXu2009}, Thrm.~4.5).
\end{example}

\section{Exact Augmented Lagrangian Functions}
\label{Section_ExactAL}

Being inspired by the ideas of Di Pillo, Grippo and Lucidi
\cite{DiPilloGrippo1979,DiPilloGrippo1980,DiPilloGrippo1982,Lucidi1988,DiPilloLucidi1996,DiPilloLucidi2001,
DiPilloEtAl2002,DiPilloLiuzzi2003} (see also \cite{DuZhangGao2006,DuLiangZhang2006,LuoWuLiu2013,
DiPilloLucidiPalagi1993,DiPilloGrippoLucidi1993,DiPilloLucidiPalagi2000,DiPilloLucidi2005,DiPilloLiuzzi,
DiPilloLiuzzi2011}), in this section, we present a general method for constructing exact augmented Lagrangian
functions, and obtain simple sufficient (and necessary) conditions for the exactness of these functions. In particular,
we demonstrate that one can easily extend the localization principle to the case of exact augmented Lagrangian
functions, thus showing that the study of the exactness of an augmented Lagrangian function can be easily reduced to a
local analysis of sufficient optimality conditions. 

Introduce the penalized augmented Lagrangian function
$$
  \mathscr{L}_e(x, \lambda, c) = f(x) + \Phi(G(x), \lambda, c) + \eta(x, \lambda) =
  \mathscr{L}(x, \lambda, c) + \eta(x, \lambda),
$$
where $\eta \colon X \times Y^* \to [0, + \infty]$ is a given non-negative function (the subscript ``e'' stands for
``exact''). The function $\eta$ must be defined in such a way that it penalizes the violation of the KKT optimality
conditions. However, at first, we suppose that $\eta$ is an arbitrary non-negative function. 

Our aim is to show that under some additional assumptions the penalized augmented Lagrangian 
$\mathscr{L}_e(x, \lambda, c)$ is \textit{exact}, i.e. that all points of global minimum $(x_*, \lambda_*)$ of
$\mathscr{L}_e(x, \lambda, c)$ in $(x, \lambda)$ on the set $A \times \Lambda$ are exactly KKT-pairs of the problem
$(\mathcal{P})$ corresponding to globally optimal solutions of this problem. Note that unlike the case of augmented
Lagrangian functions studied in the previous sections, one must \textit{simultaneously} minimize the penalized augmented
Lagrangian $\mathscr{L}_e(x, \lambda, c)$ \textit{both} in primal variable $x$ and in dual variable $\lambda$ in order
to recover optimal solutions of the original problem.

We start by studying the behaviour of global minimizers of $\mathscr{L}_e(x, \lambda, c)$ in $(x, \lambda)$ as 
the penalty parameter $c$ increases unboundedly. To this end, we need to introduce a stronger version of 
assumption $(A6)$.
\begin{itemize}
\item[$(A6)_s$]{$\forall y \notin K$ $\forall \lambda \in \Lambda$ $\forall c_0 > 0$ $\exists r > 0$ such that
\begin{multline*}
  \lim_{c \to \infty} \inf\Big\{ \Phi(z, \mu, c) - \Phi(z, \mu, c_0) \Bigm| \\ 
  z \in B(y, r), \: \mu \in B(\lambda, r) \cap \Lambda \colon \Phi(z, \mu, c_0) < + \infty \Big\} = + \infty;
\end{multline*}
}
\end{itemize}
\vspace{-4mm}
One can verify that this assumption is satisfied in Example~\ref{Example_RockafellarWetsAL}, provided $\sigma$ has a
valley at zero, and hence it is valid in Examples~\ref{Example_RockWetsAL_SOC}, \ref{Example_RockWetsAL_SemiDefProg} and
\ref{Example_RockafellarWetsAL_SemiInfProg}. This assumption is also valid in 
Examples~\ref{Example_EssentiallyQuadraticAL}--\ref{Example_Mangasarian},
\ref{Example_PenalizedExpPenFunc}, \ref{Example_ModBarrierFunc}, \ref{Example_HeWuMengLagrangian} and
\ref{Example_NonlinearRescalePenalized_SemiDefProg} in the general case, and in
Examples~\ref{Example_NonlinearRescale_SOC}, \ref{Example_NonlinearRescale_SemiDefProg} and
\ref{Example_NonlinearRescale_SemiInfProg} in the case when $\varepsilon_0 < + \infty$. Finally, assumption $(A6)_s$ (as
well as $(A6)$) is never satisfied in Examples~\ref{Example_ExpPenFunc} and \ref{Example_pthPowerAugmLagr}.

The following result extends Lemma~\ref{Lemma_MinimizingSeq} to the case of the penalized augmented Lagrangian
function $\mathscr{L}_e(x, \lambda, c)$.

\begin{lemma} \label{Lemma_EAL_MinimizingSeq}
Let $A$ be closed, $G$ be continuous on $A$, $\mathscr{L}_e(\cdot, \cdot, c)$ be l.s.c. on $A \times \Lambda$ for 
all $c > 0$, and $\mathscr{L}_e(\cdot, \cdot, c_0)$ be bounded from below on $A \times \Lambda$ for some $c_0 > 0$.
Suppose also that assumptions $(A2)$, $(A4)$, $(A6)_s$ and $(A12)$ are satisfied, and there exists 
$(\overline{x}, \overline{\lambda}) \in \Omega_* \times \Lambda$ such 
that $\eta(\overline{x}, \overline{\lambda}) = 0$. Let, finally, a pair $(x_*, \lambda_*)$ be a cluster point of a
sequence $\{ (x_n, \lambda_n) \} \subset A \times \Lambda$ such that
$$
  \mathscr{L}_e(x_n, \lambda_n, c_n) \le \inf_{(x, \lambda) \in A \times \Lambda} \mathscr{L}_e(x, \lambda, c_n) +
  \varepsilon_n \quad \forall n \in \mathbb{N},
$$
where $\{ c_n \} \subset [c_0, + \infty)$ is an increasing unbounded sequence, and the sequence
$\{ \varepsilon_n \} \subset (0, + \infty)$ is such that $\varepsilon_n \to 0$ as $n \to \infty$. Then $x_*$
is a globally optimal solution of $(\mathcal{P})$ and $\eta(x_*, \lambda_*) = 0$.
\end{lemma}

\begin{proof}
Replacing, if necessary, the sequences $\{ (x_n, \lambda_n) \}$, $\{ c_n \}$ and $\{ \varepsilon_n \}$ with their
subsequences, one can suppose that $(x_*, \lambda_*)$ is a limit point of the sequence $\{ (x_n, \lambda_n) \}$.
Furthermore, note that $(x_*, \lambda_*) \in A \times \Lambda$ by virtue of the fact that $A$ and $\Lambda$ are closed.

Let us verify, at first, that $x_*$ is a feasible point of the problem $(\mathcal{P})$. Indeed, with the use of $(A2)$
one gets that $\mathscr{L}_e(\overline{x}, \overline{\lambda}, c) \le f(\overline{x}) < + \infty$ for all $c > 0$.
Therefore for any $n \in \mathbb{N}$ one has
\begin{equation} \label{EAL_MinimSeq_UpperBound}
  \mathscr{L}_e(x_n, \lambda_n, c_n) \le 
  \inf_{(x, \lambda) \in A \times \Lambda} \mathscr{L}_e(x, \lambda, c_n) + \varepsilon_n \le
  f(\overline{x}) + \varepsilon_n.
\end{equation}
Hence, in particular, $f(x_n) < + \infty$ and $\Phi(G(x_n), \lambda_n, c_n) < + \infty$ for all $n \in \mathbb{N}$.
Furthermore, one has
\begin{equation} \label{EAL_FiniteLimitValue}
  \limsup_{n \to \infty} \mathscr{L}_e(x_n, \lambda_n, c_n) \le f(\overline{x}) < + \infty.
\end{equation}
Arguing by reductio ad absurdum, suppose that $G(x_*) \notin K$. Then by $(A6)_s$ there exists $r > 0$ such that
\begin{multline} \label{ExAugmLagr_AssumpA6strong}
  \lim_{c \to \infty} \inf\Big\{ \Phi(y, y^*, c) - \Phi(y, y^*, c_0) \Bigm| \\
  y \in B(G(x_*), r), \: y^* \in B(\lambda_*, r) \cap \Lambda \colon \Phi(y, y_*, c_0) < + \infty \Big\} = + \infty.
\end{multline}
From the facts that $G$ is continuous on $A$, and $(x_*, \lambda_*)$ is a limit point of 
the sequence $\{ (x_n, \lambda_n) \}$ it follows that there exists $n_0 \in \mathbb{N}$ such that $G(x_n) \in B(G(x_*),
r)$ and $\lambda_n \in B(\lambda_*, r)$ for all $n \ge n_0$. Consequently, for any $n \ge n_0$ one has
\begin{multline*}
  \mathscr{L}_e(x_n, \lambda_n, c_n) = 
  \mathscr{L}_e(x_n, \lambda_n, c_0) + \Phi(G(x_n), \lambda_n, c_n) - \Phi(G(x_n), \lambda_n, c_0) \\
  \ge \gamma + \inf\Big\{ \Phi(y, y^*, c) - \Phi(y, y^*, c_0) \Bigm| \\
  y \in B(G(x_*), r), \: y^* \in B(\lambda_*, r) \cap \Lambda \colon \Phi(y, y_*, c_0) < + \infty \Big\},
\end{multline*}
where $\gamma = \inf\{ \mathscr{L}_e(x, \lambda, c_0) \mid (x, \lambda) \in A \times \Lambda \} > - \infty$.
Consequently, applying \eqref{ExAugmLagr_AssumpA6strong} one obtains that 
$\mathscr{L}_e(x_n, \lambda_n, c_n) \to + \infty$ as $n \to \infty$, which contradicts \eqref{EAL_FiniteLimitValue}.
Thus, $G(x_*) \in K$.

Observe that from $(A4)$ and \eqref{EAL_MinimSeq_UpperBound} it follows that for any $c > 0$ there exists 
$n_0 \in \mathbb{N}$ such that $\mathscr{L}_e(x_n, \lambda_n, c) \le f(\overline{x}) + \varepsilon_n$ for 
all $n \ge n_0$. Passing to the limit inferior as $n \to \infty$, and taking into account the fact that
$\mathscr{L}_e(x, \lambda, c)$ is l.s.c. in $(x, \lambda)$ one obtains that
$\mathscr{L}_e(x_*, \lambda_*, c) \le f(\overline{x})$ for all $c > 0$.
Taking into account the facts that $G(x_*) \in K$ and $\lambda_* \in \Lambda$, and passing to the limit as 
$c \to \infty$ with the use of $(A12)$ one gets that $f(x_*) + \eta(x_*, \lambda_*) \le f(\overline{x})$. Therefore 
$x_* \in \Omega_*$ and  $\eta(x_*, \lambda_*) = 0$ due to the fact that $x_*$ is feasible, $\overline{x} \in \Omega_*$,
and $\eta$ is a non-negative function.	 
\end{proof}

By analogy with the theory of exact penalty functions \cite{Dolgopolik}, let us introduce the definition of locally and
globally exact augmented Lagrangian functions.

\begin{definition}
Let $x_*$ be a locally optimal solution of $(\mathcal{P})$, and $\lambda_* \in \Lambda$ be such 
that $\eta(x_*, \lambda_*) = 0$. The penalized augmented Lagrangian function $\mathscr{L}_e(x, \lambda, c)$ is called
(\textit{locally}) \textit{exact} at $(x_*, \lambda_*)$ (with respect to the function $\eta$) if there exist $c_0 > 0$
and a neighbourhood $U$ of $(x_*, \lambda_*)$ such that 
$$
  \mathscr{L}_e(x, \lambda, c) \ge \mathscr{L}_e(x_*, \lambda_*, c) \quad 
  \forall (x, \lambda) \in U \cap (A \times \Lambda) \quad \forall c \ge c_0.
$$
\end{definition}

Note that if assumption $(A4)$ holds true, then $\mathscr{L}_e(x, \lambda, c)$ is locally exact 
at $(x_*, \lambda_*)$ iff there exists $c_0 > 0$ such that the pair $(x_*, \lambda_*)$ is a point of local minimum of 
$\mathscr{L}_e(\cdot, \cdot, c_0)$ on the set $A \times \Lambda$. 

\begin{definition}
The penalized augmented Lagrangian function $\mathscr{L}_e(x, \lambda, c)$ is called (\textit{globally}) \textit{exact}
(with respect to the function $\eta$) if there exists $c_0 > 0$ such that for any $c \ge c_0$ the function 
$\mathscr{L}_e(\cdot, \cdot, c)$ attains a global minimum on $A \times \Lambda$, and 
$(x_*, \lambda_*) \in \argmin_{(x, \lambda) \in A \times \Lambda} \mathscr{L}_e(x, \lambda, c)$ if and only if $x_*$ is
a globally optimal solution of $(\mathcal{P})$ and $\eta(x_*, \lambda_*) = 0$.
\end{definition}

Thus, if the augmented Lagrangian $\mathscr{L}_e(x, \lambda, c)$ is globally exact, then the problem of minimizing  
$\mathscr{L}_e(x, \lambda, c)$ in $(x, \lambda)$ over the set $A \times \Lambda$ is equivalent to the original problem
$(\mathcal{P})$ for any sufficiently large value of the penalty parameter $c$, since in this case points of global
minimum of $\mathscr{L}_e(x, \lambda, c)$ on $A \times \Lambda$ are exactly those pairs $(x_*, \lambda_*)$ for which 
$x_* \in \Omega_*$ and  $\eta(x_*, \lambda_*) = 0$. In particular, in the case when $\eta(x, \lambda) = 0$ iff 
$(x, \lambda)$ is a KKT-pair, points of global minimum of $\mathscr{L}_e(x, \lambda, c)$ on $A \times \Lambda$ coincide
with KKT-pairs of the problem $(\mathcal{P})$ corresponding to globally optimal solutions of this problem.

Our aim is to show that (under some additional assumptions) the augmented Lagrangian $\mathscr{L}_e(x, \lambda, c)$ is
exact if and only if it is locally exact at every pair $(x_*, \lambda_*)$ such that $x_* \in \Omega_*$ and 
$\eta(x_*, \lambda_*) = 0$. In other words, our aim is to prove the validity of the localization principle for 
the penalized augmented Lagrangian function $\mathscr{L}_e(x, \lambda, c)$. As in the case of the localization principle
for global saddle points, the localization principle for the penalized augmented Lagrangian 
$\mathscr{L}_e(x, \lambda, c)$ allows one to study a local behaviour of $\mathscr{L}_e(x, \lambda, c)$ near globally
optimal solutions of the problem $(\mathcal{P})$ in order to prove the global exactness of this function. 

We need to introduce a stronger version of assumption $(A4)$. We say that a non-decreasing function 
$h \colon (0, + \infty) \to \mathbb{R} \cup \{ + \infty \}$ is strictly increasing at a point $t > 0$ such that 
$h(t) < + \infty$, if $h(\tau) > h(t)$ for any $\tau > t$.
\begin{itemize}
\item[$(A4)_s$]{assumption $(A4)$ holds true, and $\forall y \in Y$ $\forall \lambda \in \Lambda$ $\forall c > 0$ such
that $\Phi(y, \lambda, c) < + \infty$ either the function $\Phi(y, \lambda, \cdot)$ is strictly increasing at $c$ or
$\Phi(y, \lambda, c) = 0$ and $y \in K$.
}
\end{itemize}
Note that this assumption is satisfied in Example~\ref{Example_RockafellarWetsAL}, if the infimum in the definition of
$\Phi(y, \lambda, c)$ is attained for all $y \in Y$, $\lambda \in \Lambda$ and $c > 0$. Hence, in particular,
assumption $(A4)_s$ is satisfied in Examples~\ref{Example_RockWetsAL_SOC} and \ref{Example_RockWetsAL_SemiDefProg}.
This assumption is always valid in Examples~\ref{Example_EssentiallyQuadraticAL}--\ref{Example_Mangasarian} and
\ref{Example_HeWuMengLagrangian}, and it is valid in Examples~\ref{Example_PenalizedExpPenFunc} and
\ref{Example_NonlinearRescalePenalized_SemiDefProg}, provided $\phi$ is strictly convex, and $\xi$ is strictly convex 
on $\mathbb{R}_+$. Finally, observe that assumption $(A4)_s$ (unlike $(A4)$) is never satisfied in
Examples~\ref{Example_ExpPenFunc}, \ref{Example_ModBarrierFunc}, \ref{Example_pthPowerAugmLagr},
\ref{Example_NonlinearRescale_SOC}, \ref{Example_NonlinearRescale_SemiDefProg} and
\ref{Example_NonlinearRescale_SemiInfProg}.

Before we proceed to the localization principle, let us point out that instead of verifying that all pairs 
$(x_*, \lambda_*)$ with $x_* \in \Omega_*$ and $\eta(x_*, \lambda_*) = 0$ are global minimizers of 
$\mathscr{L}_e(x, \lambda, c)$ on $A \times \Lambda$, it is sufficient to check that at least one such pair is a point
of global minimum of $\mathscr{L}_e(x, \lambda, c)$ on $A \times \Lambda$ in order to prove 
that $\mathscr{L}_e(x, \lambda, c)$ is globally exact.

\begin{lemma} \label{Lemma_GlobalExactEquivFormulation}
Let assumption $(A4)_s$ be valid, and suppose that 
\begin{equation} \label{AugmFuncVanishAtKKTpoints}
  \Phi(G(x_*), \lambda_*, c) = 0 \quad \forall c > 0 \quad 
  \forall (x_*, \lambda_*) \in \Omega_* \times \Lambda \colon \eta(x_*, \lambda_*) = 0.
\end{equation}
Then $\mathscr{L}_e(x, \lambda, c)$ is globally exact if and only if there exist 
$(x_0, \lambda_0) \in \Omega_* \times \Lambda$ and $c_0 > 0$ such that $\eta(x_0, \lambda_0) = 0$ and 
the pair $(x_0, \lambda_0)$ is a point of global minimum of $\mathscr{L}_e(\cdot, \cdot, c_0)$ on $A \times \Lambda$.
\end{lemma}

\begin{proof}
The validity of the ``only if'' part of the proposition follows directly from the definition of global exactness.
Therefore, let us prove the ``if'' part of the proposition.

By the definition of $(x_0, \lambda_0)$ and assumption \eqref{AugmFuncVanishAtKKTpoints} one has 
$$
  \min_{(x, \lambda) \in A \times \Lambda} \mathscr{L}_e(x, \lambda, c_0) = 
  \mathscr{L}_e (x_0, \lambda_0, c_0) = f(x_0) = f_*.
$$
Note that $(A4)_s$ implies that the function $\mathscr{L}_e(x, \lambda, c)$ is non-decreasing in $c$, while
\eqref{AugmFuncVanishAtKKTpoints} implies that $\mathscr{L}_e(x_*, \lambda_*, c) = f(x_*) = f_*$ for all 
$c > 0$ and for any $(x_*, \lambda_*) \in \Omega_* \times \Lambda$ such that $\eta(x_*, \lambda_*) = 0$. Therefore any
such pair $(x_*, \lambda_*)$ is a point of global minimum of $\mathscr{L}_e(\cdot, \cdot, c)$ on $A \times \Lambda$
and $\min_{(x, \lambda) \in A \times \Lambda} \mathscr{L}_e(x, \lambda, c) = f_*$ for all $c \ge c_0$.

Let, now, $(x_*, \lambda_*)$ be a point of global minimum of $\mathscr{L}_e(\cdot, \cdot, c_1)$ on $A \times \Lambda$
for some $c_1 > c_0$. Since $\mathscr{L}(x, \lambda, c)$ is non-decreasing in $c$,
$\mathscr{L}_e(x_*, \lambda_*, c) = f_*$ for all $c \in [c_0, c_1]$ (recall that 
$\mathscr{L}(\cdot, \cdot, c) \ge f_*$ for all $c \ge c_0$). Hence with the use of $(A4)_s$ one gets that
$G(x_*) \in K$ and $\Phi(G(x_*), \lambda_*, c) = 0$ for all $c \in [c_0, c_1)$. Therefore 
$\mathscr{L}_e(x_*, \lambda_*, c) = f(x_*) + \eta(x_*, \lambda_*) = f_*$ for any $c \in [c_0, c_1)$, which implies that 
$x_* \in \Omega_*$ and $\eta(x_*, \lambda_*) = 0$ due to the fact that $x_*$ is feasible and $\eta(\cdot)$ is
non-negative.

Thus, for any $c > c_0$ the function $\mathscr{L}_e(\cdot, \cdot, c)$ attains a global minimum on $A \times \Lambda$ (at
the point $(x_0, \lambda_0)$), and 
$(x_*, \lambda_*) \in \argmin_{(x, \lambda) \in A \times \Lambda} \mathscr{L}_e(x, \lambda, c)$ iff
$x_* \in \Omega_*$ and $\eta(x_*, \lambda_*) = 0$. In other words, $\mathscr{L}_e(x, \lambda, c)$ is globally exact.
 
\end{proof}

Now, we can extend the localization principle to the case of the penalized augmented Lagrangian function
$\mathscr{L}_e(x, \lambda, c)$. Let us note that the localization principle holds true only in the finite dimensional
case (cf.~\cite{Dolgopolik}, Examples~3--5 and \cite{Dolgopolik_AugmLagrMult}, Example~4). Therefore, hereinafter, we
must suppose that the space $Y$ is finite dimensional.

\begin{theorem}[Localization principle] \label{Thrm_EAL_LocalizationPrinciple}
Let $Y$ be finite dimensional, $A$ be closed, $G$ be continuous on $A$ and $\mathscr{L}_e(\cdot, \cdot, c)$ be l.s.c. on
$A \times \Lambda$ for all $c > 0$. Suppose also that assumptions $(A2)$, $(A4)_s$, $(A6)_s$ and $(A12)$ are satisfied,
and
\begin{equation} \label{AugmFuncVanishAtKKTpoints_LocPrinciple}
  \Phi(G(x_*), \lambda_*, c) = 0 \quad \forall c > 0 \quad 
  \forall (x_*, \lambda_*) \in \Omega_* \times \Lambda \colon \eta(x_*, \lambda_*) = 0.
\end{equation}
Then $\mathscr{L}_e(x, \lambda, c)$ is globally exact (with respect to the function $\eta$) if and only if
\begin{enumerate}
\item{there exists $(\overline{x}, \overline{\lambda}) \in \Omega_* \times \Lambda$ such 
that $\eta(\overline{x}, \overline{\lambda}) = 0$;
}

\item{$\mathscr{L}_e(x, \lambda, c)$ is locally exact at every $(x_*, \lambda_*) \in \Omega_* \times \Lambda$
such that $\eta(x_*, \lambda_*) = 0$;
} 

\item{there exists $c_0 > 0$ such that the set
$$
  S_e(c_0) := \Big\{ (x, \lambda) \in A \times \Lambda \Bigm| \mathscr{L}_e(x, \lambda, c_0) < f_* \Big\}
$$
is either bounded or empty.
}
\end{enumerate}
\end{theorem}

\begin{proof}
Let $\mathscr{L}_e(x, \lambda, c)$ be globally exact. Then there exists $c_0 > 0$ such that for all $c \ge c_0$ 
the function $\mathscr{L}_e(\cdot, \cdot, c)$ attains a global minimum on $A \times \Lambda$ and 
$(x_*, \lambda_*) \in \argmin_{(x, \lambda) \in A \times \Lambda} \mathscr{L}_e(x, \lambda, c)$ iff $x_* \in \Omega_*$
and $\eta(x_*, \lambda_*) = 0$. Therefore, in particular, there exists a pair 
$(x_*, \lambda_*) \in \Omega_* \times \Lambda$ such that $\eta(x_*, \lambda_*) = 0$ 
(otherwise, $\mathscr{L}_e(x, \lambda, c)$ would not attain a global minimum), and $\mathscr{L}_e(x, \lambda, c)$ is
locally exact at any such pair $(x_*, \lambda_*)$, since any such $(x_*, \lambda_*)$ is a global
minimizer of $\mathscr{L}_e(\cdot, \cdot, c)$ on $A \times \Lambda$ for all $c \ge c_0$. Furthermore, applying
\eqref{AugmFuncVanishAtKKTpoints_LocPrinciple} one gets that 
$\min_{(x, \lambda) \in A \times \Lambda} \mathscr{L}_e(x, \lambda, c) = f_*$ for all $c \ge c_0$, which implies that
$S_e(c_0) = \emptyset$.

Let us prove the converse statement. If $S_e(c_0) = \emptyset$, then applying
\eqref{AugmFuncVanishAtKKTpoints_LocPrinciple} one obtains that a pair 
$(\overline{x}, \overline{\lambda}) \in \Omega_* \times \Lambda$ such that $\eta(\overline{x}, \overline{\lambda}) = 0$
(that exists by our assumption) is a point of global minimum of $\mathscr{L}_e(\cdot, \cdot, c_0)$ on 
$A \times \Lambda$. Consequently, $\mathscr{L}_e(x, \lambda, c)$ is globally exact by
Lemma~\ref{Lemma_GlobalExactEquivFormulation}.

Note that the function $c \to \mathscr{L}_e(x, \lambda, c)$ is non-decreasing by $(A4)_s$. Therefore it remains to
consider the case when $S_e(c) \ne \emptyset$ for all $c > 0$. Choose an increasing unbounded 
sequence $\{ c_n \} \subset [c_0, + \infty)$. Assumption $(A4)_s$ implies that $S_e(c_n) \subseteq S_e(c_0)$ for
all $n \in \mathbb{N}$. Hence taking into account the facts that $S_e(c_0)$ is bounded, and 
$\mathscr{L}_e(\cdot, \cdot, c)$ is l.s.c., one obtains that for any $n \in \mathbb{N}$ there exists 
$(x_n, \lambda_n) \in \argmin_{(x, \lambda) \in A \times \Lambda} \mathscr{L}_e(x, \lambda, c_n)$, and, moreover, 
$(x_n, \lambda_n) \in S_e(c_n) \subseteq S_e(c_0)$, which implies that $\{ (x_n, \lambda_n) \}$ is a bounded sequence.
Consequently, without loss of generality one can suppose that it converges to some point $(x_*, \lambda_*)$ that belongs
to $A \times \Lambda$, since this set is closed. 

Applying Lemma~\ref{Lemma_EAL_MinimizingSeq} one gets that $x_* \in \Omega_*$ and $\eta(x_*, \lambda_*) = 0$, which
implies that $\mathscr{L}_e(x, \lambda, c)$ is locally exact at $(x_*, \lambda_*)$. Therefore there exist 
$\widehat{c} > 0$ and a neighbourhood $U$ of $(x_*, \lambda_*)$ such that
$$
  \mathscr{L}_e(x, \lambda, c) \ge \mathscr{L}_e(x_*, \lambda_*, c) 
  \quad \forall (x, \lambda) \in U \cap \big( A \times \Lambda \big) \quad \forall c \ge \widehat{c}.
$$
From the facts that $(x_n, \lambda_n) \to (x_*, \lambda_*)$ and $c_n \to + \infty$ as $n \to \infty$ it follows that
there exists $n_0 \in \mathbb{N}$ such that $(x_n, \lambda_n) \in U$ and $c_n \ge \widehat{c}$ for all $n \ge n_0$.
Hence for any such $n \in \mathbb{N}$ one has 
$\mathscr{L}_e(x_n, \lambda_n, c_n) \ge \mathscr{L}_e(x_*, \lambda_*, c_n)$. Consequently, taking into account the
definition of $(x_n, \lambda_n)$ one obtains that $(x_*, \lambda_*)$ is a point of global minimum of 
$\mathscr{L}_e(x, \lambda, c_n)$ on $A \times \Lambda$. Applying
\eqref{AugmFuncVanishAtKKTpoints_LocPrinciple} one gets that 
$\min_{(x, \lambda) \in A \times \Lambda} \mathscr{L}_e(x, \lambda, c_n) = \mathscr{L}_e(x_*, \lambda_*, c_n) = 
f(x_*) = f_*$, which contradicts the assumption that $S_e(c) \ne \emptyset$ for all $c > 0$.	 
\end{proof}

\begin{remark}
{(i)~Recall that for the existence of a global saddle point of the augmented Lagrangian function 
$\mathscr{L}(x, \lambda, c)$ it is necessary that there exists $\lambda_* \in K^*$ such that for any globally optimal
solution $x_*$ of the problem $(\mathcal{P})$ the pair $(x_*, \lambda_*)$ is a KKT-pair of this problem. In particular,
a global saddle point of $\mathscr{L}(x, \lambda, c)$ \textit{cannot} exist if the problem $(\mathcal{P})$ has two
globally optimal solution with disjoint sets of Lagrange multipliers. In contrast, the penalized augmented Lagrangian 
function $\mathscr{L}_e(x, \lambda, c)$ \textit{can} be globally exact even if the problem $(\mathcal{P})$ has two (or
more) globally optimal solutions with disjoint sets of Lagrange multipliers.
}

\noindent{(ii)~Note that the theorem above contains some existing results as simple particular cases
(see, e.g., \cite{LuoWuLiu2013}, Thms.~4.3 and 4.4).
}
\end{remark}

Note that, as in the case of the localization principle for global saddle point of the augmented Lagrangian
$\mathscr{L}(x, \lambda, c)$, one can reformulate the assumption on the boundedness of the set $S_e(c_0)$ from the
theorem above in terms of behaviour of points of global minimum of $\mathscr{L}(\cdot, \cdot, c)$ as $c \to \infty$.
Namely, arguing in a similar way to the proof of Proposition~\ref{Prp_GSP_LP_Nondegeneracy} one can verify that the
following result holds true.

\begin{proposition}
Let $Y$ be finite dimensional, $A$ be closed, $G$ be continuous on $A$, $\mathscr{L}_e(\cdot, \cdot, c)$ be l.s.c.
on $A \times \Lambda$ for all $c > 0$, and $\mathscr{L}_e(x, \lambda, c)$ be locally exact at every point 
$(x_*, \lambda_*) \in \Omega_* \times \Lambda$ such that $\eta(x_*, \lambda_*) = 0$. Suppose also that assumptions
$(A2)$, $(A4)_s$, $(A6)_s$ and $(A12)$ are satisfied, there exists 
$(\overline{x}, \overline{\lambda}) \in \Omega_* \times \Lambda$ such that 
$\eta(\overline{x}, \overline{\lambda}) = 0$, and \eqref{AugmFuncVanishAtKKTpoints_LocPrinciple} holds true.
Then for the existence of $c_0 > 0$ such that $S_e(c_0)$ is either bounded or empty it is necessary and sufficient that
there exist $\tau > 0$ and $R > 0$ such that for any $c \ge \tau$ there exists 
$(x(c), \lambda(c)) \in \argmin_{(x, \lambda) \in A \times \Lambda} \mathscr{L}_e(x, \lambda, c)$ with 
$\| x(c) \|_X + \| \lambda(c) \|_Y \le R$.
\end{proposition}

\begin{remark}
Suppose that $\eta(x_*, \lambda_*) = 0$ iff $(x_*, \lambda_*)$ is a KKT-pair of the problem $(\mathcal{P})$, and there
exist $x_* \in \Omega_*$ and $\lambda_* \in K^*$ such that $(x_*, \lambda_*)$ is a KKT-pair of $(\mathcal{P})$. From 
the proposition above it follows that, roughly speaking, the augmented Lagrangian function 
$\mathscr{L}_e(x, \lambda, c)$ is globally exact if and only if it is locally exact at every KKT-pair of
$(\mathcal{P})$ corresponding to a globally optimal solution of this problem, and points of global minimum 
of $\mathscr{L}_e(x, \lambda, c)$ in $(x, \lambda)$ do not escape to infinity as $c \to + \infty$.
\end{remark}

Sometimes, the penalized augmented Lagrangian function $\mathscr{L}_e(x, \lambda, c)$ is locally exact at every
pair $(x_*, \lambda_*) \in \Omega_* \times \Lambda$ such that $\eta(x_*, \lambda_*) = 0$, but the set $S_e(c)$ is
unbounded for all $c > 0$, which implies that $\mathscr{L}_e(x, \lambda, c)$ is not globally exact. In this case it is
natural to ask whether the augmented Lagrangian $\mathscr{L}_e(x, \lambda, c)$ possesses an exactness property that is,
in a sense, intermediate between local and global exactness. Let us show that that the answer to this question is
positive, and in this case the function $\mathscr{L}_e(x, \lambda, c)$ is exact on bounded subsets of 
$A \times \Lambda$. We need the following definition in order to clarify this statement.

\begin{definition} \label{Def_ExactOnBoundedSets}
Let $Q \subseteq A \times \Lambda$ be a nonempty set. The penalized augmented Lagrangian 
function $\mathscr{L}_e(x, \lambda, c)$ is called \textit{exact on the set} $Q$ (with respect to $\eta$) if
there exists $c_0 > 0$ such that $\mathscr{L}_e(x, \lambda, c) \ge f_*$ for all $(x, \lambda) \in Q$ and $c \ge c_0$.
\end{definition}

Let us demonstrate that the above definition is a natural extension of the definition of global exactness.

\begin{proposition} \label{Prp_ExactOnBoundedSets_EquivDef}
Let $Q \subseteq A \times \Lambda$ be a nonempty set such that there exists 
$(\overline{x}, \overline{\lambda}) \in (\Omega_* \times \Lambda) \cap Q$ for which 
$\eta(\overline{x}, \overline{\lambda}) = 0$. Let also assumption $(A4)_s$ hold true, and
\begin{equation} \label{AugmF_VanishAtKKTpoints}
  \Phi(G(x_*), \lambda_*, c) = 0 \quad \forall c > 0 \quad 
  \forall (x_*, \lambda_*) \in \Omega_* \times \Lambda \colon \eta(x_*, \lambda_*) = 0.
\end{equation}
Then $\mathscr{L}_e(x, \lambda, c)$ is exact on $Q$ if and only if there exists $c_0 > 0$ such that for any
$c \ge c_0$ the function $\mathscr{L}_e(\cdot, \cdot, c)$ attains a global minimum on $Q$, and
$(x_*, \lambda_*) \in \argmin_{(x, \lambda) \in Q} \mathscr{L}_e(x, \lambda, c)$ iff 
$x_* \in \Omega_*$ and $\eta(x_*, \lambda_*) = 0$.
\end{proposition}

\begin{proof}
Let $\mathscr{L}_e(x, \lambda, c)$ be exact on $Q$, and $c_0 > 0$ be from Def.~\ref{Def_ExactOnBoundedSets}. Then by
\eqref{AugmF_VanishAtKKTpoints} one has that for any $c \ge c_0$ the function $\mathscr{L}_e(\cdot, \cdot, c)$ attains
the global minimum on $Q$ at $(\overline{x}, \overline{\lambda})$, and 
$\min_{(x, \lambda) \in Q} \mathscr{L}_e(x, \lambda, c) = f_*$.

Let, now, $(x_*, \lambda_*)$ be a point of global minimum of $\mathscr{L}_e(x, \lambda, c_1)$ on $Q$ for some 
$c_1 > c_0$. Then $\mathscr{L}_e(x_*, \lambda_*, c_1) = f_*$. By $(A4)_s$ the function 
$\mathscr{L}(x_*, \lambda_*, \cdot)$ is non-decreasing, which implies that $\mathscr{L}(x_*, \lambda_*, c) = f_*$ for
all $c \in [c_0, c_1]$. Therefore with the use of $(A4)_s$ one obtains that $x_*$ is feasible, and 
$\Phi(G(x_*), \lambda_*, c_0) = 0$. Hence $\mathscr{L}_e(x_*, \lambda_*, c_0) = f(x_*) + \eta(x_*, \lambda_*) = f_*$,
which yields that $x_* \in \Omega_*$ and $\eta(x_*, \lambda_*) = 0$. Thus, for any $c > c_0$ one has 
$(x_*, \lambda_*) \in \argmin_{(x, \lambda) \in Q} \mathscr{L}_e(x, \lambda, c)$ iff $x_* \in \Omega_*$ and 
$\eta(x_*, \lambda_*) = 0$.

It remains to note that the validity of the converse statement follows directly from \eqref{AugmF_VanishAtKKTpoints}.
 
\end{proof}

Now, we can obtain simple necessary and sufficient conditions for the augmented Lagrangian 
$\mathscr{L}_e(x, \lambda, c)$ to be exact on every bounded subset of $A \times \Lambda$.

\begin{theorem}[Localization Principle] \label{Th_ExactnessOnBoundedSets}
Let $Y$ be finite dimensional, $A$ be closed, $G$ be continuous on $A$ and $\mathscr{L}_e(\cdot, \cdot, c)$ be l.s.c. on
$A \times \Lambda$ for all $c > 0$. Suppose also that assumptions $(A4)_s$, $(A6)_s$ and $(A12)$ are satisfied,
and
\begin{equation} \label{AugmFuncVanishAtKKTpoints_LP}
  \Phi(G(x_*), \lambda_*, c) = 0 \quad \forall c > 0 \quad 
  \forall (x_*, \lambda_*) \in \Omega_* \times \Lambda \colon \eta(x_*, \lambda_*) = 0.
\end{equation}
Then $\mathscr{L}_e(x, \lambda, c)$ is exact on any bounded subset of $A \times \Lambda$ if and only if
$\mathscr{L}_e(x, \lambda, c)$ is locally exact at every pair $(x_*, \lambda_*) \in \Omega_* \times \Lambda$ such that
$\eta(x_*, \lambda_*) = 0$.
\end{theorem}

\begin{proof}
If $\mathscr{L}_e(x, \lambda, c)$ is exact on any bounded subset of $A \times \Lambda$, then it is exact on the
intersection of a bounded neighbourhood of a point $(x_*, \lambda_*) \in \Omega_* \times \Lambda$ such that 
$\eta(x_*, \lambda_*) = 0$ and the set $A \times \Lambda$, which with the use of
Proposition~\ref{Prp_ExactOnBoundedSets_EquivDef} implies that $\mathscr{L}_e(x, \lambda, c)$ is locally exact at every
such point.

Let us prove the converse statement. Let $Q \subset A \times \Lambda$ be a bounded set. Replacing, if necessary, the
set $Q$ with its closure one can suppose that $Q$ is closed and, thus, compact (note that 
if $\mathscr{L}_e(x, \lambda, c)$ is exact on the closure of $Q$, then it is exact on $Q$). Choose an increasing
unbounded sequence $\{ c_n \} \subset (0, + \infty)$. Since $\mathscr{L}_e(\cdot, \cdot, c)$ is l.s.c., and $Q$ is
compact, for any $n \in \mathbb{N}$ the function $\mathscr{L}_e(\cdot, \cdot, c_n)$ attains a global minimum on $Q$
at a point $(x_n, \lambda_n)$. Applying the compactness of $Q$ again one can suppose that the sequence 
$\{ (x_n, \lambda_n) \}$ converges to a point $(x_*, \lambda_*) \in Q$.

From $(A4)_s$ it follows that the sequence $\{ \mathscr{L}_e(x_n, \lambda_n, c_n) \}$ is non-decreasing. If this
sequence is unbounded, then there exists $n \in \mathbb{N}$ such that $\mathscr{L}_e(x_n, \lambda_n, c_n) \ge f_*$,
which implies that $\mathscr{L}_e(x, \lambda, c)$ is exact on $Q$ by the definition of $(x_n, \lambda_n)$. On the other
hand, if this sequence is bounded, then arguing in the same way as in the proof of Lemma~\ref{Lemma_EAL_MinimizingSeq}
one can check that $x_*$ is a feasible point of $(\mathcal{P})$.

If $\eta(x_*, \lambda_*) > 0$, then applying $(A12)$ one gets that there exists $c_0 > 0$ such that
$\mathscr{L}_e(x_*, \lambda_*, c) \ge f(x_*) + \eta(x_*, \lambda_*) / 2$ for any $c \ge c_0$. Hence taking into
account $(A4)_s$ and the fact that $\mathscr{L}_e(\cdot, \cdot, c)$ is l.s.c., one obtains that 
$\mathscr{L}_e(x_n, \lambda_n, c_n) \ge f(x_*) + \eta(x_*, \lambda_*) / 4 > f_*$ for any sufficiently large 
$n \in \mathbb{N}$, which implies that $\mathscr{L}_e(x, \lambda, c)$ is exact on $Q$ by the definition of 
$(x_n, \lambda_n)$.

If $\eta(x_*, \lambda_*) = 0$, but $f(x_*) > f_*$, then applying $(A12)$ and taking into account the fact that
$\mathscr{L}_e(\cdot, \cdot, c)$ is l.s.c. one can check that
$\mathscr{L}_e(x_n, \lambda_n, c_n) \ge f(x_*) - \varepsilon > f_*$, provided $\varepsilon > 0$ is small enough, and
$n \in \mathbb{N}$ is large enough. Thus, $\mathscr{L}_e(x, \lambda, c)$ is exact on $Q$ in this case as well.

Finally, if $\eta(x_*, \lambda_*) = 0$ and $f(x_*) = f_*$, i.e. $x_* \in \Omega_*$, then applying the fact that 
$\mathscr{L}_e(x, \lambda, c)$ is locally exact at $(x_*, \lambda_*)$ one can easily verify that this function is exact
on $Q$.	 
\end{proof}

\begin{remark}
Note that the theorem above along with Proposition~\ref{Prp_ExactOnBoundedSets_EquivDef} contain some existing results
(\cite{DiPilloGrippo1979}, Thms.~4--7; \cite{DuZhangGao2006}, Thms.~4.1 and 6.1--6.6; \cite{DuLiangZhang2006},
Thms.~7--12) as simple particular cases.
\end{remark}

\section{Applications of the Localization Principle for Exact Augmented Lagrangian Functions}
\label{Section_ApplLocPrinciple_ExactAL}

In this section, we demonstrate that one can easily prove the local exactness of the penalized augmented Lagrangian
function $\mathscr{L}_e(x, \lambda, c)$ with the use of sufficient optimality conditions and/or a proper constraint
qualification. This result along with the localization principle allows one to reduce the study of global exactness of
the penalized augmented Lagrangian function $\mathscr{L}_e(x, \lambda, c)$ to the local analysis of sufficient
optimality conditions and/or constraint qualifications. In the end of the section, we provide several particular
examples of globally exact augmented Lagrangian functions.

\subsection{Local Exactness via Sufficient Optimality Conditions}

For the sake of simplicity, suppose that $X = \mathbb{R}^d$, $A$ is a convex set, and $f$ is differentiable. In order to
prove the local exactness of $\mathscr{L}_e(x, \lambda, c)$ via second order sufficient optimality conditions one needs
to utilize a suitable second order expansion of the function $(x, \lambda) \to \mathscr{L}_e(x, \lambda, c)$ in
a neighbourhood of a given KKT-pair. The following definition describes what we mean by ``a suitable expansion'' (cf.
Def.~\ref{Def_SeconOrdeExpansion}).

\begin{definition} \label{Def_SeconOrdeExpansionBothVar}
Let assumption $(A11)$ be satisfied, and $G$ be twice Fr\'echet differentiable at a feasible point $x_* \in A$. Let also
$\lambda_* \in K^*$ be such that $\langle \lambda_*, G(x_*) \rangle = 0$. One says that the function 
$\Phi(G(x), \lambda, c)$ admits \textit{the second order expansion} in $(x, \lambda)$ at $(x_*, \lambda_*)$, if for any
$c > 0$ there exist a function $\varphi_c \colon X \times Y^* \to \mathbb{R}$ such that for any $h \in X$, 
$\nu \in Y^*$ and $\alpha > 0$ one has $\varphi_c(\alpha h, \alpha \nu) = \alpha^2 \varphi_c(h, \nu)$, and
\begin{multline*}
  \Phi(G(x_* + h), \lambda_* + \nu, c) - \Phi(G(x_*), \lambda_*, c) =
  \langle \mu_*, D G(x_*) h \rangle \\
  + \frac{1}{2} \langle \mu_*, D^2 G(x_*)(h, h) \rangle + \frac{1}{2} \varphi_c(h, \nu) + 
  o( \| h \|^2 + \| \nu \|^2 ),
\end{multline*}
where $o( \| h \|^2 + \| \nu \|^2 ) / ( \| h \|^2 + \| \nu \|^2)  \to 0$ as $(h, \nu) \to (0, 0)$, 
$\mu_* = \Phi_0(\lambda_*)$, and if 
$$
  \limsup_{[h, \nu, c] \to [h_*, \nu_*, + \infty]} \varphi_c(h, \nu)
$$
is finite for some $h_* \in T_A(x_*)$ and $\nu_* \in Y^*$, then $h_* \in C(x_*, \mu_*)$, and the limit is greater than
or equal to $- \sigma(\mu_*, \mathcal{T}(h_*))$.
\end{definition}

\begin{remark} \label{Rmrk_2OrderExpansInXandLambda}
Let us point out when the assumption that the function $\Phi(G(x), \lambda, c)$ admits the second order expansion in
$(x, \lambda)$ at a point $(x_*, \lambda_*)$ is valid (cf.~Remark~\ref{Rmrk_2OrderExpansInX}). 

This assumption is satisfied in Example~\ref{Example_RockafellarWetsAL}, provided $Y$ is finite dimensional, 
$\sigma(y) = \| y \|^2 / 2$, $(x_*, \lambda_*)$ is a KKT-pair, and the restriction of 
$\sigma(\lambda_*, T^2_K(G(x_*), \cdot))$ to its effective domain is u.s.c. (see~\cite{ShapiroSun}, 
formulae~(3.7), (3.23) and (3.25), and \cite{Dolgopolik_UnifExactnessII}, Theorem~3.3). In this case one has
$$
  \varphi_c(h, \mu) = \min_{z \in \mathscr{C}(x_*, \lambda_*)}
  \Big( \big\| c (D G(x_*) h - z) + \mu \big\|^2 - \sigma\big(\lambda_*, T_K^2(G(x_*), z) \big) \Big) 
  - \frac{1}{c} \| \mu \|^2,
$$
where $\mathscr{C}(x_*, \lambda_*) = \{ z \in T_K(G(x_*)) \mid \langle \lambda_*, z \rangle = 0 \}$. In particular, the
assumption holds true in Examples~\ref{Example_RockWetsAL_SOC} and \ref{Example_RockWetsAL_SemiDefProg}, if 
$(x_*, \lambda_*)$ is a KKT-pair.

The assumption is satisfied in Example~\ref{Example_EssentiallyQuadraticAL} in the case $\phi(s) = s^2 / 2$, if one
defines
\begin{multline*}
  \varphi_c(h, \mu) = \sum_{i \in I_+(x_*, \lambda_*)} 
  \Big( c \| \nabla g_i(x_*) h \|^2 + \mu_i \langle \nabla g_i(x_*), h \rangle \Big) 
  - \frac{1}{c} \sum_{i \in I \setminus I(x_*)} \mu_i^2 \\
  + \sum_{i \in I_0(x_*, \lambda_*)} 
  \Big( \mu_i \max\left\{ \langle \nabla g_i(x_*), h \rangle, - \frac{\mu_i}{c} \right\} + 
  \frac{c}{2} \max\left\{ \langle \nabla g_i(x_*), h \rangle, - \frac{\mu_i}{c} \right\}^2 \Big).
\end{multline*}
The assumption is also valid in Example~\ref{Example_Mangasarian}, provided $\phi''(t) > 0$ for all 
$t \in \mathbb{R}$. In all other examples, one must modify the function $\Phi(y, \lambda, c)$ in order to ensure that
the function $\Phi(G(x), \cdot, c)$ is smooth enough, and $D_{\lambda} \Phi(G(x_*), \lambda_*, c) = 0$ (this equality
is necessary for $\Phi(G(x), \lambda, c)$ to admit the second order expansion in $(x, \lambda)$; see
Def.~\ref{Def_SeconOrdeExpansionBothVar}).

Namely, suppose for the sake of shortness that there are no equality constraints, and choose a twice continuously
differentiable function $\zeta \colon \mathbb{R} \to \mathbb{R}$ such that $\zeta'(0) = 0$, 
$\zeta(\mathbb{R}_+) = \mathbb{R}_+$, and $\zeta(t) = 0$ iff $t = 0$. Let us replace the function $\Phi(y, \lambda, c)$
with the function $\Phi(y, \zeta(\lambda), c)$, where $\zeta(\lambda) = (\zeta(\lambda_1), \ldots, \zeta(\lambda_l))$ in
the case when $(\mathcal{P})$ is a mathematical programming problem, and $\zeta(\lambda)$ is L\"{o}wner's operator
associated with $\zeta$ in the case when $(\mathcal{P})$ is either a second order cone or semidefinite programming
problem. Then one can verify that the assumption on the function $\Phi(G(x), \zeta(\lambda), c)$ is satisfied in
Example~\ref{Example_CubilAL} with $\zeta(t) = t^2$ iff s.c. condition holds true. The assumption holds true in
Examples~\ref{Example_ExpPenFunc}--\ref{Example_ModBarrierFunc} iff $\phi'(0) \ne 0$, $\phi''(0) > 0$, and s.c.
condition holds true. This assumption is satisfied in Example~\ref{Example_pthPowerAugmLagr} iff $\phi'(b) \ne 0$,
$\phi''(b) > - \phi'(b)^2$, and s.c. condition holds true, and it is always satisfied in
Example~\ref{Example_HeWuMengLagrangian}, provided $\zeta(t) = t^{1 + \varepsilon}$ for some $\varepsilon > 0$.

The assumption is also satisfied in Examples~\ref{Example_NonlinearRescale_SemiDefProg} and
\ref{Example_NonlinearRescalePenalized_SemiDefProg}, if s.c. condition holds true (\cite{Stingl2006}, Thm.~5.1, and
\cite{LuoWuLiu2015}, Prp.~4.2), and $\zeta(t) = t^2$.

Finally, let us note that under some natural assumptions on the function $\zeta$ (that are satisfied, e.g., if
$\zeta(t) = t^2$) the function $\Phi(y, \zeta(\lambda), c)$ has the same properties as the function 
$\Phi(y, \lambda, c)$. In particular, $\Phi(y, \zeta(\lambda), c)$ satisfies the same main assumption of this paper as
$\Phi(y, \lambda, c)$.
\end{remark}

Now, we can prove that under some additional assumptions the penalized augmented Lagrangian function 
$\mathscr{L}_e(x, \lambda, c)$ is locally exact at a KKT-pair $(x_*, \lambda_*)$ of the problem $(\mathcal{P})$,
provided this KKT-pair satisfies the second order sufficient optimality condition, and the function 
$\eta(x_*, \cdot)$ behaves like a positive definite quadratic function in a neighbourhood of $\lambda_*$.

\begin{theorem} \label{Th_LocalExactnessOfAL}
Let $Y$ be finite dimensional, assumptions $(A4)$ and $(A11)$ be satisfied, $x_*$ be a locally optimal solution of the
problem $(\mathcal{P})$, $f$ and $G$ be twice Fr\'echet differentiable at $x_*$, and $(x_*, \mu_*)$ be a KKT-pair of the
problem $(\mathcal{P})$ satisfying the second order sufficient optimality condition. Suppose also that the function 
$\Phi(G(x), \lambda, c)$ admits the second order expansion in $(x, \lambda)$ at $(x_*, \lambda_*)$ for some 
$\lambda_* \in \Phi_0^{-1}(\mu_*)$, the function $\eta$ is twice Fr\'echet differentiable at $(x_*, \lambda_*)$,
$D^2_{\lambda \lambda} \eta(x_*, \lambda_*)$ is positive definite, and $\eta(x_*, \lambda_*) = 0$. Then the penalized
augmented Lagrangian function $\mathscr{L}_e(x, \lambda, c)$ is locally exact at $(x_*, \lambda_*)$.
\end{theorem}

\begin{proof}
Denote $\xi = (x, \lambda)$, $\xi_* = (x_*, \lambda_*)$ and $\overline{\xi} = (x_*, \mu_*)$. Under the assumptions
of the theorem, for any $c > 0$ there exists a neighbourhood $U_c$ of $\xi_*$ such that for 
all $\xi = (x, \lambda) \in U_c$ one has
\begin{multline} \label{EAL_TaylorExpansion}
  \Big| \mathscr{L}_e(\xi, c) - \mathscr{L}_e(\xi_*, c) - 
  \big\langle D_x L(\overline{\xi}), x - x_* \big\rangle - 
  \frac{1}{2} \big\langle x - x_*, D^2_{xx} L(\overline{\xi}) (x - x_*) \big\rangle \\
  - \frac{1}{2} \varphi_c(x - x_*, \lambda - \lambda_*) 
  - \frac{1}{2} D^2 \eta(\xi_*) (\xi - \xi_*, \xi - \xi_*) \Big| \\
  < \frac{1}{2 c} \Big( \| x - x_* \|^2 + \| \lambda - \lambda_* \|^2 \Big).
\end{multline}
Here we used the equalities $D_x \eta(\xi_*) = 0$ and $D_{\lambda} \eta(\xi_*) = 0$ that hold true
due to the facts that $\eta(x, \lambda)$ is nonnegative, and $\eta(\xi_*) = 0$.

Arguing by reductio ad absurdum, suppose that $\mathscr{L}_e(x, \lambda, c)$ is not locally exact at $(x_*, \lambda_*)$.
Then for any $n \in \mathbb{N}$ there exists $\xi_n = (x_n, \lambda_n) \in U_n \cap (A \times \Lambda)$ such
that $\mathscr{L}_e(\xi_n, n) < \mathscr{L}_e(\xi_*, n)$. Applying \eqref{EAL_TaylorExpansion} one gets that
\begin{multline} \label{EAL_TaylorExp_2}
  0 > \langle D_x L(\overline{\xi}), x_n - x_* \rangle +
  \frac{1}{2} \big\langle x_n - x_*, D^2_{xx} L(\overline{\xi}) (x_n - x_*) \big\rangle
  + \frac{1}{2} \varphi_n(x_n - x_*, \lambda_n - \lambda_*) \\
  + \frac{1}{2} D^2 \eta(\xi_*) (\xi_n - \xi_*, \xi_n - \xi_*) 
  - \frac{1}{2 n} \Big( \| x_n - x_* \|^2 + \| \lambda_n - \lambda_* \|^2 \Big)
\end{multline}
for any $n \in \mathbb{N}$. Denote $\alpha_n = \| x_n - x_* \| + \| \lambda_n - \lambda_* \|$, 
$h_n = (x_n - x_*) / \alpha_n$ and $\nu_n = (\lambda_n - \lambda_*) / \alpha_n$. Replacing, if necessary, the sequence
$\{ (x_n, \lambda_n) \}$ by its subsequence, one can suppose that the sequence $\{ (h_n, \nu_n) \}$ converges to a point
$(h_*, \nu_*)$ such that $\| h_* \| + \| \nu_* \| = 1$ (recall that both $X$ and $Y$ are finite dimensional). Note that
$h_* \in T_A(x_*)$.

Suppose, at first, that $h_* = 0$. In other words, suppose that $(x_n - x_*) / \alpha_n$ converges to zero as 
$n \to \infty$. Observe that
\begin{equation} \label{EAL_KKTCond_Impl}
  x_n - x_* \in T_A(x_*), \quad \langle D_x L(\overline{\xi}), x_n - x_* \rangle \ge 0 
  \quad \forall n \in \mathbb{N}
\end{equation}
due to the facts that $x_n \in A$, $A$ is convex, and $(x_*, \mu_*)$ is a KKT-pair. Hence dividing
\eqref{EAL_TaylorExp_2} by $\alpha_n^2$, and passing to the limit as $n \to \infty$ one obtains 
that $D^2_{\lambda \lambda} \eta(x_*, \lambda_*)(\nu_*, \nu_*) \le 0$ by virtue of
Def.~\ref{Def_SeconOrdeExpansionBothVar}, which contradicts our assumption that
$D^2_{\lambda \lambda} \eta(x_*, \lambda_*)$ is positive definite. Thus, $h_* \ne 0$.

Recall that $\eta$ is a nonnegative function and $\eta(\xi_*) = 0$, i.e. $\xi_*$ is a point of global minimum of the
function $\eta$. Therefore $D^2 \eta(\xi_*) (\xi_n - \xi_*, \xi_n - \xi_*) \ge 0$ for all $n \in \mathbb{N}$. Applying
this estimate and \eqref{EAL_KKTCond_Impl} in \eqref{EAL_TaylorExp_2}, and dividing by $\alpha_n^2$ one gets that
\begin{equation} \label{EAL_TaylorExp_Final}
  0 > \langle h_n, D^2_{xx} L(\overline{\xi}) h_n \rangle
  + \varphi_n(h_n, \mu_n) - \frac{1}{n} 
\end{equation}
for any $n \in \mathbb{N}$. Here we used the fact that $\alpha^2 \varphi_c(h, \mu) = \varphi_c( \alpha h, \alpha \mu)$
for all $\alpha \ge 0$ by Def.~\ref{Def_SeconOrdeExpansionBothVar}.

Passing to the limit superior in \eqref{EAL_TaylorExp_Final} with the use of Def.~\ref{Def_SeconOrdeExpansionBothVar}
one obtains that $h_* \in C(x_*, \lambda_*)$ and 
$0 \ge \langle h_*, D^2_{xx} L(\overline{\xi}) h_* \rangle - \sigma(\mu_*, \mathcal{T}(h_*))$,
which contradicts the assumption that the KKT-pair $(x_*, \mu_*)$ satisfies the second order sufficient optimality
condition.	 
\end{proof}

\begin{remark}
Note that the theorem above contains some existing results (e.g. \cite{DiPilloLucidi2001}, Theorem~6.3;
\cite{LuoWuLiu2013}, Theorem~4.1) as simple particular cases. Furthermore, it is easy to verify that under the
assumptions of the theorem there exist $\gamma > 0$ and $c_0 > 0$ such that 
$\mathscr{L}_e(x, \lambda, c) \ge \mathscr{L}_e(x_*, \lambda_*, c) 
+ \gamma (\| x - x_* \|^2 + \| \lambda - \lambda_* \|^2)$ for all $c \ge c_0$ and $(x, \lambda) \in A \times \Lambda$
that are sufficiently close to $(x_*, \lambda_*)$.
\end{remark}

Let us now provide several particular examples of globally exact augmented Lagrangian functions, and demonstrate how one
can utilize the theorem above and the localization principle in order to prove the global exactness of these augmented
Lagrangians.

\subsection{Mathematical Programming}

Let the problem $(\mathcal{P})$ be the mathematical programming problem of the form
\begin{equation} \label{MathProg_EAL}
  \min f(x) \quad \text{subject to} 
  \quad g_i(x) \le 0, \quad i \in I, \quad g_j(x) = 0, \quad j \in J,
\end{equation}
where $g_i \colon X \to \mathbb{R}$, $I = \{ 1, \ldots, l \}$ and $J = \{ l + 1, \ldots, l + s \}$.
Suppose that the functions $f$ and $g_i$, $i \in I \cup J$ are twice continuously differentiable. Define
$$
  \eta_1(x, \lambda) = 
  \sum_{i = 1}^l \Big( \big\langle D_x L(x, \lambda), \nabla g_i(x) \big\rangle + g_i(x)^2 \lambda_i \Big)^2 +
  \sum_{j = l + 1}^{l + s} \big\langle D_x L(x, \lambda), \nabla g_j(x) \big\rangle^2
$$
(see~\cite{DiPilloLucidi2001}). Let $x_*$ be a locally optimal solution of problem \eqref{MathProg_EAL}, and let LICQ
holds at $x_*$. Then one can easily verify that $\eta_1(x_*, \lambda_*) = 0$ for some $\lambda_* \in \mathbb{R}^{l + s}$
iff $(x_*, \lambda_*)$ is a KKT-pair of problem \eqref{MathProg_EAL} (note that such $\lambda_*$ is unique).
Furthermore, the matrix $D^2_{\lambda \lambda} \eta_1(x_*, \lambda_*)$ is positive definite. Thus, one can apply
Theorem~\ref{Th_LocalExactnessOfAL} in order to prove the local exactness of the penalized augmented Lagrangian
function $\mathscr{L}_e(x, \lambda, c)$ for problem \eqref{MathProg_EAL}. 

In order to ensure that the function $\mathscr{L}_e(\cdot, \cdot, c)$ is level-bounded and, thus, globally exact, one
must add barrier terms into the definition of this function (see~\cite{DiPilloLucidi2001}, and
Remark~\ref{Remark_BarrierTerms} above). Choose $\alpha > 0$ and $\varkappa > 2$, and define
$$
  p(x, \lambda) = \frac{a(x)}{1 + \sum_{i = 1}^l \lambda_i^2}, \quad 
  q(x, \lambda) = \frac{b(x)}{1 + \sum_{j = l + 1}^{l + s} \lambda_j^2},
$$
where
$$
  a(x) = \alpha - \sum_{i = 1}^l \max\{ 0, g_i(x) \}^{\varkappa}, \quad 
  b(x) = \alpha - \sum_{j = l + 1}^{l + s} g_j(x)^2.
$$
Denote $\Omega_{\alpha} = \{ x \in \mathbb{R}^d \mid a(x) > 0, \: b(x) > 0 \}$. Then one can introduce the following
penalized augmented Lagrangian function for problem \eqref{MathProg_EAL} (see
Example~\ref{Example_EssentiallyQuadraticAL}). Namely, define
\begin{multline} \label{EAL_HPR_AugmLagr}
  \mathscr{L}_e(x, \lambda, c) = f(x) \\
  + \sum_{i = 1}^l \Big( \lambda_i \max\Big\{ g_i(x), - \frac{p(x, \lambda)}{c} \lambda_i \Big\} +
  \frac{c}{2 p(x, \lambda)} \max\Big\{ g_i(x), - \frac{p(x, \lambda)}{c} \lambda_i \Big\}^2 \Big) \\
  + \sum_{j = l + 1}^{l + s} \Big( \lambda_j g_j(x) + \frac{c}{2 q(x, \lambda)} g_j(x)^2 \Big) + \eta_1(x, \lambda),
\end{multline}
if $x \in \Omega_{\alpha}$, and $\mathscr{L}_e(x, \lambda, c) = + \infty$ otherwise. 

One can easily verify that the augmented Lagrangian function introduced above is l.s.c. jointly in $(x, \lambda)$ on
$\mathbb{R}^{d} \times \mathbb{R}^{l + s}$, and continuously differentiable in $(x, \lambda)$ on its effective
domain, i.e. on $\Omega_{\alpha} \times \mathbb{R}^{l + s}$ (cf.~\cite{DiPilloLucidi2001}). One can also check
that the function $\Phi(y, \lambda, c)$ corresponding to penalized augmented Lagrangian function
\eqref{EAL_HPR_AugmLagr} satisfies assumptions $(A2)$, $(A4)_s$, $(A6)_s$, $(A11)$ 
with $\Phi_0(\lambda) \equiv \lambda$, and $(A12)$. Furthermore, the function $\Phi(G(x), \lambda, c)$ admits the
second order expansion in $(x, \lambda)$ at every KKT-pair of problem \eqref{MathProg_EAL}, and
$\Phi(G(x_*), \lambda_*, c) = 0$, if $(x_*, \lambda_*)$ is a KKT-pair. Therefore one can obtain the following result.

\begin{theorem} \label{Thrm_EAL_HPR}
Let $f$ and $g_i$, $i \in I \cup J$, be twice continuously differentiable, LICQ hold true at every globally optimal
solution of problem \eqref{MathProg_EAL}, and for any $x_* \in \Omega_*$ a unique KKT-pair $(x_*, \lambda_*)$ satisfy 
the second order sufficient optimality condition. Then penalized augmented Lagrangian function \eqref{EAL_HPR_AugmLagr}
is globally exact if and only if there exists $c_0 > 0$ such that the set
$$
  S_e(c_0) := 
  \big\{ (x, \lambda) \in \mathbb{R}^d \times \mathbb{R}^{l + s} \bigm| \mathscr{L}_e(x, \lambda, c_0) < f_* \big\}
$$
is either bounded or empty. In particular, if there exists $\gamma > 0$ such that the set
$$
  \Omega(\gamma, \alpha) := \big\{ x \in \mathbb{R}^d \mid f(x) < f_* + \gamma, \: a(x) > 0, \: b(x) > 0 \}
$$
is bounded, then penalized augmented Lagrangian function \eqref{EAL_HPR_AugmLagr}is globally exact in the sense that
its points of global minimum in $(x, \lambda)$ on $\mathbb{R}^d \times \mathbb{R}^{l + s}$ are exactly KKT-pairs of
problem \eqref{MathProg_EAL} corresponding to globally optimal solutions of this problem, provided $c > 0$ is large
enough.
\end{theorem}

\begin{proof}
Under the assumptions of the theorem augmented Lagrangian \eqref{EAL_HPR_AugmLagr} is locally exact at every KKT-pair
corresponding to a globally optimal solution of problem \eqref{MathProg_EAL} by Theorem~\ref{Th_LocalExactnessOfAL}.
Then applying the localization principle (Theorem~\ref{Thrm_EAL_LocalizationPrinciple}) one obtains that this augmented
Lagrangian is globally exact iff the set $S_e(c_0)$ is either bounded or empty for some $c_0 > 0$.

Let us verify that if the set $\Omega(\gamma, \alpha)$ is bounded, then the set $S_e(c_0)$ is bounded for some 
$c_0 > 0$. Minimizing the function $q(t) = \lambda_i t + c t^2 / 2 p(x, \lambda)$ one obtains that for 
any $(x, \lambda) \in \Omega_{\alpha} \times \mathbb{R}^{l + s}$ and $c > 0$ the following inequalities hold true:
\begin{multline} \label{EAL_HPR_LowerEstimate}
  \mathscr{L}_e(x, \lambda, c) \ge f(x) - 
  \frac{p(x, \lambda)}{2c} \sum_{i = 1}^l \lambda_i^2
  - \frac{q(x, \lambda)}{2c} \sum_{j = l + 1}^{l + s} \lambda_j^2 + \eta_1(x, \lambda) \\
  \ge f(x) - \frac{\alpha}{c} + \eta_1(x, \lambda).
\end{multline}
Consequently, for any $c > \alpha / \gamma$ and $(x, \lambda) \in S_e(c)$ one has 
$x \in \Omega(\gamma, \alpha)$. Hence applying \eqref{EAL_HPR_LowerEstimate} and the fact that the set
$\Omega(\gamma, \alpha)$ is bounded one gets that there exists $\tau \in \mathbb{R}$ such that 
$\mathscr{L}_e(x, \lambda, c) \ge \tau$ for all 
$(x, \lambda) \in \mathbb{R}^d \times \mathbb{R}^{l + s}$ and $c > \alpha / \gamma$.

Arguing by reductio ad absurdum, suppose that the set $S_e(c)$ is unbounded for any $c > 0$. Choose an increasing
unbounded sequence $\{ c_n \} \subset \mathbb{R}_+$ such that $c_1 > \alpha / \gamma$. Then for any 
$n \in \mathbb{N}$ there exists $(x_n, \lambda_n) \in S_e(c_n)$ such that $\| \lambda_n \| \ge n$. By the choice of
$c_1$ one has that $\{ x_n \} \subset \Omega(\gamma, \alpha)$. Therefore without loss of generality one can suppose
that the sequence $\{ x_n \}$ converges to a point $x_*$. Let us show that $x_*$ is a globally optimal solution of
problem \eqref{MathProg_EAL}.

Denote $w_n^j = q(x_n, \lambda_n) (\lambda_n)_j g_j(x_n) + c_n g_j(x_n)^2 / 2$
\begin{multline*}
  u_n^i = p(x_n, \lambda_n) (\lambda_n)_i \max\Big\{ g_i(x_n), - \frac{p(x_n, \lambda_n)}{c_n} (\lambda_n)_i \Big\} \\
  + \frac{c_n}{2} \max\Big\{ g_i(x_n), - \frac{p(x_n, \lambda_n)}{c_n} (\lambda_n)_i \Big\}^2.
\end{multline*}
Arguing as above, one can easily verify that $\liminf_{n \to \infty} u_n / p(x_n, \lambda_n) \ge 0$ and 
$\liminf_{n \to \infty} w_n / q(x_n, \lambda_n) \ge 0$.

From \eqref{EAL_HPR_LowerEstimate} and the fact that $(x_n, \lambda_n) \in S_e(c_n)$ it follows that
the sequences $\{ f(x_n) \}$ and $\{ \eta_1(x_n, \lambda_n) \}$ are bounded. Moreover, by the definition of 
$\{ (x_n, \lambda_n) \}$ one has $\tau \le \mathscr{L}_e(x_n, \lambda_n, c_n) < f_*$. Therefore the sequence
$\sum_{i = 1}^l u^i_n / p(x_n, \lambda_n) + \sum_{j = l + 1}^{l + s} w^j_n / q(x_n, \lambda_n)$ is bounded, which
implies that $\{ u^i_n / p(x_n, \lambda_n) \}$, $i \in I$, and $\{ w_n^j / q(x_n, \lambda_n) \}$, $j \in J$, are bounded
sequences. By definition $0 < p(x_n, \lambda_n) \le \alpha$ and $0 < q(x_n, \lambda_n) \le \alpha$ for all 
$n \in \mathbb{N}$. Therefore, the sequences $\{ u_n^i \}$ and $\{ w_n^j \}$ are bounded as well. Hence taking into
account the fact that $c_n \to + \infty$ as $n \to \infty$ one can easily check that $\max\{ g_i(x_n), 0 \} \to 0$ and
$g_j(x_n) \to 0$ as $n \to \infty$ for all $i \in I$ and $j \in J$, which implies that $x_*$ is a feasible point.

Indeed, suppose, for instance, that $g_j(x_n)$ does not converge to zero. Then there exist $\varepsilon > 0$ and
a subsequence $\{ x_{n_k} \}$ such that $|g_j(x_{n_k})| \ge \varepsilon$ for all $k \in \mathbb{N}$. Consequently,
$w_{n_k}^j \ge c_{n_k}^2 \varepsilon^2 / 2 - \alpha |g_j(x_{n_k})|$. Passing to the limit as $k \to \infty$ one obtains
that $w_{n_k}^j \to + \infty$ as $k \to \infty$, which is impossible. Thus, $x_*$ is a feasible point of problem
\eqref{MathProg_EAL}.

From \eqref{EAL_HPR_LowerEstimate} and the fact that $(x_n, \lambda_n) \in S_e(c_n)$ it follows that
$f(x_*) \le f_*$, which implies that $x_*$ is a globally optimal solution of problem \eqref{MathProg_EAL}. Consequently,
LICQ holds at this point. 

Observe that the function $\eta(x, \cdot)$ is quadratic. Hence and from the fact that LICQ holds at $x_*$ one obtains
that there exist a neighbourhood $U$ of $x_*$, $\delta_1 > 0$ and $\delta_2, \delta_3 \in \mathbb{R}$ such that
$\eta(x, \lambda) \ge \delta_1 \| \lambda \|^2 + \delta_2 \| \lambda \| + \delta_3$ for all $x \in U$ and 
$\lambda \in \mathbb{R}^{l + s}$. Applying this estimate in \eqref{EAL_HPR_LowerEstimate}, and taking into
account the fact that $\| \lambda_n \| \to + \infty$ as $n \to \infty$ one gets that 
$\mathscr{L}_e(x_n, \lambda_n, c_n) \to + \infty$ as $n \to \infty$, which contradicts the definition of the sequence
$\{ (x_n, \lambda_n) \}$. Thus, $S_e(c)$ is bounded for some $c > 0$.	 
\end{proof}

\begin{remark}
{(i)~Note that the theorem above strengthens all existing results on global exactness of augmented Lagrangian functions
(see, e.g., \cite{DiPilloLucidi2001}, Thm.~4.6; \cite{DiPilloLiuzzi2011}, Prp.~1; 
\cite{LuoWuLiu2013}, Thms.~4.3 and 4.4), since it provides first \textit{necessary and sufficient} conditions for the
global exactness, and is formulated for the optimization problem with \textit{both} equality and inequality
constraints. Furthermore, to the best of author's knowledge the theorem above provides first sufficient conditions for
the \textit{global} exactness of augmented Lagrangian functions for equality constrained optimization problems
(cf.~\cite{DiPilloGrippo1979,DiPilloGrippo1980,DuZhangGao2006,DuLiangZhang2006,DiPilloLiuzzi2011}).
}

\noindent{(ii)~It should be mentioned that one can construct an exact augmented Lagrangian function from Mangasarian's
augmented Lagrangian (Example~\ref{Example_Mangasarian}). However, since $\Phi_0(\lambda) \equiv \phi'(\lambda)$ for
this augmented Lagrangian, one has to consider the penalty term $\eta(x, \lambda) = \eta_1(x, \phi'(\lambda))$, and
impose some rather restrictive assumptions on the function $\phi$ in order to ensure the global exactness of the
corresponding penalized augmented Lagrangian. 
}
\end{remark}

Let us extend the previous theorem to the case of other augmented Lagrangian functions. As it was noted in
Remark~\ref{Rmrk_2OrderExpansInXandLambda}, only augmented Lagrangian functions from
Examples~\ref{Example_EssentiallyQuadraticAL} and \ref{Example_Mangasarian} admit the second order expansion in 
$(x, \lambda)$ without some modifications of the function $\Phi(y, \lambda, c)$. In order to accommodate necessary
modifications of this function one needs to modify the penalty term $\eta(x, \lambda)$ as well. Below, we utilize the
transformation $\lambda \to \zeta(\lambda)$ with $\zeta(t) = t^2$ in order to ensure the desired properties of the
function $\Phi(y, \lambda, c)$. Therefore we define
$$
  \eta_2(x, \lambda) = \sum_{i = 1}^l \big\langle D_x L(x, \zeta(\lambda) ), \nabla g_i(x) \big\rangle^2 + 
  \sum_{i = 1}^l g_i(x)^2 \lambda_i^2.
$$
Hereinafter, $\zeta(\lambda) = (\lambda_1^2, \ldots, \lambda_l^2)$. 

Let $x_*$ be a locally optimal solution of problem \eqref{MathProg_EAL}, and let LICQ holds at $x_*$. Then, as in the
case of $\eta_1(x, \lambda)$, one can easily verify that $\eta_2(x_*, \lambda_*) = 0$ for some 
$\lambda_* \in \mathbb{R}^{l + s}$ iff $(x_*, \zeta(\lambda_*))$ is a KKT-pair of problem \eqref{MathProg_EAL}.
Furthermore, the matrix $D^2_{\lambda \lambda} \eta_1(x_*, \lambda_*)$ is positive definite, provided the pair 
$(x_*, \lambda_*)$ satisfies s.c. condition. Finally, one can check that there exists a neighbourhood $U$ of $x_*$
such that $\inf_{x \in U} \eta_2(x, \lambda) \to + \infty$ as $\| \lambda \| \to + \infty$.

Now, we can define the following penalized augmented Lagrangian functions. Suppose that there are no equality
constraints, and define
\begin{multline} \label{EAL_Cubic}
  \mathscr{L}_e(x, \lambda, c) = f(x) + \eta_2(x, \lambda) \\ 
  + \frac{1}{3 c p(x, \zeta(\lambda))^2} \sum_{i = 1}^l 
  \Big[ \max\Big\{ c g_i(x) + p(x, \zeta(\lambda)) \lambda_i, 0 \Big\}^3 - 
  \big| p(x, \zeta(\lambda)) \lambda_i \big|^3 \Big],
\end{multline}
if $x \in \Omega_{\alpha}$, and $\mathscr{L}_e(x, \lambda, c) = + \infty$ otherwise (see~Example~\ref{Example_CubilAL}).
Similarly, one can define
\begin{equation} \label{EAL_PenalizedExpType}
  \mathscr{L}_e(x, \lambda, c) = f(x) + \frac{p(x, \lambda)}{c} \sum_{i = 1}^l
  \left[ \lambda_i^2 \phi\left( \frac{c g_i(x)}{p(x, \lambda)} \right) + 
  \xi\left( \frac{c g_i(x)}{p(x, \lambda)} \right) \right] + \eta_2(x, \lambda),
\end{equation}
if $x \in \Omega_{\alpha}$, and $\mathscr{L}_e(x, \lambda, c) = + \infty$ otherwise
(see~Example~\ref{Example_PenalizedExpPenFunc}). Here the function 
$\phi \colon \mathbb{R} \to \mathbb{R} \cup \{ + \infty \}$ is non-decreasing, strictly convex, and such that 
$\dom \phi = (- \infty, \varepsilon_0)$ for some $\varepsilon_0 \in (0, + \infty]$, $\phi(t) \to + \infty$ as 
$t \to \varepsilon_0$, $\phi(t) / t \to + \infty$ as $t \to + \infty$ if $\varepsilon_0 = + \infty$, $\phi$ is twice
continuously differentiable on $\dom \phi$, $\phi(0) = 0$, $\phi'(0) = 1$, and $\phi''(0) > 0$, while 
the function $\xi \colon \mathbb{R} \to \mathbb{R}_+$ is twice continuously differentiable, and such that $\xi$ is
strictly convex on $\mathbb{R}_+$, $\xi(t) = 0$ for all $t \le 0$, and $\xi(t) > 0$ for all $t > 0$.

Finally, one can define
\begin{equation} \label{EAL_HeWuMeng}
  \mathscr{L}_e(x, \lambda, c) = f(x) + \frac{1}{c p(x, \lambda)}
  \sum_{i = 1}^l \int_0^{c g_i(x)} \big( \sqrt{t^2 + p(x, \lambda)^2 \lambda_i^4} + t \big) \, dt 
  + \eta_2(x, \lambda),
\end{equation}
if $x \in \Omega_{\alpha}$, and $\mathscr{L}_e(x, \lambda, c) = + \infty$ otherwise
(see~Example~\ref{Example_HeWuMengLagrangian}).

As in the case of augmented Lagrangian \eqref{EAL_HPR_AugmLagr}, one can easily verify that the penalized augmented
Lagrangian functions \eqref{EAL_Cubic}, \eqref{EAL_PenalizedExpType} and \eqref{EAL_HeWuMeng} are l.s.c. jointly in 
$(x, \lambda)$ on $\mathbb{R}^{d} \times \mathbb{R}^{l + s}$, and continuously differentiable in $(x, \lambda)$ on their
effective domains. One can also check that the functions $\Phi(y, \lambda, c)$ corresponding to these augmented
Lagrangians satisfy assumptions $(A2)$, $(A4)_s$, $(A6)_s$, $(A11)$ with $\Phi_0(\lambda) \equiv \zeta(\lambda)$, and
$(A12)$ in the case  $\Lambda = Y^* = \mathbb{R}^{l + s}$. Furthermore, the corresponding functions 
$\Phi(G(x), \lambda, c)$ admit the second order expansion in $(x, \lambda)$ at every point $(x_*, \lambda_*)$ such that
$x_*$ is feasible, and $\lambda_*$ satisfies s.c. condition. Finally, $\Phi(G(x_*), \lambda_*, c) = 0$, if $x_*$ is
feasible, and $\eta(x_*, \lambda_*) = 0$. 

Thus, one can apply the localization principle (Theorem~\ref{Thrm_EAL_LocalizationPrinciple}) and
Theorem~\ref{Th_LocalExactnessOfAL} in order to prove the global exactness of penalized augmented Lagrangian functions
\eqref{EAL_Cubic} and \eqref{EAL_PenalizedExpType}. Augmented Lagrangian \eqref{EAL_HeWuMeng}, as one can verify, is not
bounded from below, which implies that it is not globally exact.

Arguing in the same way as in the proof of Theorem~\ref{Thrm_EAL_HPR} one can verify that the following result holds
true.

\begin{theorem}
Let $f$ and $g_i$, $i \in I$, be twice continuously differentiable, LICQ hold true at every globally optimal
solution of problem \eqref{MathProg_EAL}, and for any $x_* \in \Omega_*$ a unique KKT-pair $(x_*, \lambda_*)$ satisfy
the second order sufficient optimality and the strict complementarity conditions. Then penalized augmented Lagrangian
functions \eqref{EAL_Cubic} and \eqref{EAL_PenalizedExpType} are globally exact (with respect to the function 
$\eta_2(x, \lambda)$) if and only if there exists $c_0 > 0$ such that the corresponding set $S_e(c_0)$ is
either bounded or empty. In particular, if there exists $\gamma > 0$ such that the set $\Omega(\gamma, \alpha)$ is
bounded, then penalized augmented Lagrangian functions \eqref{EAL_Cubic} and \eqref{EAL_PenalizedExpType} (provided
$\phi$ is bounded from below) are globally exact in the sense that for any sufficiently large $c > 0$ a pair 
$(x_*, \lambda_*)$ is a point of global minimum of these functions in $(x, \lambda)$ on 
$\mathbb{R}^d \times \mathbb{R}^l$ iff $x_* \in \Omega_*$ and $(x_*, \zeta(\lambda_*))$ is a KKT-pair of problem
\eqref{MathProg_EAL}.
\end{theorem}

As it was noted above, augmented Lagrangian \eqref{EAL_HeWuMeng} is not globally exact. However, with the use of
Theorems~\ref{Th_ExactnessOnBoundedSets} and \ref{Th_LocalExactnessOfAL} one can provide simple sufficient contidions
for the the exactness of this augmented Lagrangian on bounded sets. Namely, the following result holds true.

\begin{theorem}
Let $f$ and $g_i$, $i \in I$, be twice continuously differentiable, LICQ hold true at every globally optimal
solution of problem \eqref{MathProg_EAL}, and for any $x_* \in \Omega_*$ a unique KKT-pair $(x_*, \lambda_*)$ satisfy
the second order sufficient optimality and the strict complementarity conditions. Then penalized augmented Lagrangian
function \eqref{EAL_HeWuMeng} is exact on any bounded subset of $\mathbb{R}^d \times \mathbb{R}^l$ with respect to 
the function $\eta_2$.
\end{theorem}

\begin{remark}
{(i)~To the best of author's knowledge, all exact augmented Lagrangian function studied in the literature were
constructed from the Hestenes-Powell-Rockafellar augmented Lagrangian function. Thus, penalized augmented Lagran\-gian
functions \eqref{EAL_Cubic}, \eqref{EAL_PenalizedExpType} and \eqref{EAL_HeWuMeng} as well as the theorems on exactness
of these functions are completely new. Furthermore, even augmented Lagrangian \eqref{EAL_HPR_AugmLagr} has not been
considered for the case of general equality and inequality constraints before.
}

\noindent{(ii)~Note that in order to guarantee that there is a unique point of global minimum of augmented Lagrangians
\eqref{EAL_Cubic} and \eqref{EAL_PenalizedExpType} in $(x, \lambda)$ corresponding to a globally optimal solution of
problem~\eqref{MathProg_EAL} one can add the penalty term 
$\omega(\lambda) = \sum_{i = 1}^l \max\{ 0, - \lambda_i \}^{\varkappa}$ to these functions. Note also that one can
incorporate equality constraints into augmented Lagrangians \eqref{EAL_Cubic}, \eqref{EAL_PenalizedExpType} and
\eqref{EAL_HeWuMeng} in the same way as in \eqref{EAL_HPR_AugmLagr}.
}

\noindent{(iii)~It should be mentioned that globally exact augmented Lagrangian functions cannot be constructed from
the exponential penalty function (Example~\ref{Example_ExpPenFunc}), the modified barrier function
(Example~\ref{Example_ModBarrierFunc}) and the p-th power augmented Lagrangian (Example~\ref{Example_pthPowerAugmLagr})
without some nontrivial transformations of these functions due to the facts that for these augmented Lagrangians one has
$\mathscr{L}(x, 0, c) \equiv f(x)$, and it is very difficult (if at all possible) to construct a continuously
differentiable function $\eta(x, \lambda)$ that satisfies the main assumption of this article, and such that 
$f(x) + \eta(x, 0) \ge f_*$ for all $x \in \mathbb{R}^d$.
}
\end{remark}

\subsection{Nonlinear Second Order Cone Programming}

Let the problem $(\mathcal{P})$ be the nonlinear second order cone programming problem of the form
\begin{equation} \label{SecondOrderConeProg_EAL}
  \min f(x) \quad \text{subject to} 
  \quad g_i(x) \in Q_{l_i + 1}, \quad i \in I, \quad h(x) = 0,
\end{equation}
where the functions $f \colon X \to \mathbb{R}$, $g_i \colon X \to \mathbb{R}^{l_i + 1}$, $I = \{ 1, \ldots, r \}$, and 
$h \colon X \to \mathbb{R}^s$ are twice continuously differentiable, and $Q_{l_i + 1}$ is the second order (Lorentz)
cone of dimension $l_i + 1$. 

Let $x_*$ be a locally optimal solution of problem \eqref{SecondOrderConeProg_EAL}. Recall that the point $x_*$ is
called \textit{nondegenerate} (see \cite{BonnansShapiro}, Def.~4.70), if
$$
  \begin{bmatrix}
    J g_1(x_*) \\
    \vdots \\
    J g_r(x_*) \\
    J h(x_*)
  \end{bmatrix} \mathbb{R}^d +
  \begin{bmatrix}
    \lineal T_{Q_{l_1 + 1}} \big( g_1(x_*) \big) \\
    \vdots \\
    \lineal T_{Q_{l_r + 1}} \big( g_r(x_*) \big) \\
    \{ 0 \}
  \end{bmatrix} =
  \begin{bmatrix}
    \mathbb{R}^{l_1 + 1} \\
    \vdots \\
    \mathbb{R}^{l_r + 1} \\
    \mathbb{R}^s
  \end{bmatrix}.
$$
where $J g_i(x)$ is the Jacobian of $g_i(x)$, and ``lin'' stands for the lineality subspace of a convex cone, i.e. the
largest linear space contained in this cone. Let us note that the nondegeneracy condition can be expressed as a
``linear independence-type'' condition (see~\cite{FukudaSilva}, Lemma~3.1, and \cite{BonnansRamirez}, Proposition~19).
Furthermore, by \cite{BonnansShapiro}, Proposition~4.75, the nondegeneracy condition guarantees that there exists a
\textit{unique} Lagrange multiplier at $x_*$.

Being inspired by the ideas of \cite{FukudaSilva}, for any $x \in X$ and 
$\lambda = (\lambda_1, \ldots, \lambda_r, \mu) \in Y^* := 
\mathbb{R}^{l_1 + 1} \times \ldots \times \mathbb{R}^{l_r + 1} \times \mathbb{R}^s$ define
$$
  \eta(x, \lambda) = \| D_x L(x, \lambda) \|^2 + 
  \sum_{i = 1}^r \Big( \langle \lambda_i, g_i(x) \rangle^2 + 
  \| (\lambda_i)_0 \overline{g}_i(x) + (g_i)_0(x) \overline{\lambda}_i \|^2 \Big),
$$
where $\lambda_i = ((\lambda_i)_0, \overline{\lambda}_i) \in \mathbb{R} \times \mathbb{R}^{l_i}$, and the same notation
is used for $g_i(x)$. Suppose that $x_*$ is a nondegenerate locally optimal solution of problem
\eqref{SecondOrderConeProg_EAL}. Then arguing in the same way as in the proof of Proposition~3.3 in \cite{FukudaSilva}
one can verify that $\eta(x_*, \lambda_*) = 0$ for some $\lambda_*$ iff $(x_*, \lambda_*)$ is a KKT-pair, and such
$\lambda_*$ is unique. Furthermore, the matrix $D^2_{\lambda \lambda} \eta(x_*, \lambda_*)$ is positive definite.
Therefore one can utilize Theorem~\ref{Th_LocalExactnessOfAL} in order to prove the local exactness of the penalized
augmented Lagrangian function $\mathscr{L}_e(x, \lambda, c)$ for problem \eqref{SecondOrderConeProg_EAL}. 

As in the case of the mathematical programming problem, one must add barrier terms into the definition of 
$\mathscr{L}_e(x, \lambda, c)$ in order to ensure that it is level-bounded and globally exact. Choose $\alpha > 0$ and
$\varkappa > 2$, and for any $\lambda = (\lambda_1, \ldots, \lambda_r, \mu) \in Y^*$ define
\begin{equation} \label{BarrierTerms_SOC}
  p(x, \lambda) = \frac{a(x)}{1 + \sum_{i = 1}^r \| \lambda_i \|^2}, \quad
  q(x, \lambda) = \frac{b(x)}{1 + \| \mu \|^2},
\end{equation}
where
$$
  a(x) = \alpha - \sum_{i = 1}^r \dist^{\varkappa}\big( g_i(x), Q_{l_i + 1} \big), \quad
  b(x) = \alpha - \| h(x) \|^2.
$$
Denote $\Omega_{\alpha} = \{ x \in \mathbb{R}^d \mid a(x) > 0, \: b(x) > 0 \}$. Then one can introduce a penalized
augmented Lagrangian function for problem \eqref{SecondOrderConeProg_EAL} as follows (see
Example~\ref{Example_RockWetsAL_SOC}, and Remark~\ref{Rmrk_2OrderExpansInXandLambda}). For any
$\lambda = (\lambda_1, \ldots, \lambda_r, \mu) \in Y^*$ define
\begin{multline} \label{EAL_RockWetsAL_SOC}
  \mathscr{L}_e(x, \lambda, c) = f(x) \\
  + \frac{c}{2 p(x, \lambda)} \sum_{i = 1}^r 
  \Big[ \dist^2\Big( g_i(x) + \frac{p(x, \lambda)}{c} \lambda_i, Q_{l_i + 1} \Big) - 
  \frac{p(x, \lambda)^2}{c^2} \| \lambda_i \|^2 \Big] \\
  + \langle \mu, h(x) \rangle + \frac{c}{2 q(x, \lambda)} \| h(x) \|^2 + \eta(x, \lambda).
\end{multline}
if $x \in \Omega_{\alpha}$, and $\mathscr{L}_e(x, \lambda, c) = + \infty$, otherwise.

Observe that the function $\mathscr{L}_e(x, \lambda, c)$ is l.s.c. jointly in $(x, \lambda)$ on 
$\mathbb{R}^d \times Y^*$, and continuously differentiable in $(x, \lambda)$ on its effective domain $\Omega_{\alpha}
\times Y^*$ by \cite{BonnansShapiro}, Theorem~4.13. One can also check that the function $\Phi(y, \lambda, c)$
corresponding to penalized augmented Lagrangian function \eqref{EAL_RockWetsAL_SOC} satisfies assumptions $(A2)$,
$(A4)_s$, $(A6)_s$, $(A11)$ with $\Phi_0(\lambda) \equiv \lambda$, and $(A12)$. Furthermore, 
$\Phi(G(x_*), \lambda_*, c) = 0$, if $(x_*, \lambda_*)$ is a KKT-pair, and arguing in the same way as in
(\cite{ShapiroSun}, pp.~487--488) one can check that $\Phi(G(x), \lambda, c)$ admits the second order expansion in 
$(x, \lambda)$ at every KKT-pair of problem~\eqref{SecondOrderConeProg_EAL} 
(see also \cite{Dolgopolik_UnifExactnessII}). Therefore one can obtain the following result.

\begin{theorem} \label{Thrm_EAL_SeconOrderCone}
Let the functions $f$, $g_i$, $i \in I$, and $h$ be twice continuously differentiable. Suppose also that every globally
optimal solution of problem \eqref{SecondOrderConeProg_EAL} is nondegenerate, and for any $x_* \in \Omega_*$ a unique
KKT-pair $(x_*, \lambda_*)$ satisfies the second order sufficient optimality condition. Then penalized augmented
Lagrangian function \eqref{EAL_RockWetsAL_SOC} is globally exact if and only if there exists $c_0 > 0$ such that the set
$S_e(c_0) := \{ (x, \lambda) \in \mathbb{R}^d \times Y^* \mid| \mathscr{L}_e(x, \lambda, c_0) < f_* \}$
is either bounded or empty. In particular, if the set
$\Omega(\gamma, \alpha) = \{ x \in \mathbb{R}^d \mid f(x) < f_* + \gamma, \: a(x) > 0, \: b(x) > 0 \}$
is bounded for some $\gamma > 0$, then penalized augmented Lagrangian function \eqref{EAL_RockWetsAL_SOC} is globally
exact.
\end{theorem}

\begin{proof}
Under the assumptions of the theorem augmented Lagrangian \eqref{EAL_RockWetsAL_SOC} is locally exact at every KKT-pair
corresponding to a globally optimal solution of problem \eqref{SecondOrderConeProg_EAL} by
Theorem~\ref{Th_LocalExactnessOfAL}. Then applying the localization principle
(Theorem~\ref{Thrm_EAL_LocalizationPrinciple}) one obtains that this augmented Lagrangians is globally exact iff the set
$S_e(c_0)$ is either bounded or empty for some $c_0 > 0$.

Suppose, now, that the set $\Omega(\gamma, \alpha)$ is bounded for some $\gamma > 0$. Let us check that in this case
the set $S_e(c)$ is bounded for sufficiently large $c > 0$. From \eqref{BarrierTerms_SOC} and
\eqref{EAL_RockWetsAL_SOC} it follows that for any $(x, \lambda) \in \Omega_{\alpha} \times Y^*$ and $c > 0$ one has
\begin{equation} \label{EAL_RW_SOC_LowerEstimate}
  \mathscr{L}_e(x, \lambda, c) \ge f(x) - \frac{\alpha}{c} + \eta(x, \lambda).
\end{equation}
Hence taking into account the fact that the function $\eta(x, \lambda)$ is nonnegative one obtains that for 
any $c > \alpha / \gamma$ and $(x, \lambda) \in S_e(c)$ one has $x \in \Omega(\gamma, \alpha)$. Consequently, there
exists $\tau > - \infty$ such that $\mathscr{L}_e(x, \lambda, c) \ge \tau$ for any $(x, \lambda) \in S_e(c)$ and 
$c > \alpha / \gamma$ due to the boundedness of the set $\Omega(\gamma, \alpha)$.

Arguing by reductio ad absurdum, suppose that the set $S_e(c)$ is unbounded for any $c > 0$. Choose an increasing
unbounded sequence $\{ c_n \} \subset \mathbb{R}_+$ such that $c_1 > \alpha / \gamma$. By our assumption for any
$n \in \mathbb{N}$ there exists $(x_n, \lambda_n) \in S_e(c_n)$ such that $\| \lambda_n \| \ge n$. Note that by the
choice of $c_1$ one has $x_n \in \Omega(\gamma, \alpha)$ for all $n \in \mathbb{N}$, which implies that without loss of
generality one can suppose that the sequence $\{ x_n \}$ converges to a point $x_*$. 

Let us check that $x_*$ is a globally optimal solution of problem \eqref{SecondOrderConeProg_EAL}. Indeed, denote
\begin{gather*}
  u_n^i = \frac{c_n}{2} 
  \Big[ \dist^2\Big( g_i(x_n) + \frac{p(x_n, \lambda_n)}{c_n} (\lambda_n)_i, Q_{l_i + 1} \Big) - 
  \frac{p(x_n, \lambda_n)^2}{c_n^2} \| (\lambda_n)_i \|^2 \Big], \\
  w_n = q(x_n, \lambda_n) \langle \mu_n, h(x_n) \rangle + \frac{c_n}{2} \| h(x_n) \|^2,
\end{gather*}
where $\lambda_n = ( (\lambda_n)_1, \ldots, (\lambda_n)_r, \mu_n )$, and $i \in I$. Note that 
$u_n^i / p(x_n, \lambda_n) \ge - \alpha / 2 c_n$ and $w_n / q(x_n, \lambda_n) \ge - \alpha / 2 c_n$ for all $i \in I$
and $n \in \mathbb{N}$. In other words, the sequences $\{ u_n^i / p(x_n, \lambda_n) \}$, $i \in I$, and 
$\{ w_n / q(x_n, \lambda_n) \}$ are bounded below.

Observe that from the fact that $(x_n, \lambda_n) \in S_e(c_n)$ and $c_n \ge c_1 > \alpha / \gamma$ due to our choice
it follows that $\tau \le \mathscr{L}_e(x_n, \lambda_n, c_n) < f_*$ for all $n \in \mathbb{N}$. Hence applying
\eqref{EAL_RW_SOC_LowerEstimate} and the fact that $\eta(x, \lambda)$ is nonnegative one obtains that
the sequences $\{ f(x_n) \}$ and $\{ \eta(x_n, \lambda_n) \}$ are bounded. Therefore the sequence
$\{ \sum_{i = 1}^r u_n^i / p(x_n, \lambda_n) + w_n / q(x_n, \lambda_n) \}$ is bounded as well, which implies that the
sequences $\{ u_n^i / p(x_n, \lambda_n) \}$, $i \in I$, and $\{ w_n / q(x_n, \lambda_n) \}$ are bounded due to the fact
that they are bounded from below. Since by definition one has $0 < p(x_n, \lambda_n) \le \alpha$ and 
$0 < q(x_n, \lambda_n) \le \alpha$ for all $n \in \mathbb{N}$, the sequences $\{ u_n^i \}$, $i \in I$, and 
$\{ w_n \}$ are bounded as well. Consequently, applying the fact that $c_n \to + \infty$ as $n \to \infty$ one can
easily check that $\dist( g_i(x_n), Q_{l_i + 1}) \to 0$ and $h(x_n) \to 0$ as $n \to \infty$, which yields that $x_*$ is
a feasible point of problem \eqref{SecondOrderConeProg_EAL}. 

From \eqref{EAL_RW_SOC_LowerEstimate} and the fact that $(x_n, \lambda_n) \in S_e(c_n)$ it follows that 
$f(x_n) - \alpha / c_n < f_*$. Passing to the limit as $n \to \infty$ and taking into account the fact that $x_*$ is
feasible one obtains that $x_*$ is a globally optimal solution of problem \eqref{SecondOrderConeProg_EAL}.

Note that the function $\eta(x_*, \cdot)$ is quadratic. Furthermore, arguing in the same way as in the proof of
Proposition~3.3 in \cite{FukudaSilva} one can check that the Hessian of $\eta(x_*, \cdot)$ at $\lambda_*$ is positive
definite, which implies that there exist a neighbourhood $U$ of $x_*$, $\delta_1 > 0$ and 
$\delta_2, \delta_3 \in \mathbb{R}$ such that 
$\eta(x, \lambda) \ge \delta_1 \| \lambda \|^2 + \delta_2 \| \lambda \| + \delta_3$ for all $x \in U$ and 
$\lambda \in Y^*$. Therefore $\eta(x_n, \lambda_n) \to + \infty$ as $n \to \infty$, since $\{ x_n \}$ converges to $x_*$
and $\| \lambda \| \ge n$. Consequently, with the use of \eqref{EAL_RW_SOC_LowerEstimate} one obtains that
$\mathscr{L}_e(x_n, \lambda_n, c_n) \to + \infty$ as $n \to \infty$, which contradicts the fact that 
$(x_n, \lambda_n) \in S_e(c_n)$. Thus, the set $S_e(c)$ is bounded for some $c > 0$.	 
\end{proof}

\begin{remark}
To the best of author's knowledge, penalized augmented Lagrangian functions for nonlinear second-order cone programming
problems have never been studied before. Thus, augmented Lagrangian function \eqref{EAL_RockWetsAL_SOC} is the first
globally exact augmented Lagrangian function for problem \eqref{SecondOrderConeProg_EAL}.
\end{remark}

\subsection{Nonlinear Semidefinite Programming}

Let the problem $(\mathcal{P})$ be the nonlinear semidefinite programming problem of the form
\begin{equation} \label{SemiDefProg_EAL}
  \min f(x) \quad \text{subject to} 
  \quad G_0(x) \preceq 0, \quad h(x) = 0,
\end{equation}
where the functions $f \colon X \to \mathbb{R}$, $G_0 \colon X \to \mathbb{S}^l$ and $h \colon X \to \mathbb{R}^s$ are
twice continuously differentiable.

Let $x_*$ be a locally optimal solution of problem \eqref{SemiDefProg_EAL}. Recall that the point $x_*$ is called
\textit{nondegenerate} (see \cite{BonnansShapiro}, Def.~4.70), if
$$
  \begin{bmatrix}
    D G_0(x_*) \\
    J h(x_*)
  \end{bmatrix} \mathbb{R}^d +
  \begin{bmatrix}
    \lineal T_{\mathbb{S}^l_-} \big( G_0(x_*) \big) \\
    \{ 0 \}
  \end{bmatrix} =
  \begin{bmatrix}
    \mathbb{S}^l \\
    \mathbb{R}^s
  \end{bmatrix}.
$$
As in the case of second order cone programming problems, the above nondegeneracy condition can be rewritten as a
``linear independence-type'' condition. Namely, let $\rank G_0(x_*) = r$. Then the point $x_*$ is nondegenerate iff the
$d$-dimensional vectors
\begin{equation} \label{Nondegeneracy_SemiDef}
  v_{ij} = 
  \left( e_i^T \frac{\partial G_0(x_*)}{\partial x_1} e_j, \ldots, 
  e_i^T \frac{\partial G_0(x_*)}{\partial x_d} e_j \right)^T,
  \quad	\nabla h_k(x_*)
\end{equation}
are linearly independent, where $1 \le i \le j \le l - r$, $e_1, \ldots e_{l - r}$ is a basis of the null space of the
matrix $G_0(x_*)$, $1 \le k \le s$, and $h(x) = (h_1(x), \ldots, h_s(x))$ (see \cite{BonnansShapiro}, Proposition~5.71).

For any $\lambda = (\lambda_0, \mu) \in Y^* = \mathbb{S}^l \times \mathbb{R}^s$ define
$$
  \eta(x, \lambda) = \big\| D_x L(x, \lambda) \big\|^2 + \trace(\lambda_0^2 G_0(x)^2).
$$
Let us demonstrate that the nondegeneracy condition ensures that the function $\eta(x, \lambda)$ has desired
properties. 

\begin{lemma} \label{Lemma_SemiDef_PenTerm_Nondegen}
Let a locally optimal solution $x_*$ of problem \eqref{SemiDefProg_EAL} be nondegenerate. Then there exists a unique
Lagrange multiplier $\lambda_*$ at $x_*$, and $\eta(x_*, \lambda) = 0$ if and only if $\lambda = \lambda_*$.
Furthermore, the matrix $D^2_{\lambda \lambda} \eta(x_*, \lambda_*)$ is positive definite.
\end{lemma}

\begin{proof}
The existence of a unique Lagrange multiplier $\lambda_*$ at $x_*$ follows directly from \cite{BonnansShapiro},
Proposition~4.75. Furthermore, note that the function $\eta(x_*, \cdot)$ is quadratic. Therefore it remains to check
that the matrix $D^2_{\lambda \lambda} \eta(x_*, \lambda_*)$ is positive definite.

Suppose that $\lambda = (\lambda_0, \mu) \in Y^*$ is such that 
$D^2_{\lambda \lambda} \eta(x_*, \lambda_*)(\lambda, \lambda) = 0$. Then $\trace(\lambda_0^2 G_0(x_*)^2) = 0$. Let
$G_0(x_*) = E \diag(\rho_1(x_*), \ldots, \rho_l(x_*)) E^T$ be a spectral decomposition of $G_0(x_*)$ such that the
eigenvalues $\rho_i(x_*)$ are listed in the decreasing order. Then 
$G_0(x_*)^2 = E \diag(\rho_1(x_*)^2, \ldots, \rho_l(x_*)^2) E^T$. Applying the fact that the trace operator is
invariant under cyclic permutations one obtains that
\begin{equation*}
\begin{split}
  \trace\Big( \lambda_0^2 G_0(x_*)^2 \Big) 
  &= \trace\Big( E^T \lambda_0^2 E \diag(\rho_1(x_*)^2, \ldots, \rho_l(x_*)^2) \Big) \\
  &= \sum_{i = 1}^l \rho_i(x_*)^2 e_i^T \lambda_0^2 e_i = 0,
\end{split}
\end{equation*}
where $e_i$ is an eigenvector of $G_0(x_*)$ corresponding to the eigenvalue $\rho_i(x_*)$. Let $\rank G_0(x_*) = r$.
Then the above equalities imply that $\lambda_0 e_i = 0$ for any $i \in \{ l - r + 1, \ldots, l \}$. Therefore there
exists a $(l - r) \times (l - r)$ symmetric matrix $\Gamma$ such that 
$$
  E^T \lambda_0 E = \begin{pmatrix}
    \Gamma & 0 \\
    0 & 0
  \end{pmatrix}, \quad
  \lambda_0 = E \begin{pmatrix}
    \Gamma & 0 \\
    0 & 0
  \end{pmatrix} E^T
  = E_0 \Gamma E_0^T,
$$
where $E_0$ is a $l \times (l - r)$ matrix whose columns are a basis of the null space of the matrix $G_0(x_*)$ 
(note that in the case $r = l$ one has $\lambda_0 = 0$).

Since $D^2_{\lambda \lambda} \eta(x_*, \lambda_*)(\lambda, \lambda) = 0$, for the function
$\omega(\cdot) = \| D_x L(x_*, \cdot) \|^2$ one has $D^2 \omega(\lambda_*) (\lambda, \lambda) = 0$  or,
equivalently,
\begin{multline} \label{PenaltyTerm_SemiDef_Hessian}
  \sum_{i = 1}^d \trace\left( \lambda_0 \frac{\partial G_0(x_*)}{\partial x_i} \right) \cdot
  \trace\left( \lambda_0 \frac{\partial G_0(x_*)}{\partial x_i} \right) + 2
  \sum_{j = 1}^s \sum_{i = 1}^d \trace\left( \lambda_0 \frac{\partial G_0(x_*)}{\partial x_i} \right) \cdot
  \mu_j \frac{\partial h_j(x_*)}{\partial x_i} \\
  + \sum_{j, k = 1}^s \sum_{i = 1}^d \mu_j \mu_k \frac{\partial h_j(x_*)}{\partial x_i} 
  \frac{\partial h_k(x_*)}{\partial x_i} = 0.
\end{multline}
Applying the equality $\lambda_0 = E_0 \Gamma E_0^T$ one obtains that
$$
  \trace\left( \lambda_0 \frac{\partial G_0(x_*)}{\partial x_i} \right) = 
  \trace\left( \Gamma E_0^T \frac{\partial G_0(x_*)}{\partial x_i} E_0 \right) = 
  \sum_{j, k = 1}^{l - r} \Gamma_{jk} e_j^T \frac{\partial G_0(x_*)}{\partial x_i} e_k.
$$
Hence and from \eqref{PenaltyTerm_SemiDef_Hessian} one gets that
\begin{multline*}
  \sum_{i, j = 1}^{l - r} \sum_{p, q = 1}^{l - r} \Gamma_{ij} \Gamma_{pq} \langle v_{ij}, v_{pq} \rangle +
  2 \sum_{i, j = 1}^{l - r} \sum_{k = 1}^s \Gamma_{ij} \mu_k \langle v_{ij}, \nabla h_k(x_*) \rangle \\
  + \sum_{i, j = 1}^s \mu_i \mu_j \langle \nabla h_i(x_*), \nabla h_j(x_*) \rangle = 0,
\end{multline*}
where $v_{jk}$ are defined as in \eqref{Nondegeneracy_SemiDef}. Denote
$z = \sum_{i, j = 1}^{l - r} \Gamma_{ij} v_{ij} + \sum_{k = 1}^s \mu_k \nabla h_k(x_*)$. The equality above implies that
$\| z \|^2 = 0$, i.e. $z = 0$ or, equivalently,
$$
  \sum_{i = 1}^{l - r} \Gamma_{ii} v_{ii} + \sum_{1 \le i < j \le l - r} 2 \Gamma_{ij} v_{ij} + 
  \sum_{k = 1}^s \mu_k \nabla h_k(x_*) = 0.
$$
Here we used the facts that the matrix $\Gamma$ is symmetric and $v_{ij} = v_{ji}$. With the use of the nondegeneracy
condition one obtains that $\Gamma = 0$ and $\mu = 0$, i.e. $\lambda = 0$, which implies that
the matrix $D^2_{\lambda \lambda} \eta(x_*, \lambda_*)$ is positive definite.	 
\end{proof}

Now, we can introduce the penalized augmented Lagrangian function for problem \eqref{SemiDefProg_EAL}. Choose 
$\alpha > 0$ and $\varkappa > 1$, and for any $\lambda = (\lambda_0, \mu) \in Y^*$ define
$$
  p(x, \lambda) = \frac{a(x)}{1 + \trace(\lambda_0^2)}, \quad
  q(x, \lambda) = \frac{b(x)}{1 + \| \mu \|^2},
$$
where
$$
  a(x) = \alpha - \trace\big( [ G_0(x) ]_+^2 \big)^{\varkappa}, \quad
  b(x) = \alpha - \| h(x) \|^2.
$$
Denote $\Omega_{\alpha} = \{ x \in \mathbb{R}^d \mid a(x) > 0, \: b(x) > 0 \}$. Finally, for any 
$\lambda = (\lambda_0, \mu) \in Y^*$ define
\begin{multline} \label{EAL_RockWetsAL_SemiDef}
  \mathscr{L}_e(x, \lambda, c) = f(x) 
  + \frac{1}{2c p(x, \lambda)} \Big( \trace\big( [c G_0(x) + p(x, \lambda) \lambda_0]_+^2 \big) - 
  p(x, \lambda)^2 \trace(\lambda_0^2) \Big) \\
  + \langle \mu, h(x) \rangle + \frac{c}{2 q(x, \lambda)} \| h(x) \|^2 + \eta(x, \lambda),
\end{multline}
if $x \in \Omega_{\alpha}$, and $\mathscr{L}_e(x, \lambda, c) = + \infty$, otherwise 
(see Example~\ref{Example_RockWetsAL_SemiDefProg}). Note that the function $\mathscr{L}_e(x, \lambda, c)$ is l.s.c.
jointly in $(x, \lambda)$ on $\mathbb{R}^d \times Y^*$, and continuously differentiable in $(x, \lambda)$ on its
effective domain $\Omega_{\alpha} \times Y^*$ by \cite{BonnansShapiro}, Theorem~4.13. One can also check that
the function $\Phi(y, \lambda, c)$ corresponding to penalized augmented Lagrangian function
\eqref{EAL_RockWetsAL_SemiDef} satisfies assumptions $(A2)$, $(A4)_s$, $(A6)_s$, $(A11)$ with 
$\Phi_0(\lambda) \equiv \lambda$, and $(A12)$. Furthermore, $\Phi(G(x_*), \lambda_*, c) = 0$, if $(x_*, \lambda_*)$ is a
KKT-pair, and arguing in the same way as in (\cite{ShapiroSun}, pp.~487--488) one can check that $\Phi(G(x), \lambda,
c)$ admits the second order expansion in $(x, \lambda)$ at every KKT-pair of problem~\eqref{SemiDefProg_EAL}.

Applying Lemma~\ref{Lemma_SemiDef_PenTerm_Nondegen}, Theorem~\ref{Th_LocalExactnessOfAL} and the localization principle,
and arguing in the same way as in the proofs of Theorems~\ref{Thrm_EAL_HPR} and \ref{Thrm_EAL_SeconOrderCone} one can
easily verify that the following result holds true (see~\cite{Dolgopolik_UnifExactnessII} for the detailed proof of
this result).

\begin{theorem}
Let the functions $f$, $G_0$, and $h$ be twice continuously differentiable. Suppose also that every globally optimal
solution of problem \eqref{SemiDefProg_EAL} is nondegenerate, and for any $x_* \in \Omega_*$ a unique KKT-pair 
$(x_*, \lambda_*)$ satisfies the second order sufficient optimality condition. Then penalized augmented
Lagrangian function \eqref{EAL_RockWetsAL_SemiDef} is globally exact if and only if there exists $c_0 > 0$ such that 
the set $\{ (x, \lambda) \in \mathbb{R}^d \times Y^* \mid \mathscr{L}_e(x, \lambda, c_0) < f_* \}$ is either bounded or
empty. In particular, if the set $\{ x \in \mathbb{R}^d \mid f(x) < f_* + \gamma, \: a(x) > 0, \: b(x) > 0 \}$ is
bounded for some $\gamma > 0$, then penalized augmented Lagrangian function \eqref{EAL_RockWetsAL_SemiDef} is globally
exact in the sense that for any sufficiently large $c > 0$ its points of global minimum in $(x, \lambda)$ are exactly
KKT-pairs of problem \eqref{SemiDefProg_EAL} corresponding to globally optimal solutions of this problem.
\end{theorem}

\begin{remark}
{(i)~One can easily construct a globally exact penalized augmented Lagrangian function for problem
\eqref{SemiDefProg_EAL} from augmented Lagrangian from Example~\ref{Example_NonlinearRescalePenalized_SemiDefProg}.
Namely, for any $\lambda = (\lambda_0, \mu) \in Y^*$ set
$$
  \eta(x, \lambda) = \big\| D_x L(x, \lambda_0^2, \mu) \big\|^2 + \trace(\lambda_0^2 G_0(x)^2),
$$
and define
\begin{multline} \label{EAL_PenalizedExpType_SemiDef}
  \mathscr{L}_e(x, \lambda, c) = f(x) + 
  \langle \mu, h(x) \rangle + \frac{c}{2 q(x, \lambda)} \| h(x) \|^2 + \eta(x, \lambda) \\
  + \frac{p(x, \lambda)}{c} \left\langle \lambda_0^2, \Psi\left( \frac{c}{p(x, \lambda)} G_0(x) \right) \right\rangle
  + \frac{p(x, \lambda)}{c} \trace\left[ \Xi\left( \frac{c}{p(x, \lambda)} G_0(x) \right) \right],
\end{multline}
if $x \notin \Omega_{\alpha}$, and $\mathscr{L}_e(x, \lambda, c) =  +\infty$ otherwise
(see~Remark~\ref{Rmrk_2OrderExpansInXandLambda}), where $\Psi$ and $\Xi$ are the same as in 
Example~\ref{Example_NonlinearRescalePenalized_SemiDefProg}. Then penalized augmented Lagrangian 
\eqref{EAL_PenalizedExpType_SemiDef} is globally exact provided every globally optimal solution of problem
\eqref{SemiDefProg_EAL} is nondegenerate, satisfies the second order sufficient optimality and strict complementarity
conditions, the set $\{ x \in \mathbb{R}^d \mid f(x) < f_* + \gamma, \: a(x) > 0, \: b(x) > 0 \}$ is bounded, and the
function $\psi(t)$ is bounded from below. We do not present the proof of this result here, and leave it to the
interested reader.
}

\noindent{(ii)~Let us note that penalized augmented Lagrangian functions \eqref{EAL_RockWetsAL_SemiDef} and
\eqref{EAL_PenalizedExpType_SemiDef} for problem \eqref{SemiDefProg_EAL} are completely new.
A different exact augmented Lagrangian function for nonlinear semidefinite programming problems was recently introduced
in \cite{FukudaLourenco}. It should be pointed out that our augmented Lagrangian function is defined via the
problem data directly, while the augmented Lagrangian function from \cite{FukudaLourenco} depends on a solution of a
certain system of linear equations, which makes the computation of the value of this augmented Lagrangian function and
its derivatives more expensive. Furthermore, in order to correctly define the augmented Lagrangian function from
\cite{FukudaLourenco} one must suppose that \textit{every} feasible point of problem \eqref{SemiDefProg_EAL} is
nondegenerate, which might be a too restrictive assumption for many applications. In contrast, we assume that only
globally optimal solutions of problem \eqref{SemiDefProg_EAL} are nondegerate.
}
\end{remark}

\subsection{A Different Approach to Global Exactness}

As the examples above show, one must suppose that \textit{both} sufficient optimality conditions and a constraint
qualification hold true at globally optimal solutions of the problem $(\mathcal{P})$ in order to guarantee the global
exactness of penalized augmented Lagrangian functions with the use of the localization principle and
Theorem~\ref{Th_LocalExactnessOfAL}. 

Being inspired by the ideas of \cite{DiPilloLucidi2001} and the localization principle, we present different necessary
and sufficient conditions for the global exactness of $\mathscr{L}_e(x, \lambda, c)$ that are based on the use of
constraint qualifications \textit{only}. However, let us point out that this conditions are applicable only in the case 
when $\mathscr{L}(x, \lambda, c)$ is the Hestenes-Powell-Rockafellar augmented Lagrangian function.

\begin{theorem} \label{Th_GlobalExactness_AlternativeApproach}
Let $Y$ be finite dimensional, $A = X$, $\Lambda = Y^*$, $G$ be continuous, and $\mathscr{L}_e(\cdot, \cdot, c)$ be
l.s.c. for all $c > 0$. Suppose also that assumptions $(A2)$, $(A4)_s$, $(A6)_s$ and $(A12)$ are satisfied, and
\begin{equation} \label{AugmFuncVanishAtKKTpoints_HPR_AL}
  \Phi(G(x_*), \lambda_*, c) = 0 \quad \forall c > 0 \quad 
  \forall (x_*, \lambda_*) \in \Omega_* \times Y^* \colon \eta(x_*, \lambda_*) = 0.
\end{equation}
Suppose, finally, that for any $(x_*, \lambda_*) \in \Omega_* \times Y^*$ such that $\eta(x_*, \lambda_*) = 0$ there
exist a neighbourhood $U$ of $(x_*, \lambda_*)$ and $\overline{c} > 0$ such that the function $\mathscr{L}_e(\cdot,
\cdot, c)$ is G\^ateaux differentiable in $U$ for all $c \ge \overline{c}$, and
\begin{equation} \label{DiPilloGrippo_Condition}
  \Big( (x, \lambda) \in U \wedge D_{(x, \lambda)} \mathscr{L}_e(x, \lambda, c) = 0 \Big) \implies 
  \Big( x \in \Omega \wedge \Phi(G(x), \lambda, c) = 0 \Big).
\end{equation}
Then $\mathscr{L}_e(x, \lambda, c)$ is globally exact (with respect to the function $\eta$) if and only if there exists
$(x_0, \lambda_0) \in \Omega_* \times Y^*$ such that $\eta(x_0, \lambda_0) = 0$, and there exists $c_0 > 0$ such that
the set $S_e(c_0)$ is either bounded or empty.
\end{theorem}

\begin{proof}
The ``only if'' part of the theorem is proved in the same way as the ``only if'' part of
Theorem~\ref{Thrm_EAL_LocalizationPrinciple}. Therefore, let us prove the ``if'' part.

Choose an increasing unbounded sequence $\{ c_n \} \subset [c_0, + \infty)$. Assumption $(A4)_s$ implies that 
$S_e(c_n) \subseteq S_e(c_0)$ for all $n \in \mathbb{N}$. If $S_e(c_n) = \emptyset$ for some $n \in \mathbb{N}$, then
with the use of \eqref{AugmFuncVanishAtKKTpoints_HPR_AL} one can easily verify that $\mathscr{L}_e(x, \lambda, c)$ is
globally exact. Therefore, let us suppose that $S_e(c_n) \ne \emptyset$ for all $n \in \mathbb{N}$.

Taking into account the facts that $S_e(c_0)$ is either bounded or empty, $S_e(c_n) \subseteq S_e(c_0)$, and
$\mathscr{L}_e(\cdot, \cdot, c)$ is l.s.c. one obtains that for any $n \in \mathbb{N}$ 
the function $\mathscr{L}_e(\cdot, \cdot, c_n)$ attains a global minimum at a point $(x_n, \lambda_n)$, and 
the sequence $\{ (x_n, \lambda_n) \}$ is bounded. Since both $X$ and $Y$ are finite dimensional, without loss of
generality one can suppose that $\{ (x_n, \lambda_n) \}$ converges to a point $(x_*, \lambda_*)$. By
Lemma~\ref{Lemma_EAL_MinimizingSeq} one has $x_* \in \Omega_*$ and $\eta(x_*, \lambda_*) = 0$. Therefore there exist a
neighbourhood $U$ of $(x_*, \lambda_*)$ and $\overline{c} > 0$ such that \eqref{DiPilloGrippo_Condition} holds true.

From the facts that $(x_*, \lambda_*)$ is a limit point of $\{ (x_n, \lambda_n) \}$, and $\{ c_n \}$ is an increasing
unbounded sequence it follows that there exists $n \in \mathbb{N}$ such that $(x_n, \lambda_n) \in U$ and 
$c_n \ge \overline{c}$. Applying the first order necessary optimality condition one obtains that 
$D_{(x, \lambda)} \mathscr{L}_e(x_n, \lambda_n, c_n) = 0$, which with the use of
\eqref{DiPilloGrippo_Condition} implies that $x_n$ is a feasible point of $(\mathcal{P})$, and 
$\mathscr{L}_e(x_n, \lambda_n, c_n) = f(x_n) + \eta(x_n, \lambda_n) \ge f(x_n) \ge f_*$. Thus, $S_e(c_n) = \emptyset$,
which contradicts our assumption.
\end{proof}

Note that condition \eqref{DiPilloGrippo_Condition} is satisfied for augmented Lagrangian \eqref{EAL_HPR_AugmLagr} (in
the case when there are no equality constraints) by \cite{DiPilloLucidi2001}, Proposition~4.3, provided LICQ holds at
every globally optimal solution. Therefore with the use of the theorem above one can obtain the following result that
improves Theorem~5.4 from \cite{DiPilloLucidi2001}, since we do note assume that LICQ holds true at every feasible point
of problem \eqref{MathProg_EAL}, and obtain \textit{necessary and sufficient} conditions for the global exactness of
augmented Lagrangian \eqref{EAL_HPR_AugmLagr} for problem \eqref{MathProg_EAL}.

\begin{theorem} 
Suppose that $J = \emptyset$. Let $f$ and $g_i$, $i \in I$, be twice continuously differentiable, and
let LICQ hold true at every globally optimal solution of problem \eqref{MathProg_EAL}. Then penalized augmented
Lagrangian function \eqref{EAL_HPR_AugmLagr} is globally exact if and only if there exists $c_0 > 0$ such that the set
$S_e(c_0)$ is either bounded or empty. In particular, if there exists $\gamma > 0$ such that 
the set $\Omega(\gamma, \alpha)$ is bounded, then augmented Lagrangian \eqref{EAL_HPR_AugmLagr} is globally exact.
\end{theorem}

\begin{remark}
{(i)~It is natural to assume that Proposition~4.3 from \cite{DiPilloLucidi2001} (and, thus, the theorem above) can be
extended to the case of other penalized augmented Lagrangian functions (in particular, to the case of augmented
Lagrangians \eqref{EAL_RockWetsAL_SOC} and \eqref{EAL_RockWetsAL_SemiDef}). However, we do not discuss the possibility
of such an extension here, and pose it as an interesting open problem.
}

\noindent{(ii)~Let us note that with the use of Theorem~\ref{Th_GlobalExactness_AlternativeApproach} one can easily
obtain \textit{necessary and sufficient} conditions for the global exactness of the augmented Lagrangian function for
nonlinear semidefinite programming problems from \cite{FukudaLourenco}, which strengthen Theorem~4.8 from this paper.
}
\end{remark}

\section{Conclusions}

In this article we developed a general theory of augmented Lagrangian functions for cone constrained optimization
problems. Let us briefly discuss some conclusions that can be drawn from this theory. 

Note that there are two main classes of augmented Lagrangian functions for cone constrained optimization problem. 
The first class consists of the well-known Hestenes-Powell-Rockafellar augmented Lagrangian for mathematical
programming problems and its various modifications and extensions to the case of other cone constrained optimization
problem (Examples~\ref{Example_RockafellarWetsAL}--\ref{Example_Mangasarian}, \ref{Example_RockWetsAL_SOC},
\ref{Example_RockWetsAL_SemiDefProg} and \ref{Example_RockafellarWetsAL_SemiInfProg}). The second class of augmented
Lagrangians consists of the exponential penalty functions and its numerous modifications
(Examples~\ref{Example_ExpPenFunc}--\ref{Example_pthPowerAugmLagr}, \ref{Example_NonlinearRescale_SOC},
\ref{Example_NonlinearRescale_SemiDefProg}, \ref{Example_NonlinearRescalePenalized_SemiDefProg} and
\ref{Example_NonlinearRescale_SemiInfProg}). It is natural to refer to the augmented Lagrangians from this class as
\textit{nonlinear rescaling augmented Lagrangians}, since, in essence, all these augmented Lagrangians are constructed
as the standard Lagrangian function for the problem with rescaled constraints (see~\cite{Polyak2002}). Finally,
there is also He-Wu-Meng's augmented Lagrangian (Example~\ref{Example_HeWuMengLagrangian}) that does not fall into
either of those classes.

Nonlinear rescaling augmented Lagrangians are smoother than Hestenes-Power-Rockafellar-type augmented Lagrangians.
However, in order to guarantee the existence of global saddle points (or local/global exactness) of nonlinear rescaling
augmented Lagrangians one must impose the strict complementarity condition, which is not necessary in the case of the 
Hestenes-Power-Rockafellar-type augmented Lagrangians and He-Wu-Meng's augmented Lagrangian 
(see Remarks \ref{Rmrk_2OrderExpansInX} and \ref{Rmrk_2OrderExpansInXandLambda}) with the only exclusion being the cubic
augmented Lagrangian (Example~\ref{Example_CubilAL}). Therefore, it seems that in order to avoid the strict
complementarity condition one must consider augmented Lagrangians that are not twice continuously differentiable.
Furthermore, as it was noted above (see~Remark~\ref{Remark_AugmLagr_SemiInfProblems}), all existing augmented
Lagrangian functions are not suitable for handling semi-infinite programming problems, and a new approach to the
construction of augmented Lagrangians for these problem is needed.

From the theoretical point of view, the augmented Lagrangian functions that are least suitable for the study of the
existence of global saddle points (and global exactness) are those augmented Lagrangians that do not contain penalty
terms (namely, the exponential penalty function and the p-th power augmented Lagrangian;
see~Example~\ref{Example_SeparationCondition} and Remark~\ref{Remark_SeparationCondition}). However, numerical methods
based on these augmented Lagrangians sometimes work well for convex (even infinite dimensional, see \cite{Pedregal})
problems.

Finally, let us note that we did not discuss augmented Lagrangian methods for cone constrained optimization problems as
well as numerical methods based on the use of globally exact augmented Lagrangian functions for these problems. It
seems that under some additional assumptions one might extend the existing augmented Lagrangian methods to the case of
the augmented Lagrangian $\mathscr{L}(x, \lambda, c)$ for problem $(\mathcal{P})$, thus developing a general theory of
augmented Lagrangian methods for cone constrained optimization problems. We leave the development of this theory as an
open problem for future research.

\bibliographystyle{abbrv}  
\bibliography{AugmLagr_bibl}

\end{document}